\documentclass[notitlepage,11pt,reqno]{amsart}
\usepackage[foot]{amsaddr}
\usepackage{amssymb,nicefrac,bm,upgreek,mathtools,verbatim}
\usepackage[final]{hyperref}
\usepackage[mathscr]{eucal}
\usepackage{dsfont}
\usepackage[normalem]{ulem}
\usepackage{amsopn,esint}
\usepackage[T1]{fontenc}
\usepackage[shortlabels]{enumitem}
\usepackage[T1]{fontenc}
\usepackage[utf8]{inputenc}
\usepackage{apptools}
\usepackage{enumitem}
\AtAppendix{\counterwithin{theorem}{section}} 

\usepackage[margin=1in]{geometry}

\newcommand{\stkout}[1]{\ifmmode\text{\sout{\ensuremath{#1}}}\else\sout{#1}\fi}
\hypersetup{
  colorlinks=true,
  citecolor=black,
  linkcolor=black,
  urlcolor = blue,
  anchorcolor = blue,
  frenchlinks=false,
  pdfborder={0 0 0},
  naturalnames=false,
  hypertexnames=false,
  breaklinks}
\usepackage{cite}
\usepackage[capitalize]{cleveref}

\crefname{equation}{}{}
\newtheorem{theorem}{Theorem}[section]

\newtheorem{lemma}[theorem]{Lemma}
\newtheorem*{theorem*}{Theorem}

\theoremstyle{definition}
\newtheorem{definition}[theorem]{Definition}
\newtheorem{remark}[theorem]{Remark}
\newtheorem*{remark*}{Remark}

\numberwithin{equation}{section} 

\newcommand{\mcf}{\mathcal{F}}

\newcommand{\la}{\lambda}
\newcommand{\La}{\Lambda}
\newcommand{\fw}{{\mathfrak{w}}}
\newcommand{\fz}{{\mathfrak{z}}}

\newcommand{\RR}{\mathds{R}}

\newcommand{\NN}{\mathds{N}}

\DeclareMathOperator*{\osc}{osc}

\DeclareMathOperator*{\dist}{dist}

\DeclareMathOperator*{\supp}{supp}

\newcommand{ \mr }{ \RR }

\newcommand{\iintss}{{\int\hspace{-0.28cm}\int}}
\newcommand{ \miints }{{\iintss\hspace{-0.56cm} -\hspace{-0.15cm}-}}
\newcommand{\miint}[1]{{\miints_{\hspace{-0.13cm}#1}}}

\def\YYint#1#2#3{{\setbox0=\hbox{$#1{#2#3}{\iint}$}
		\vcenter{\hbox{$#2#3$}}\kern-.50\wd0}}

\def\Xint#1{\mathchoice
	{\XXint\displaystyle\textstyle{#1}}%
	{\XXint\textstyle\scriptstyle{#1}}%
	{\XXint\scriptstyle\scriptscriptstyle{#1}}%
	{\XXint\scriptscriptstyle\scriptscriptstyle{#1}}%
	\!\int}
\def\XXint#1#2#3{{\setbox0=\hbox{$#1{#2#3}{\int}$}
		\vcenter{\hbox{$#2#3$}}\kern-.50\wd0}}

\def\Xint#1{\mathchoice
	{\XXint\displaystyle\textstyle{#1}}%
	{\XXint\textstyle\scriptstyle{#1}}%
	{\XXint\scriptstyle\scriptscriptstyle{#1}}%
	{\XXint\scriptscriptstyle\scriptscriptstyle{#1}}%
	\!\int}
\def\XXint#1#2#3{{\setbox0=\hbox{$#1{#2#3}{\int}$}
		\vcenter{\hbox{$#2#3$}}\kern-.50\wd0}}

\def\dashint{\Xint-}

\usepackage{xcolor}

\allowdisplaybreaks

\newcommand{\ttl}{\MakeUppercase{Higher integrability for parabolic double phase equations with an improved gap bound}}
\begin{document}
\title[H\"older continuous solutions to parabolic double phase equations]
{\ttl}
\author{Abhrojyoti Sen$^1$} 
\address{$^1$Goethe-Universit\"{a}t Frankfurt, Institut f\"{u}r Mathematik, Robert-Mayer-Str. 10,
 D-60629 Frankfurt, Germany. Email: sen@math.uni-frankfurt.de}
 
\author{Jarkko Siltakoski$^2$}
\address{$^2$Department of Mathematics and Statistics, University of Jyv\"askyl\"a, PO Box 35, 40014 Jyvaskyla, Finland. Email address: jarkko.siltakoski@jyu.fi}

\begin{abstract}
We prove a local higher integrability result for the gradient of \textit{a priori} H\"older
continuous weak solutions to the parabolic double phase equation
\[
\partial_t u - \operatorname{div} \left(|Du|^{p-2}Du + a(z)|Du|^{q-2}Du\right)
  = 0 \qquad \text{in } \Omega_T.
\]
We work under a relaxed gap condition on the exponents $p$ and $q$. The
coefficient $a$ is assumed to belong to the class $\mathcal{Z}^{\kappa}(\Omega_T)$ for some $\kappa \in (0,\infty)$. The functions in this class satisfy a one-sided
pointwise bound that controls how fast $a$ can grow away from its zero set,
and the class contains the H\"older continuous functions. We also impose a \emph{mild almost increasing} condition on $a$, which motivates the introduction of a new mollification, which we call the \emph{slanted Steklov average}. For $u \in C^{0,\gamma,\gamma/q}_{\mathrm{loc}}(\Omega_T)$ with $\gamma \in [0,1)$,
our main result holds under the gap bound
\begin{equation}\tag{G}\label{eq:G}
  2 \le p \le q, \,\,\,\,\,\,\,\,\, (q-p)(q-2\gamma)\leq q\kappa
\end{equation}
The new gap condition \eqref{eq:G} is purely parabolic in nature and is
stricter than the optimal gap relation associated with the Lavrentiev
phenomenon for the elliptic double phase functional.
\end{abstract}

\maketitle
\tableofcontents
\section{Introduction and main result}\label{intro}
\subsection{Overview}\label{subsec:overview}
We consider the model equation
\begin{equation}\label{eq:main_double}
  \partial_t u - \operatorname{div}\bigl(|Du|^{p-2}Du + a(z)|Du|^{q-2}Du\bigr)
  = 0 \qquad \text{in } \Omega_T,
\end{equation}
and, in greater generality, the equation
\begin{equation}\label{eq:main equation}
  \partial_t u - \operatorname{div} A(z, u, Du) = 0 \qquad \text{in } \Omega_T,
\end{equation}
where $\Omega_T := \Omega \times (0, T]$ with $\Omega \subset \RR^N$ a
bounded open set, $N\ge 2$. The Carath\'eodory vector field
$A : \Omega_T \times \RR \times \RR^N \to \RR^N$ is
assumed to satisfy the structural conditions
\begin{equation}\label{eq:growth_cond}
  \nu\, H(z, |\xi|) \le A(z, v, \xi) \cdot \xi,
  \qquad
  |A(z, v, \xi)|\,|\xi| \le L\, H(z, |\xi|),
\end{equation}
for all $z \in \Omega_T$, $v \in \RR$, $\xi \in \RR^N$, with
constants $0 < \nu \le L$. Here, $H : \Omega_T \times \RR_+ \to
\RR_+$ is the parabolic double phase integrand
\[
  H(z, \zeta) = \zeta^p + a(z)\zeta^q,
  \qquad z \in \Omega_T,\quad \zeta \in \RR_+,
\]
where $1 < p \le q$ and the modulating coefficient $a$ is a non-negative
function on $\Omega_T$ whose precise regularity will be specified below.
The choice $A(z, v, \xi) = |\xi|^{p-2}\xi + a(z)|\xi|^{q-2}\xi$
recovers~\eqref{eq:main_double}, the canonical parabolic double phase equation.

The elliptic ancestor of~\eqref{eq:main_double} was introduced by
Zhikov~\cite{Zhikov1986, Zhikov1993, Zhikov1995, Zhikov1997} as a model for
strongly anisotropic composite materials, in which the ellipticity changes
from $p$-growth on the set $\{a = 0\}$ to $(p, q)$-growth elsewhere. The
same model is also one of the standard examples in which the Lavrentiev
phenomenon occurs~\cite{Esposito2004, Fonseca2004, Zhikov1993, Zhikov1995} and has,
for over a decade, been a focal point of the regularity theory of
non-uniformly elliptic problems. In the elliptic setting,
$C^{0, \alpha}$- and $C^{1, \alpha}$-estimates, Calder\'on--Zygmund
estimates, and Harnack-type inequalities have been established
in~\cite{Baroni2015, Colombo2015, Colombo2015a, Colombo2016, DeFilippis2019,
Baasandorj2020, Byun2020, Byun2017}, among many other contributions. A
parabolic counterpart has only recently begun to develop: gradient higher integrability is treated
in~\cite{2023_Gradient_Higher_Integrability_for_Degenerate_Parabolic_Double-Phase_Systems,
KS24, S25, KO2025, kim2026interpolativerefinementgapbound, oh2026gradienthigherintegrabilitydegenerate, arora2026irregulardoublephaseevolution}, Calder\'on--Zygmund estimates in~\cite{kim2024calderonzygmundtypeestimateparabolic,
kim2025calderonzygmundtypeestimatesingular}, existence and uniqueness of
weak solutions in~\cite{ST16, KimLipschitz25}, and H\"older continuity
in~\cite{Kim2025holder}.

A central question in the double phase setting is how the regularity of solutions is related to
the interplay between the gap condition of $p$ and $q,$ and the regularity of the
modulating coefficient $a$. Until now, the literature has assumed $a \in C^{\alpha,
\alpha/2}(\Omega_T)$ for some $\alpha \in (0, 1]$. We instead work with a
strictly larger class of coefficients, originally identified
in~\cite{BCFM2023} in the elliptic setting and which we follow in that paper
by calling $\mathcal{Z}^{\kappa}$.

\begin{definition}\label{def:Zkappa}
For $\kappa \in (0, \infty)$, a function $a : \Omega_T \to [0, \infty)$
belongs to the class $\mathcal{Z}^{\kappa}(\Omega_T)$ if there exists a
constant $c_a \ge 1$ such that
\begin{equation}\label{eq:Z_kappa}
  a(z_1) \le c_a\bigl(a(z_2) + d_p(z_1,z_2)^{\kappa}\bigr)
  \qquad \text{for all } z_1, z_2 \in \Omega_T.
\end{equation}
where $d_p$ denotes the parabolic metric given by Definition~\ref{parabolic metric}.
\end{definition}

In the elliptic setting, $\mathcal{Z}^{\kappa}$ was introduced
in~\cite{BCFM2023} in connection with the Lavrentiev phenomenon, and
closely related conditions had previously appeared in the regularity theory
of non-autonomous elliptic problems~\cite{HHL_2021, HOK2021, HOK2022}. To the best of
our knowledge, the present paper is the first to use this class in the
parabolic context. As a guiding example, the function $x \mapsto |x|^{\kappa}$
belongs to $\mathcal{Z}^{\kappa}(\RR^N)$ for every $\kappa \in (0,
\infty)$, illustrating that $\mathcal{Z}^{\kappa}$ accommodates
non-H\"older behaviour of the coefficients and $\kappa$
precisely measures the rate at which the coefficient is allowed to vanish on its zero set $\{a=0\}$, a rate that may be taken arbitrarily large.

Throughout this paper, $a$ is assumed to lie in
$\mathcal{Z}^{\kappa}(\Omega_T)$ for some $\kappa \in (0, \infty)$, and our main result is established under the gap bound
\begin{equation}\label{eq:new_upper_bound}
  2 \le p \le q \le p + \frac{q\kappa}{q - 2\gamma},
\end{equation}
where $\gamma \in [0, 1)$ is the H\"older exponent of the solution $u$ and
$\kappa$ is the constant in~\eqref{eq:Z_kappa}. When $\gamma>0,$ the weak solution assumed a priori belong to $C^{0, \gamma, \gamma/q}_{\mathrm{loc}}(\Omega_T).$ When $\gamma =0,$ the solution is assumed to be bounded.

We also assume that the coefficient $a$ satisfies a \textit{mild monotonicity} condition, i.e.,there exist $\zeta \in \RR^N$ and a positive constant $C_a$, both independent of $(x,t)$, such that either 
\begin{equation}\label{eq:a decreasing}
    a(x,t) \leq C_a a(x+(t-s)\zeta,s) \quad \text{for all } (x,t),(x,s)\in \Omega_T \quad \text{with }\,\, s \leq t.
\end{equation}
\text{or}
\begin{equation}\label{eq:a increasing}
    a(x,t) \leq C_a a(x+(s-t)\zeta,s) \quad \text{for all } (x,t),(x,s)\in \Omega_T \quad \text{with }\,\, s \geq t.
\end{equation}
This assumption is a generalization of the assumptions considered in \cite{SS25}. As observed in \cite{SS25}, the assumptions are not optimal in the sense that one can find functions such as $a(x,t)=\max\{-(x+t+1), 0\},$ that do not satisfy the assumption on a cylinder centered at a point $(x,t)$ that lies on the line $x=-(t+1).$ Such functions are now covered by assumptions \eqref{eq:a decreasing}-\eqref{eq:a increasing}.
 \subsection{Main result} 
To state the main theorem, we first recall the notion of a weak solution and
introduce the solution space associated with~\eqref{eq:main equation}:
\begin{equation*}
  W^H(\Omega_T) := \biggl\{ u \in L^1((0,T), W^{1,1}(\Omega,\RR) ):
  \iint_{\Omega_T} \bigl(|Du|^p + a(z)|Du|^q\bigr) \, dz < \infty \biggr\}.
\end{equation*}
We also denote the parabolic cylinders by
\[
  Q_r(z_0) := B_r(x_0) \times (t_0 - r^2,\, t_0 + r^2).
\]

\begin{definition}[Weak solution]\label{def:weak_solution}
A function $u : \Omega \times (0, T] \to \RR$ with
$u \in C_{\mathrm{loc}}\!\bigl((0, T];\, L^2_{\mathrm{loc}}(\Omega, \RR)\bigr)
\cap W^H(\Omega_T)$ is a weak solution of \eqref{eq:main equation} if
\[
  \iint_{\Omega_T} \bigl(- u\, \partial_t \varphi + A(z, u, Du) \cdot D\varphi\bigr) \, dz = 0
\]
holds for every nonnegative test function $\varphi \in C^{\infty}_0(\Omega_T, \RR)$.
\end{definition}
Our main result of this paper is the following.
\begin{theorem}[Gradient higher integrability]\label{main theorem}
Let $u$ be a weak solution of \eqref{eq:main equation} that is locally $\gamma$-H\"older continuous in space and $\frac{\gamma}{q}$-H\"older continuous in time. Also, assume that $a$ satisfies \eqref{eq:Z_kappa} and \eqref{eq:a increasing}. Then there exist constants $\varepsilon_0 = \varepsilon_0(\texttt{data}) > 0$
and $c = c(\texttt{data}) \ge 1$ such that
\[
  \miint{Q_r(z_0)} H(z, |Du|)^{1+\varepsilon} \, dz
  \le c \left( \frac{2 + \|a\|_\infty}{r^{q - 2\gamma}} \right)^{\! 1 + \frac{\varepsilon q}{2}}
\]
for every $\varepsilon \in (0, \varepsilon_0)$, whenever $Q_{3r}(z_0) \subset \Omega_T$.
\end{theorem}

The H\"older continuity in the above theorem is an a priori assumption. For bounded weak solutions to the model equation \eqref{eq:main_double}, under the coefficient assumption \eqref{eq:Z_kappa} of Definition \ref{def:Zkappa} and the restriction $q-p\leq \kappa,$ Theorem \ref{holder cts thm} provides a H\"older exponent $\upkappa\in (0, 1)$ such that $u \in C^{0, \upkappa, \frac{\upkappa}{q}}_{\mathrm{loc}}(\Omega_T).$ Thus the above result, Theorem \ref{main theorem}, applies with $\gamma=\upkappa$ in that setting. In the improved range $\kappa(q-2\gamma)<(q-p)(q-2\gamma)\leq q\kappa,$ H\"older continuity of $u$ is an additional assumption and is not supplied by Theorem \ref{holder cts thm}.

We close the introduction with several remarks on this result.

\begin{remark}\label{rem:comments}
\textup{(i)} The new upper bound~\eqref{eq:new_upper_bound} reflects the a priori time
H\"older regularity of solutions to~\eqref{eq:main equation} assumed in Theorem \ref{main theorem} and differs from the
classical elliptic Lavrentiev gap bound for H\"older continuous solutions. In
fact, for $p \ge 2$ it is worse than the elliptic one in the sense that
\[
  \frac{q\kappa}{q - 2\gamma} \le \frac{\kappa}{1 - \gamma},
\]
where $\gamma$ denotes the H\"older exponent of $u$. We do not know whether
this new upper bound is sharp. In the elliptic case, the upper bound
$q-p \le \kappa/(1 - \gamma)$ is known to be sharp. In the parabolic setting,
the phenomenon is governed entirely by the time regularity: in particular, if
the time regularity could be improved from $\gamma/q$ to $\gamma/2$, the same
gap condition as in the elliptic case would be recovered. Note also that
$q\kappa/(q - 2\gamma)$ blows up when $\gamma \to 1$ and $q\to 2$.

\textup{(ii)} Since our higher integrability proof broadly follows the
framework of~\cite{2023_Gradient_Higher_Integrability_for_Degenerate_Parabolic_Double-Phase_Systems, S25},
it carries over to the singular regime and to systems. Nevertheless, we restrict our attention to scalar equations in the degenerate range, as,
to the best of our knowledge, corresponding full local H\"older regularity
results under the present coefficient assumptions have not yet been
established in the singular case or for systems. In the setting considered
here, we prove local H\"older continuity of bounded weak solutions for
$a\in Z^\kappa(\Omega_T)$ and $q-p\leq\kappa$. As mentioned earlier, this result does not cover
the improved range
\[
  \kappa<q-p\leq\frac{q\kappa}{q-2\gamma}.
\]
In this range, the regularity
$u\in C_{\mathrm{loc}}^{0,\gamma,\gamma/q}(\Omega_T)$ is an additional
hypothesis of our higher integrability theorem and is not deduced from
the H\"older regularity result. For the same reason, the main statements
concern the homogeneous problem. That said, our approach accommodates
natural nonhomogeneous sources of the form
\begin{equation*}
  \partial_tu - \operatorname{div}\left(|Du|^{p-2}Du + a(z)|Du|^{q-2}Du\right)
  = -\operatorname{div}\left(|F|^{p-2}F + a(z)|F|^{q-2}F\right),
\end{equation*}
for a given vector field $F : \Omega_T \to \RR$ satisfying
\[
  \iint_{\Omega_T} H(z, |F|)\, dz
  = \iint_{\Omega_T} \bigl(|F|^p + a(z)|F|^q\bigr)\, dz < \infty.
\]
To obtain higher integrability in the presence of such a source, an
integrability assumption on $H(z, |F|)$ stronger than $L^1(\Omega_T)$,
which is typically required in Lemma~\ref{lem: lambda rho relation}
to control the growth of $H(z, |F|)$; see~\cite{CGK2025}.

Although H\"older regularity for the parabolic double phase equation is known
in the homogeneous case, an Ishii--Lions type argument can be adapted to
handle the analogous regularity for the equation with a nonhomogeneous term,
provided that $a$ is Lipschitz in space so that the viscosity framework is well
posed. In the absence of such Lipschitz continuity of $a$, no viscosity
formulation is currently available. Consequently, while our higher
integrability theory extends to systems and to the singular regime, the
existing H\"older regularity statements are confined to scalar, homogeneous
equations (with nonhomogeneous extensions available only under the additional
Lipschitz continuity of $a$).

\textup{(iii)} A key distinction from the recent works~\cite{CGK2025, KO2025}
is that these papers assume
\begin{equation}\label{eq:Lq_assumption}
  u \in L^q\bigl(0, T;\, W^{1,q}(\Omega)\bigr)
\end{equation}
as an {\em a priori} requirement, which is not guaranteed in this setting. In
particular, the standard Lipschitz truncation method is not applicable here:
deriving an energy estimate via Lipschitz truncation requires a prescribed
relation between the scaling parameter $\lambda$ and the radius of the
parabolic cylinder, while in our approach this very relation is itself
extracted from the energy estimate. Starting from a truncation would
therefore lead to a circular dependence.

In place of~\eqref{eq:Lq_assumption}, we employ the slanted Steklov average as a time mollification; see Section \ref{subsec:slanting} for a precise definition.

\textup{(iv)} We work with $a \in \mathcal{Z}^{\kappa}(\Omega_T)$ for
$\kappa \in (0, \infty)$, which is a strictly larger class than
$C^{\alpha, \alpha/2}(\Omega_T)$ with $\alpha \in (0, 1]$. This class was
introduced in~\cite{BCFM2023} in the study of the Lavrentiev phenomenon for
the elliptic double phase functional. To our knowledge, the present paper is
the first to use this class to investigate the regularity theory of parabolic
double phase equations. Previously, the class was used in the regularity
theory of non-autonomous elliptic problems~\cite{HHL_2021, HOK2021, HOK2022}.
Note that for every $\kappa \in (0, \infty)$, the function
$x \mapsto |x|^{\kappa}$ belongs to $\mathcal Z^{\kappa}(\RR^N)$.

To give some insight into this class, we recall that, for
$\Omega \subset \RR^N$, every function in $\mathcal Z^{\kappa}(\Omega)$ with
$\kappa \in (0, 1]$ is comparable to a H\"older continuous function, and
$C^{1+\kappa}(\Omega) \subset \mathcal Z^{1+\kappa}(\Omega)$. Moreover, a function $a$
belongs to $\mathcal Z^{\kappa}$ for some $\kappa \in (0, \infty)$ if and only if there
exists $\tilde{a}$ comparable to $a$ such that $\tilde{a}^{1/\kappa}$ is
Lipschitz. However, as observed in~\cite[Proposition~1.3]{BCFM2023}, the
class $\mathcal Z^{\kappa}$ does not share many features with smoothness, in general.
Indeed, the function $t \mapsto t^2$ is $C^{\infty}$ but does not belong to
$\mathcal Z^{2 + \delta}$ for any $\delta > 0$.
\end{remark}

 \section{Preliminaries}
 \subsection{Notations and definitions}\label{notations and definitions} Here we collect the standard notation and definitions  that will be used throughout the paper:
\begin{enumerate}[label=(\roman*),series=theoremconditions]
\item \label{not1} We denote by $N$ the space dimension and by $z=(x,t)$ a point in $ \RR^N\times (0,T]$. We will always assume that $N \geq 2$.   
\item\label{not2} A ball with center $x_0 \in \RR^N$ and radius $r$ is denoted as
 \begin{align*}
     B_r(x_0)=\left\{x \in \RR^N : |x-x_0|<r\right\}.
 \end{align*}
\item\label{general cylinder} In general, a parabolic cylinder centered at $z_0=(x_0, t_0)$ is denoted as 
\begin{align*}
    U_{r, \varrho}(z_0):=B_r(x_0)\times \left(t_0-\varrho, t_0+\varrho\right)=:B_r(x_0)\times l_{\varrho}(t_0).
\end{align*}

\item\label{not5} Let $\lambda\geq 1.$ We shall use the following three types of cylinders for the $p$-,and $(p, q)$-phases. 

\noindent The notation $Q_{\varrho, \lambda}(z_0)$ is used to denote $p$-intrinsic cylinders at $z_0=(x_0, t_0),$
\begin{align*}
Q_{\varrho, \lambda}(z_0):=B_{\varrho}(x_0)\times \left(t_0-\lambda^{2-p}\varrho^2, t_0+\lambda^{2-p}\varrho^2\right)
=:B_{\varrho}(x_0)\times I_{\varrho, \lambda}(t_0).
\end{align*}
The notation $G_{\varrho, \lambda}(z_0)$ is used to denote $(p,q)$-intrinsic cylinders at $z_0=(x_0, t_0),$
\begin{align*}
G_{\varrho, \lambda}(z_0):=B_{\varrho}(x_0)\times \left(t_0-\frac{\lambda^{2}\varrho^2}{\Lambda}, t_0+\frac{\lambda^{2}\varrho^2}{\Lambda}\right)
=:B_{\varrho}(x_0)\times J_{\varrho, \lambda}(t_0),
\end{align*}
where $\Lambda:=\lambda^p+a(z_0)\lambda^q.$
\item \label{not5.5} In the case $\la =1,$ we denote 
\[
Q_\varrho(z_0) := B_{\varrho}(x_0) \times (t_0 - \varrho^2,t_0+\varrho^2) =: B_\varrho(x_0) \times I_\varrho(t_0).
\]

\item\label{not7} Integration with respect to either space or time only will be denoted by a single integral $\int$ whereas integration on $\Omega\times (0, T)$ will be denoted by a double integral $\iint$. 
	
	\item\label{not9} The notation $a \lesssim b$ is shorthand for $a\leq c b$ where $c$ is a universal constant that depends on $``\text{\texttt{data}}"$ where 
 \begin{align*}
\text{\texttt{\texttt{data}}}:=\left(N, \widetilde M, p, q, \nu, L, [u]_{\gamma. \gamma/q}, \kappa, c_a, ||a||_{L^{\infty}(\Omega_T)}, ||H(z, |D u|)||_{L^1(\Omega_T)}\right).
 \end{align*}

 \item \label{integral avarage} We shall denote
 \begin{align*}
     (f)_{\mathcal{D}}:= \frac{1}{|\mathcal{D}|}\iint_{\mathcal{D}} f(z)\, dz= \miint{\mathcal{D}} f(z)\, dz
 \end{align*}
 as the integral average of $f$ over a measurable set $\mathcal{D} \subset \Omega_T$ with $0<|\mathcal{D}|<\infty.$ 

 \end{enumerate}
 We also introduce the definition of the parabolic metric and its scaled version.
\begin{definition}[Parabolic metric]\label{parabolic metric}
Given any two points $z_1=(x_1, t_1)$ and $z_2=(x_2, t_2)$ in $\RR^{N+1},$ we define the parabolic metric $d_p$ as 
\begin{align*}
d_p(z_1, z_2):=\max\left\{|x_1-x_2|, \sqrt{|t_1-t_2|}\right\}.
\end{align*}
\end{definition}
A suitable adaptation of the above definition in intrinsically scaled cylinders centered at $z_0\in \RR^{N+1}$ is as follows.
\begin{definition}[Scaled parabolic metric]\label{para metric}
Fix $z_0 \in \RR^{N+1}.$ Given any two points $z_1=(x_1, t_1), z_2=(x_2, t_2) \in \RR^{N+1}$ and $\lambda\geq 1$, we define the parabolic metric $d^{\la}_{p}(\cdot, \cdot)$ as
\begin{align*}
\displaystyle
d^{\la}_{p}(z_1,z_2):=\left\{\begin{array}{l}
\max\left\{|x_1-x_2|, \sqrt{\lambda^{p-2}|t_1-t_2|}\right\}, \,\,\,\,\text{in $p$-intrinsic case,}\\
\max\left\{|x_1-x_2|, \sqrt{\Lambda \lambda^{-2}|t_1-t_2|}\right\},\,\,\,\text{in $(p,q)$-intrinsic case,}
\end{array}\right.
\end{align*}
where $\La:=\la^p+a(z_0)\la^q.$   
\end{definition}
\subsection{Slanted Steklov average}\label{subsec:slanting}
In this section, we introduce the idea of a slanted Steklov average. For any $\zeta \in \RR^N,$ we define \textit{right slanted Steklov average} and \textit{left slanted Steklov average} respectively. 
\begin{align}\label{right slanted steklov}
[u]_{h}^{R}(z):=\frac{1}{h}\int_{t}^{t+h}u(x+(s-t)\zeta,s)\, ds,
\end{align}
and 
\begin{align}\label{left slanted steklov}
 [u]_{h}^{L}(z):=\frac{1}{h}\int_{t-h}^{t}u(x+(t-s)\zeta,s)\, ds.
\end{align}
As can be seen from the definition, the name ``slanted'' suggests a shift in the spatial direction. This makes the definition different from the usual Steklov average as the usual definition only averages in the time direction. We describe some advantages and limitations of this mollification below.

(i) First, we note that, similar to the standard spatial mollification, one can also use a space-time mollification. This was first explicitly written in \cite{OSS2025} and later used in \cite{HOK25, kimkimoh2026}; 
it avoids the assumptions \eqref{eq:a decreasing}--\eqref{eq:a increasing} on the coefficient, 
and it is enough to assume $a\in \mathcal{Z}^{\kappa}(\Omega_T)$. The precise definition of this space-time mollification is as follows. Let $\widetilde{\Theta}_M \in C^\infty(\RR)$ be the 
standard mollifier
\[
  \widetilde{\Theta}_M(s) :=
  \begin{cases}
    c_M\, \exp\!\left(\dfrac{1}{s^{2} - 1}\right) & \text{if } |s| < 1,\\[2mm]
    0 & \text{if } |s| \ge 1,
  \end{cases}
\]
where $c_M > 0$ is chosen so that
$\int_{\RR^M} \widetilde{\Theta}_M(|x|)\, dx = 1$. We define
\[
  \Theta(z) = \Theta(x, t)
        := \widetilde{\Theta}_N(|x|)\, \widetilde{\Theta}_1(t),
  \qquad
  \Theta_h(x, t)
        := \frac{1}{h^{N+2}}\, \Theta\!\left(\frac{x}{h},\, \frac{t}{h^{2}}\right)
  \quad \text{for } h > 0.
\]
Then $\Theta_h \in C^\infty(\RR^{N+1})$, $\Theta_h \ge 0$,
$\|\Theta_h\|_{L^1(\RR^{N+1})} = 1$, and
$\mathrm{supp}(\Theta_h) \subset Q_h(0)$. Moreover, with constants
depending only on $N$, we have
\[
  0 \le \Theta_h \le c\, h^{-(N+2)},
  \qquad
  |D\Theta_h| \le c\, h^{-(N+3)},
  \qquad
  |\partial_t \Theta_h| \le c\, h^{-(N+4)}.
\]
For a bounded open set $U \subset \RR^{N+1}$ and $f \in L^1(U)$, the
parabolic mollification of $f$ is
\[
  [f]^{h}(z)
  := (f * \Theta_h)(z)
  = \iint_{\RR^{N+1}} f(z - \sigma)\, \Theta_h(\sigma)\, d\sigma
  = \iint_{Q_h(0)} f(z - \sigma)\, \Theta_h(\sigma)\, d\sigma,
\]
defined whenever $Q_h(z) \subset U$. Under this mollification, as 
$h \to 0$, we would have
\begin{align*}
    &\iint_{U_{R, S}(z_0)}|[f]^h-f|^p+a(z)|[f]^h-f|^q\, dz\to 0,\\
    &\iint_{U_{R, S}(z_0)}|D[f]^h-Df|^p+a(z)|D[f]^h-Df|^q\, dz\to 0, \\
    &\iint_{U_{R, S}(z_0)}\left[|Df|^{p-2}Df+a(z)|Df|^{q-2}Df\right]^h\cdot D[f]^h\,dz\to \iint_{U_{R, S}(z_0)}|Df|^p+a(z)|Df|^q\,dz.
\end{align*}
Notice that the most crucial convergence is the third one. Once we 
have these convergences, we can prove 
Lemma~\ref{lem: energy inequality} and Lemma~\ref{Caccio 1}, and 
completely avoid assumptions  \eqref{eq:a decreasing}--\eqref{eq:a increasing} in our main result. However, to use this mollification, we need the  $(\kappa, \frac{\kappa}{2} )$-regularity in space and time of the coefficient $a$ with  $\kappa \in (0, 1].$ Since, in our case, the exponent $\kappa$ can be greater than $1$ the above convergences are not entirely clear with the gap bound \eqref{eq:new_upper_bound} and we need to seek a different route than the usual one.

(ii) Another important motivation for our mollification is to understand the optimal time regularity of the coefficient $a$. Our strategy is to use separate approximation lemmas in space and in time; see Lemma~\ref{lem: spcae approximation} and Lemma~\ref{lem:steklov lemma}. In this way, we do not need H\"older continuity of $a$ in time; continuity in time is enough. In place of this, we assume the almost monotonicity \eqref{eq:a decreasing}--\eqref{eq:a increasing}, which essentially matters only on the small cylinders that meet the zero set of $a$. When $a \geq \delta > 0$, assumptions \eqref{eq:a decreasing}--\eqref{eq:a increasing} are trivially satisfied. The novelty of our approach is that the time regularity part of \eqref{eq:Z_kappa} is not used in the approximation lemmas and hence not in the energy estimates; only the spatial regularity of $a$ is used at this stage. For parabolic double phase problems, approximation is a crucial step, and since $a$ can vanish in some set, to perform the approximation in a correct way, one needs some regularity of $a$ in time. Our monotonicity assumption on $a$ is a direct, pointwise condition in time, used as such in the time approximation, whereas H\"older regularity in time enters only indirectly, through the density theorem in generalized Orlicz spaces. However, we use the $\frac{\kappa}{2}$-regularity in time of $a$ as in \eqref{eq:Z_kappa} in the other parts of the paper.

 \subsection{Approximation lemmas and energy estimate} This section begins with the time and space approximation lemmas, where we use the slanted Steklov average, and those lemmas are used to prove the Caccioppoli estimate.
 \begin{lemma}[Space approximation]\label{lem: spcae approximation}
Let $2\leq p\leq q\leq p+\frac{q\kappa}{q-2\gamma}$ and $R,S>0$. Moreover, suppose
that 
\[a(x_1, t)\leq c_a \left(a(x_2, t)+|x_1-x_2|^{\kappa}\right)\quad \text{for}\quad (x_1, t), (x_2, t)\in \Omega\times \{t\},\] and 
$u\in W^{H}(U_{R,S})\cap L^\infty(U_{R,S})$ is uniformly $\gamma$-H\"older continuous in space when $\gamma>0$, and is compactly supported in $U_{R, S}$. For $\delta>0$, define
\[
u_{\delta}(x,t):=\int_{\RR^{N}}u\left(x-y,t\right)\varrho_{\delta}(y)\,dy,
\]
where $\varrho_{\delta}$ is the standard mollifier and the integrand
is read as zero when $x-y\not\in U_{R, S}$. Then $\supp u_{\delta}\Subset U_{R, S}$
for all small enough $\delta>0$. Furthermore, $u_{\delta}\in W^{H}(U_{R, S})$ and
\begin{equation}
\lim_{\delta\rightarrow0}\iint_{U_{R, S}(z_0)}\left|Du-Du_{\delta}\right|^{p}+a(z)\left|Du-Du_{\delta}\right|^{q}\, dz=0.\label{eq:space approx grad conv}
\end{equation}
\end{lemma}
\begin{proof} We follow \cite[Theorem 3.2]{SS25}. We need to verify \eqref{eq:space approx grad conv} and that
\[\int_{U_{R, S}(z_0)}|Du_{\delta}|^p+a(z)|Du_{\delta}|^q\, dz< \infty.\]
\textbf{Step 1:} Let $d:=\operatorname{dist}(\supp u,\partial_{\mathcal{P}} U_{R, S}(z_0))$, $\delta_{0}:=d/8$
and $r:=R-d/4, \tau=S-d/4$. Then $u_{\delta}\equiv0$ in $U_{R, S}\setminus U_{r, \tau}$
for all $\delta<\delta_{0}$. Fix $(x,t)\in U_{r, \tau}$, where $t$ is such that $u(\cdot, t)$ is $\gamma$-H\"{o}lder continuous in $B_r.$ Further, assume that $\delta<\delta_{0}$ and let $x_{t}^{\ast}\in\overline{B}_{\delta}(x)$
be such that 
\[
a(x_{t}^{\ast},t)=\inf_{y\in B_{\delta}(x)}a(y,t),
\]
which is well defined since $B_{\delta}(x)\Subset B_{r}.$ Using the uniform spatial H\"older bound when $\gamma>0$, and boundedness when $\gamma=0$, we have
\begin{align*}
&\left|Du_{\delta}(x,t)\right|=\left|D\int_{\RR^{N}}u\left(x-y,t\right)\varrho_{\delta}(y)\, dy\right|\\
& \leq \left|\int_{B(x, \delta)}\left(u(y,t)-u(x,t)\right)D\varrho_{\delta}(x-y)\,dy\right|+\left|u(x, t)\int_{B(x, \delta)}D\varrho_{\delta}(x-y)\, dy\right|\\
 & \leq([u]_{\gamma}+||u||_{\infty})\delta^{\gamma}\left\Vert D\varrho_{\delta}\right\Vert _{L^{1}(\RR^{N})}\leq([u]_{\gamma}+||u||_{\infty})\delta^{-1+\gamma}\left\Vert D\varrho\right\Vert _{L^{1}(\RR^{N})}=:\delta^{-1+\gamma}C_{0}.
\end{align*}
Here $[u]_\gamma:=\operatorname*{ess\,sup}_{t}[u(\cdot,t)]_{C^{0,\gamma}}$ when $\gamma>0$, and $[u]_0:=\|u\|_\infty$ when $\gamma=0$. In the latter case the gap bound is $q-p\leq\kappa$.
Note that,  we also used that $\left\Vert D\varrho_{\delta}\right\Vert _{L^{1}(\RR^{N})}=\delta^{-1}\left\Vert D\varrho\right\Vert _{L^{1}(\RR^{N})}$ and the fundamental theorem of calculus to conclude $\int_{B(x, \delta)}D\varrho_{\delta}(x-y)\, dy=0.$
Using the above display, the spatial $\mathcal{Z}^{\kappa}$ inequality
as well as that $\left|x-x_{t}^{\ast}\right|\leq\delta$, we obtain
\begin{align*}
 & \left|Du_{\delta}(x,t)\right|^{p}+a(x,t)\left|Du_{\delta}(x,t)\right|^{q}\nonumber \\
 & =\frac{\left|Du_{\delta}(x,t)\right|^{p}+a(x,t)\left|Du_{\delta}(x,t)\right|^{q}}{\left|Du_{\delta}(x,t)\right|^{p}+c_a a(x_{t}^{\ast},t)\left|Du_{\delta}(x,t)\right|^{q}}\left(\left|Du_{\delta}(x,t)\right|^{p}+c_a a(x^{\ast}_t,t)\left|Du_{\delta}(x,t)\right|^{q}\right)\nonumber \\
 & =\left(1+\frac{(a(x,t)-c_ a a(x_{t}^{\ast},t))\left|Du_{\delta}(x,t)\right|^{q}}{\left|Du_{\delta}(x,t)\right|^{p}+c_a a(x_{t}^{\ast},t)\left|Du_{\delta}(x,t)\right|^{q}}\right)\left(\left|Du_{\delta}(x,t)\right|^{p}+c_a a(x^{\ast}_t,t)\left|Du_{\delta}(x,t)\right|^{q}\right)\nonumber \\
 & \leq\left(1+\frac{c_a \delta^{\kappa} \left|Du_{\delta}(x,t)\right|^{q}}{\left|Du_{\delta}(x,t)\right|^{p}}\right)\left(\left|Du_{\delta}(x,t)\right|^{p}+c_a a(x_{t}^{\ast},t)\left|Du_{\delta}(x,t)\right|^{q}\right)\nonumber \\
 & \leq\left(1+c_a C^{q-p}_0\delta^{\kappa}\left(\delta^{-1+\gamma}\right)^{q-p}\right)\left(\left|Du_{\delta}(x,t)\right|^{p}+c_a a(x_{t}^{\ast},t)\left|Du_{\delta}(x,t)\right|^{q}\right)\nonumber \\
 &\leq \left(1+c_a C^{q-p}_0 \delta^{\kappa}\left(\delta^{-1+\gamma}\right)^{\frac{q\kappa}{q-2\gamma}}\right)\left(\left|Du_{\delta}(x,t)\right|^{p}+c_a a(x_{t}^{\ast},t)\left|Du_{\delta}(x,t)\right|^{q}\right)\nonumber \\
 &\leq \left(1+c_a C^{q-p}_0 \delta^{\kappa}\left(\delta^{-1+\gamma}\right)^{\frac{\kappa}{1-\gamma}}\right)\left(\left|Du_{\delta}(x,t)\right|^{p}+c_a a(x_{t}^{\ast},t)\left|Du_{\delta}(x,t)\right|^{q}\right)\nonumber \\
 & =\left(1+c_a C^{q-p}_0\right)\left(\left|Du_{\delta}(x,t)\right|^{p}+c_a a(x_{t}^{\ast},t)\left|Du_{\delta}(x,t)\right|^{q}\right)\nonumber \\
 & =:C_{1}\left(\left|Du_{\delta}(x,t)\right|^{p}+a(x_{t}^{\ast},t)\left|Du_{\delta}(x,t)\right|^{q}\right)
\end{align*}
where $C_1:=\max\{(1+C^{q-p}_0 c_a), c_a(1+C^{q-p}_0 c_a)\}.$
Further, since $\int_{B_{\delta}}\rho_{\delta}(y)\, dy=1$, we can apply Jensen's inequality to estimate
\begin{align*}
 & \left|Du_{\delta}(x,t)\right|^{p}+a(x_{t}^{\ast},t)\left|Du_{\delta}(x,t)\right|^{q}\\
 & =\left|\int_{B_{\delta}}Du(x-y,t)\rho_{\delta}(y)\, d y\right|^{p}+a(x_{t}^{\ast},t)\left|\int_{B_{\delta}}Du\left(x-y,t\right)\rho_{\delta}(y)\, d y\right|^{q}\\
 & \leq\int_{B_{\delta}}\left|Du(x-y,t)\right|^{p}\rho_{\delta}(y)\, dy+\int_{B_{\delta}}a(x_{t}^{\ast},t)\left|Du\left(x-y,t\right)\right|^{q}\rho_{\delta}(y)\, dy\\
 & \leq\int_{B_{\delta}}\left|Du(x-y,t)\right|^{p}\rho_{\delta}(y)\, d y+\int_{B_{\delta}}a(x-y,t)\left|Du\left(x-y,t\right)\right|^{q}\rho_{\delta}(y)\, d y,\\
 & =(\left|Du\right|^{p})_{\delta}(x,t)+(a\left|Du\right|^{q})_{\delta}(x,t),
\end{align*}
where the last estimate follows from the choice of $x_{t}^{\ast}$. Combining the last two displays, we arrive at
\begin{equation}
\left|Du_{\delta}(x,t)\right|^{p}+a(x,t)\left|Du_{\delta}(x,t)\right|^{q}\leq C_{1}(\left|Du\right|^{p}+a\left|Du\right|^{q})_{\delta}(x,t)\label{eq:space approx 4-1}
\end{equation}
for all $x\in B_{r}(x_0)$ and almost all $t\in(t_0-\tau, t_0+\tau)$.

\textbf{Step 2:} We define $w:\RR^{N+1}\rightarrow\RR$
by setting $w=(\left|Du\right|^{p}+a\left|Du\right|^{q})$ in $U_{r, \tau}$
and $w=0$ in $\RR^{N+1}\setminus U_{r, \tau}.$ Since $w\in L^{1}(\RR^{N+1})$
with compact support, we have $w_{\delta}\rightarrow w$ in $L^{1}(\RR^{N+1})$
by the standard properties of the mollifier. On the other hand, by (\ref{eq:space approx 4-1}),
we have
\[
\left|Du_{\delta}\right|^{p}+a\left|Du_{\delta}\right|^{q}\leq C_{1}w_{\delta}\quad\text{almost everywhere in }U_{r, \tau}.
\]
It follows that $\left|Du_{\delta}\right|^{p}+a\left|Du_{\delta}\right|^{q}$
is uniformly integrable in $U_{r, \tau}$ with respect to $\delta$.
Therefore, by Vitali's convergence theorem, 
\[
\left|Du_{\delta}\right|^{p}+a\left|Du_{\delta}\right|^{q}\rightarrow\left|Du\right|^{p}+a\left|Du\right|^{q}\quad\text{in }L^{1}(U_{r, \tau}).
\]
Then observe that
\begin{align*}
 & \int_{U_{r, \tau}}\left|Du-Du_{\delta}\right|^{p}+a(z)\left|Du-Du_{\delta}\right|^{q}\, d z\\
 & \leq\int_{U_{r, \tau}}2^{p-1}(\left|Du\right|^{p}+\left|Du_{\delta}\right|^{p})+2^{q-1}a(z)(\left|Du\right|^{q}+\left|Du_{\delta}\right|^{q})\, d z,
\end{align*}
so that $\left|Du-Du_{\delta}\right|^{p}+a\left|Du-Du_{\delta}\right|^{q}$
is again uniformly integrable, and by Vitali's convergence theorem,
the left-hand side of the above inequality tends to zero as $\delta\rightarrow0$.
Since both $\left|Du\right|$ and $\left|Du_{\delta}\right|$ vanish
in $U_{R, S}\setminus U_{r, \tau}$ (provided that $\delta<\delta_{0}$),
the claim follows.
\end{proof}
Next, we prove the time approximation lemma. For that, we recall the definitions \eqref{right slanted steklov} and \eqref{left slanted steklov}.
\begin{lemma}[Time approximation]\label{lem:steklov lemma}
Let $a:U_{R, S}(z_0) \rightarrow [0,\infty)$ be bounded and measurable and $u \in W^H(U_{R, S})$ and compactly supported in $U_{R, S}(z_0).$ If \eqref{eq:a increasing} holds, then for any small $h_{0}>0$, we have
\[
\lim_{h\rightarrow0}\iint_{B_{R}(x_0)\times (t_0-S, t_0+S+h_0)}\left|D[u]^R_h-Du \right|^p+a(z)\left|D[u]^R_h-Du\right|^{q} \, dz =0,
\] up to a subsequence.
Moreover, if $u\in W^{H}(U_{R, S}(z_0)){\cap C(U_{R, S}(z_0))} {\color{blue}}$, then $[u]_{h}^{R}\in W^{H}(U_{R, S}(z_0))$
with $D[u]_{h}^{R}=[Du]_{h}^{R}$ and the weak time
derivative of $[u]_{h}^{R}$ and $[u]^L_{h}$ are given by
\[
-\partial_{t}[u]_{h}^{R}(x,t)=\frac{u(x, t)-u(x+h\zeta,t+h)}{h}+\zeta\cdot D[u]_{h}^{R}(x,t),
\]
and
\[\partial_{t}[u]_{h}^{L}(x,t)=\frac{u(x, t)-u(x+h\zeta,t-h)}{h}+\zeta\cdot D[u]_{h}^{L}(x,t).\]
If (\ref{eq:a decreasing}) holds, then we have the analogous convergence result for $[u]^{L}_h$.
\end{lemma}
\begin{proof}
Consider the linear mapping $\Phi:\RR^{N}\rightarrow\RR^{N}$,
\[
\Phi(x,t)=(x-t\zeta,t)\quad\text{and its inverse}\quad\Phi^{-1}(x,t)=(x+t\zeta,t).
\]
Since $u$ is compactly supported in $U_{R, S}(z_0)$, there exists a cylinder
$U_{r, \tau}(z_0)\Subset U_{R, S}(z_0)$ such that for small enough $h>0$, we have
\begin{equation*}
u\equiv 0 \,\,\, \text{and}\,\,\,  [u]_{h}^{R}\equiv0\quad\text{in }U_{R, S}(z_0)\setminus U_{r, \tau}(z_0),
\end{equation*}
and 
\[
(x+h\zeta,t+h)\in U_{R, S}(z_0)\quad\text{for all }(x,t)\in U_{r, \tau}(z_0).
\]
Denote the slanted cylinders $U^{\prime}_{r, \tau}:=\Phi(U_{r, \tau}),$ $U_{R, S}^{\prime}:=\Phi(U_{R, S})$.
 Denote by $[\cdot]_{h}$ the standard Steklov average (without slanting). Fix $h_{0}>0$ and suppose that $h<h_{0}$. For $i \in \{1,\ldots,N\}$, we denote $v := \partial_i u$. By standard properties of Steklov averages, we have $\partial_i[u]_h = [\partial_i u]_h = [v]_h$.
Using \eqref{eq:a increasing} we have for all $(x,t)\in U^{\prime}_{r, \tau}$ 

\begin{align} \label{eq:bla3}
\left(a\circ\Phi^{-1}\left|[v\circ\Phi^{-1}]_{h}\right|^{q}\right)(x,t) & =\left|\frac{1}{h}\int_{t}^{t+h}a^{1/q}(x+t\zeta,t)v(x+s\zeta,s)\, ds\right|^{q}\nonumber\\ 
 & \leq C_{0}\left(\frac{1}{h}\int_{t}^{t+h}a^{1/q}(x+s\zeta,s)\left|v(x+s\zeta,s)\right|\, ds\right)^{q} \nonumber \\ 
 & =C_{0}[(a^{1/q}\left|v\right|)\circ\Phi^{-1}]_{h}^{q}(x,t).
\end{align}
Since $(a^{1/q}\left|v\right|)\circ\Phi^{-1}\in L^{q}(U_{r, \tau}^{\prime})$,
the above display and standard properties of Steklov averages imply
that
\[
\iint_{U^{\prime}_{r, \tau}(z_0)}(a\circ\Phi^{-1})\left|[v\circ\Phi^{-1}]_{h}\right|^{q}\, dz\leq C<\infty \quad\text{for all small \ensuremath{h>0,}}
\]
where $C$ is independent of $h$. On the other hand, for any $(x,t)\in U^{\prime}_{r, \tau} (z_0)$, we have
\begin{align}
[v\circ\Phi^{-1}]_{h}(x,t) & =\frac{1}{h}\int_{t}^{t+h}v(x+s\zeta,s)\,ds\nonumber \\
 & =\frac{1}{h}\int_{t}^{t+h}v(x+t\zeta+(s-t)\zeta,s)\, ds\nonumber \\
 & =\frac{1}{h}\int_{t}^{t+h}v(\Phi^{-1}(x,t)+(s-t)\zeta,s)\,ds\nonumber \\
 & =[v]_{h}^{R}\circ\Phi^{-1}(x,t).\label{eq:bla2}
\end{align}
Then by change of variables and the last two displays, we obtain (note that $\det J_{\Phi}=1$)
\begin{align*}
\iint_{U_{r, \tau}}a(x, t)\left|[v]_{h}^{R}(x, t)\right|^{q}\,dx\, dt & =\iint_{U_{r, \tau}^{\prime}}(a\circ\Phi^{-1}(y, \tau))\left|[v]_{h}^{R}\circ\Phi^{-1}(y, \tau)\right|^{q}\,dy\, d\tau\\
 & =\iint_{U_{r, \tau}^{\prime}}(a\circ\Phi^{-1}(x, t))\left|[v\circ\Phi^{-1}]_{h}(y, \tau)\right|^{q}\,dy\,d\tau<\infty
\end{align*}
for small $h>0$. Since $[v]_{h}^{R}$ vanishes outside of $U_{r, \tau}$,
it follows that $[v]_{h}^{R}\in L^{H}(U_{r, \tau})$. To show the convergence, observe that by (\ref{eq:bla2}) we have
\begin{align}
\iint_{U_{r, \tau}}a(x , t)\left|([v]_{h}^{R}-v)(x, t)\right|^{q}\,dx\, dt & =\iint_{U^{\prime}_{r, \tau}}a\circ\Phi^{-1}(y, \tau)\left|[v]_{h}^{R}\circ\Phi^{-1}(y, \tau)-v\circ\Phi^{-1}(y, \tau)\right|^{q}\,dy\, d\tau\nonumber \\
 & =\iint_{U^{\prime}_{r, \tau}}a\circ\Phi^{-1}(y,s)\left|[v\circ\Phi^{-1}]_{h}(y, \tau)-v\circ\Phi^{-1}(y, \tau)\right|^{q}\,dy\,d\tau.\label{eq:convegence}
\end{align}
Observe that 
\begin{align*}
a\circ\Phi^{-1}\left|[v\circ\Phi^{-1}]_{h}-v\circ\Phi^{-1}\right|^{q}&\leq2^{q-1}a\circ\Phi^{-1}\left(\left|[v\circ\Phi^{-1}]_{h}\right|^{q}+\left|v\circ\Phi^{-1}\right|^{q}\right) \\&\leq2^{q-1}\left(C_0 [(a^{1/q}|u|)\circ\Phi^{-1}]_h^q +a\circ\Phi^{-1}|v\circ\Phi^{-1}|^q \right),
\end{align*}
where in the last inequality we used \eqref{eq:bla3}. Since the right-hand side of the above display converges in $L^{1}(U^{\prime}_{r, \tau})$ by the properties of the standard Steklov average, it follows that the left-hand side
is uniformly integrable in $U^{\prime}_{r, \tau}$. Therefore, the right-hand
side of (\ref{eq:convegence}) converges to zero.

Now we prove the second assertion. Let $\varphi\in C_{0}^{\infty}(Q)$. Then, using that
$u\circ\Phi^{-1}$ has a weak space derivative in $U^{\prime}_{r, \tau}$ with
$D(u\circ\Phi^{-1})=Du\circ\Phi^{-1}$, we obtain (using also the
same computation as in (\ref{eq:bla2}), but for $Du$)
\begin{align*}
&\iint_{U_{r, \tau}}[u]_{h}^{R}(x, t)D\varphi(x,t)\,dx\,dt\\
& =\iint_{U^{\prime}_{r, \tau}}([u]_{h}^{R}\circ\Phi^{-1}(y, \tau))(D\varphi\circ\Phi^{-1}(y, \tau))\,dy\, d\tau\\
 & =\iint_{U^{\prime}_{r, \tau}}([u\circ\Phi^{-1}]_{h}(y, \tau))D(\varphi\circ\Phi^{-1}(y, \tau))\,dy\, d\tau\\
 & =\iint_{U^{\prime}_{r, \tau}}-(\varphi\circ\Phi^{-1}(y, \tau))[D(u\circ\Phi^{-1}(y, \tau))]_{h}\,dy\, d\tau\\
 & =\iint_{U^{\prime}_{r, \tau}}-(\varphi\circ\Phi^{-1}(y, \tau))[Du\circ\Phi^{-1}(y, \tau)]_{h}\,dy\, d\tau\\
 & =\iint_{U^{\prime}_{r, \tau}}-(\varphi\circ\Phi^{-1}(y, \tau))[Du]_{h}^{R}\circ\Phi^{-1}(y, \tau)\,dy\, d\tau=\iint_{U_{r, \tau}}-\varphi[Du]_{h}^{R}(x, t)\,dx\,dt.
\end{align*}

Therefore, the weak derivative of $[u]_{h}^{R}$ is $[Du]_{h}^{R}$.
It follows from the previous statement that $[u]_{h}^{R}\in W^{H}(U_{R, S}(z_0))$. For the
time derivative, we compute
\begin{align*}
\iint_{U_{r, \tau}}[u]_{h}^{R}(x,t)\partial_{t}\varphi(x, t) \,dx \,dt & =\iint_{U^{\prime}_{r, \tau}}([u]_{h}^{R}\circ\Phi^{-1}(y,s))((\partial_{s}\varphi)\circ\Phi^{-1}((y,s))\,dy \, d\tau\\
 & =\iint_{U^{\prime}_{r, \tau}}[u\circ\Phi^{-1}]_{h}(y, \tau)((\partial_{s}\varphi)\circ\Phi^{-1}(y, \tau))\,dy\, d\tau\\
 & =\iint_{U^{\prime}_{r, \tau}}[u\circ\Phi^{-1}]_{h}(y, \tau)\partial_{s}(\varphi\circ\Phi^{-1}(y, \tau))\,dy\, d\tau\\
 &-\iint_{U^{\prime}_{r, \tau}}[u\circ\Phi^{-1}]_{h}(y, \tau)\zeta\cdot(D\varphi)\circ\Phi^{-1}(y, \tau)\,dy\, d\tau,
\end{align*}
where we used that $\partial_{t}(\varphi\circ\Phi^{-1})=\partial_{t}\varphi\circ\Phi^{-1}+\zeta\cdot D\varphi\circ\Phi^{-1}$.
Further, we have
\begin{align*}
&\iint_{U^{\prime}_{r, \tau}}[u\circ\Phi^{-1}]_{h}(y, \tau)\partial_{\tau}(\varphi\circ\Phi^{-1}(y, \tau))\,dy\, d\tau\\
& =-\iint_{U^{\prime}_{r, \tau}}\partial_{\tau}[u\circ\Phi^{-1}]_{h}(y, \tau)\varphi\circ\Phi^{-1}(y, \tau)\,dy\, d\tau\\
 & =-\iint_{U^{\prime}_{r, \tau}}\frac{u\circ\Phi^{-1}(y,\tau+h)-u\circ\Phi^{-1}(y,\tau)}{h}\varphi\circ\Phi^{-1}(y,\tau)\,dy\,d\tau\\
& =-\iint_{U^{\prime}_{r, \tau}}\frac{u\circ\Phi^{-1}\circ f(y,\tau)-u\circ\Phi^{-1} \circ id \,(y,\tau)}{h}\varphi\circ\Phi^{-1}(y,\tau)\,dy\,d\tau
\end{align*}
where $f(y, \tau)=(y, \tau+h).$

Combining the last two displays (and using again that $[u\circ\Phi^{-1}]_{h}=[u]_{h}^{R}\circ\Phi^{-1}$),
we have
\begin{align*}
&\iint_{U_{r, \tau}}[u]_{h}^{R}(x, t)\partial_{t}\varphi(x, t)\,dz\\
& =-\iint_{U^{\prime}, _{r, \tau}}\frac{u\circ\Phi^{-1}\circ f(y,\tau)-u\circ\Phi^{-1}\circ id\, (y,\tau)}{h}\varphi\circ\Phi^{-1}\circ id\, (y,\tau)\,dy\,d\tau.\\
 & \phantom{=}-\iint_{U^{\prime}_{r, \tau}}([u]_{h}^{R}\circ\Phi^{-1}\circ id\,(y, \tau))\zeta\cdot((D\varphi)\circ\Phi^{-1}\circ id\, (y, \tau))\,dy\, d\tau\\
&=\iint_{U^{\prime}_{r, \tau}}\varphi(x+t\zeta, t)\frac{u(x+t\zeta, t+h)-u(x+t\zeta, t)}{h}\, dz-\iint_{U_{r, \tau}}[u]^{R}_h \zeta\cdot D\varphi\, dz\\
&=\iint_{U_{r, \tau}}\Phi\circ \varphi(x+t \zeta, t)\frac{\Phi \circ u(x+t\zeta, t+h)-\Phi \circ u(x+t\zeta, t)}{h}\, dz\\
&-\iint_{U_{r, \tau}}(\Phi \circ [u]^{R}_h) \zeta\cdot (\Phi \circ D\varphi)\, dz\\
&=\iint_{U_{r, \tau}}\varphi(x-t\zeta +t\zeta, t)\frac{u(x-(t+h)\zeta+t\zeta, t+h)-u(x-t\zeta +t\zeta, t)}{h}\, dz\\
&-\iint_{U_{r, \tau}}[u]^{R}_h \zeta\cdot D\varphi\, dz \\
 & =-\iint_{U_{r, \tau}}\varphi(x,t)\frac{ u(x+h \zeta,t+h)- u(x,t)}{h}\,dx\, dt-\iint_{U_{r, \tau}}[u]_{h}^{R}(x, t)\zeta\cdot (D\varphi)(x, t)\, dx \, dt\\
 & =\iint_{U_{r, \tau}}\varphi(x,t)\frac{u(x,t)-u(x+h\zeta,t+h)}{h}\, dx\, dt+\iint_{U_{r, \tau}}\varphi D[u]_{h}^{R}\cdot\zeta\,dz,
\end{align*}
where we used the change of variable,
\begin{align*}
 u\circ\Phi^{-1}\circ f(y, \tau)=u\circ\Phi^{-1}\circ f\circ\Phi(x,t)&=u\circ\Phi^{-1}(f(x-t\zeta,t))\\
    &=u(\Phi^{-1}(x-t\zeta,t+h))\\
   &=u((x-t\zeta)+(t+h)\zeta,t+h)  =u(x+h\zeta, t+h),   
\end{align*}
and the same computation holds for $f=id.$
This completes the proof of $$-\partial_{t}[u]^{R}_h=\frac{u(x,t)-u(x+h \zeta,t+h)}{h}+D[u]_{h}^{R}\cdot\zeta$$
in the weak sense.
\end{proof}

Using these two approximation lemmas, we prove Caccioppoli type estimates for $2\le p\le q\le p +\frac{q\kappa}{q-2\gamma}.$ Here we only use the assumption \eqref{eq:a increasing}, i.e., for $\zeta \in \RR^N,$ we assume
\begin{align}\label{almost increasing for a}
 a(x,t) \leq C_a a(x+(s-t)\zeta,s) \quad \text{for all } (x,t),(x,s)\in U_{R, S}(z_0) \text{ with } t \leq s. 
\end{align}
\begin{lemma}[Caccioppoli inequality I] \label{lem: energy inequality}
 Let $u$ be a weak solution to \eqref{eq:main equation} and let $a$ satisfy \eqref{almost increasing for a}. Assume also the spatial coefficient inequality of Lemma~\ref{lem: spcae approximation}, and that $u$ is locally bounded and, when $\gamma>0$, locally uniformly $\gamma$-H\"older continuous in space. Then, for $U_{R,S}(z_0)\subset \Omega_T$, $r\in [R/2,R)$ and $\tau\in[S/2^2,S)$, there exists a constant $c_E$ depending on $N,p,q,\nu$ and $L$ such that the following inequality holds: 
    \begin{align*}
        &\sup_{t\in(t_0-\tau.t_0+\tau)}\dashint_{B_r(x_0)} \frac{|u-(u)_{U_{r,\tau}(z_0)}|^2}{\tau}\; dx +\miint{U_{r,\tau}(z_0)}H(z,|Du|)\;dz\\
        &\qquad \leq c_E\miint{U_{R,S}(z_0)}   H\left(z,\frac{|u-(u)_{U_{R,S}(z_0)}|}{R-r}\right) dz + c_E\miint{U_{R,S}(z_0)}\frac{|u-(u)_{U_{R,S}(z_0)}|^2}{S-\tau}\;dz.
    \end{align*}   
\end{lemma}
\begin{proof}
 \textbf{Step 1:} Suppose that $\varphi\in W^H(U_{R,S}(z_0))\cap L^\infty(U_{R,S}(z_0))$, uniformly $\gamma$-H\"older continuous in space when $\gamma>0$ with $\supp \varphi \Subset U_{R, S}(z_0)$. Then the Steklov average $[\varphi]^R_h$ and $[\varphi]^L_h$ are well defined for small enough $h>0$. Moreover, $[\varphi]^R_h \in W^H(U_{R, S}(z_0))$ by Lemma \ref{lem:steklov lemma} and $[\varphi]^R_h$ is bounded and, when $\gamma>0$, uniformly $\gamma$-H\"older continuous in space. Indeed, for sufficiently small $h>0$,
 \begin{align*}
 |[\varphi]_h^R(x, t)-[\varphi]_h^R(y, t)|
 &\leq\frac{1}{h} \int_{t}^{t+h} |\varphi(x+(s-t)\zeta, s)-\varphi(y+(s-t)\zeta, s)|\,ds\\
 &\leq c\, (\text{independent of $h$})\, |x-y|^{\gamma}.
 \end{align*}
Further, set
\[
([\varphi]^R_{h})_{\delta}(x,t):=\int_{\RR^{N}}[\varphi]^R_{h}(y,t)\rho_{\delta}(y-x)\, dy.
\]
Then by space approximation from Lemma \ref{lem: spcae approximation} we have that
\begin{equation}
\iint_{U_{R, S}(z_0)}\left|D([\varphi]^R_{h})_{\delta}-D[\varphi]^R_{h}\right|^{p}+a(z)\left|D([\varphi]^R_h)_{\delta}-D[\varphi]^R_{h}\right|^{q}\,dx\, dt\rightarrow0\quad\text{as }\delta\rightarrow0.\label{eq:modular lp}
\end{equation}
For fixed $h,\delta>0$, the weak derivative formula and spatial convolution give
\[
\|D([\varphi]^R_h)_\delta\|_\infty
\leq C\delta^{-1}\|\varphi\|_\infty,\qquad
\|\partial_t([\varphi]^R_h)_\delta\|_\infty
\leq \left(\frac2h+C|\zeta|\delta^{-1}\right)\|\varphi\|_\infty.
\]
Here the test function is extended by zero and $h,\delta$ are small enough to keep its support inside the cylinder.

Since $([\varphi]^R_{h})_{\delta}$ is Lipschitz and compactly supported
in $U_{R, S}(z_0)$, it can be used as a test function in the weak formulation Definition~\ref{def:weak_solution}. Thus, we obtain the following.
\begin{equation}
\begin{aligned}
&\iint_{U_{R, S}(z_0)}-u\partial_{t}([\varphi]^R_{h})_{\delta}+A(z, u,  Du)\cdot D([\varphi]^R_{h})_{\delta}\, dx\, dt=0.\label{eq:regularized est}
\end{aligned}
\end{equation}
Now, observe that by H\"{o}lder's inequality and the growth condition \eqref{eq:growth_cond}, i.e., $|A(z, u, Du)|\leq L(|Du|^{p-1}+a(z)|Du|^{q-1}),$
\begin{align*}
 & \iint_{U_{R, S}(z_0)}A(z, u, Du)\cdot(D([\varphi]^R_{h})_{\delta}-D[\varphi]^R_{h})\, dx\, dt\\
 & \leq\left(\iint_{U_{R, S}(z_0)}\left|Du\right|^{p} \,dx\, dt\right)^{\frac{p-1}{p}}\left(\iint_{U_{R, S}(z_0)}\left|D([\varphi]^R_{h})_{\delta}-D[\varphi]^R_{h}\right|^{p} \,dx\, dt\right)^{\frac{1}{p}}\\
 & \phantom{\leq}+\left(\iint_{U_{R, S}(z_0)}a(z)\left|Du\right|^{q} \,dx\, dt\right)^{\frac{q-1}{q}}\left(\iint_{U_{R, S}(z_0)}a(z)\left|D([\varphi]^R_{h})_{\delta}-D[\varphi]^R_{h}\right|^{q} \,dx\, dt\right)^{\frac{1}{q}},
\end{align*}
where the right-hand side converges to zero as $\delta\rightarrow0$
by \eqref{eq:modular lp}.  For fixed $h>0$, Lemma~\ref{lem:steklov lemma} gives
\[
\partial_t[\varphi]^R_h
=\frac{\varphi(x+h\zeta,t+h)-\varphi(x,t)}{h}
-\zeta\cdot D[\varphi]^R_h\in L^p(U_{R,S}(z_0)).
\]
Spatial convolution commutes with this weak derivative, so
\[
\partial_t(([\varphi]^R_h)_\delta)
=(\partial_t[\varphi]^R_h)_\delta
\longrightarrow\partial_t[\varphi]^R_h
\quad\hbox{strongly in }L^p(U_{R,S}(z_0))\subset L^2(U_{R,S}(z_0)).
\]
Since $u\in L^2(U_{R,S}(z_0))$ and $p\geq2$, H\"older's inequality yields convergence of the time term in \eqref{eq:regularized est}. Throughout this passage, $h$ is fixed; only after $\delta\to0$ do we let $h\to0$. Thus, we arrive at
\begin{align*}
&\iint_{U_{R, S}(z_0)}\left(-u\partial_{t}[\varphi]^R_{h}+A(z, u, Du)\cdot D[\varphi]^R_{h}\right)\, dx\, dt=0.
\end{align*}

We denote $Q:=U_{R,S}(z_0)$ and $T_h f(x,t):=f(x+h\zeta,t+h)$. Compact support and sufficiently small $h$ justify the following changes of variables (using zero extensions)
\begin{align*}
\iint_Q-u\partial_t[\varphi]^R_h\,dz
&=\iint_Q u\,\frac{\varphi-T_h\varphi}{h}\,dz
  +\iint_Q u\,\zeta\cdot D[\varphi]^R_h\,dz\\
&=\iint_Q\frac{u-T_{-h}u}{h}\,\varphi\,dz
  +\iint_Q u\,\zeta\cdot D[\varphi]^R_h\,dz\\
&=\iint_Q(\partial_t[u]^L_h-\zeta\cdot D[u]^L_h)\varphi\,dz\\
&\quad+\iint_Q\frac{u(x+h\zeta,t-h)-u(x-h\zeta,t-h)}{h}\varphi\,dz
  +\iint_Q u\,\zeta\cdot D[\varphi]^R_h\,dz.
\end{align*}

Moving the shifts in the middle term of the last line onto $\varphi$, we use again that $\varphi$ is compactly supported.
Hence, we arrive at the following.
\begin{align}\label{weak formulation_1}
&\iint_{U_{R, S}(z_0)}\left(\partial_{t}[u]^L_{h}-D[u]^L_h\cdot \zeta\right)\varphi\,dx\, dt+\iint_{U_{R, S}(z_0)}\frac{\varphi(x-h\zeta, t+h)-\varphi(x+h\zeta, t+h)}{h}u(x, t)\,dx\, dt \nonumber\\
 &+\iint_{U_{R, S}(z_0)}uD[\varphi]^R_h\cdot \zeta \, dz+\iint_{U_{R, S}(z_0)}A(z, u, Du)\cdot D[\varphi]^R_{h}\,dx\, dt=0
\end{align}

\textbf{Step 2:} Once we have the approximation, we can complete the proof by choosing the test function suitably. Let $\eta \in C^{\infty}_{0}(U_{R, S}(z_0))$ such that $0\leq \eta\leq 1, $ $\eta \equiv 1$ in $U_{r, \tau}(z_0),$ $|D\eta| \leq \frac{2}{R-r}$ and $|\partial_t\eta|\leq \frac{2}{S-\tau}.$ Further, we define
\begin{align*}
\beta^{\tau_1,\varepsilon}(t):=\begin{cases}
1, & t\leq\tau_1-\varepsilon,\\
1-(t-\tau_1+\varepsilon)/\varepsilon, & \tau_1-\varepsilon<t\leq\tau_1,\\
0, & t>\tau_1.
\end{cases}
\end{align*}
for $\varepsilon>0$ and $\tau_1 \in (t_0-\tau, t_0+\tau).$ Now we choose an admissible test function as
\[\varphi (x,t)=\eta^q (x, t) \beta^{\tau_1, \varepsilon}(t)\left(u-(u)_{U_{R, S}(z_0)}\right)(x, t)=\eta^q(x, t)\beta^{\tau_1, \varepsilon}(t)w(x, t).\] Note that $\varphi$ is in $W^{H}(U_{R, S}(z_0))$ and we need to test the equation from \eqref{weak formulation_1}, that is, 
\begin{align*}
&\mathrm{I}+\mathrm{II}+\mathrm{III}+\mathrm{IV}\\
&=\miint{U_{R, S}(z_0)}(\partial_{t}[u-(u)_{U_{R, S}(z_0)}]^L_{h}-D[u-(u)_{U_{R, S}(z_0)}]^L_h\cdot \zeta)\varphi(x, t)\, dx\, dt\\
&+\miint{U_{R, S}(z_0)}A(z, u, Du)\cdot D[\varphi]^{R}_{h}\,dx\, dt\\
&+\miint{U_{R, S}(z_0)}\frac{\varphi(x-h\zeta, t+h)-\varphi(x+h\zeta, t+h)}{h}u(x, t)\,dx\, dt+\miint{U_{R, S}(z_0)}uD[\varphi]^R_h\cdot \zeta \, dx\, dt=0
\end{align*}
First, translation continuity and $D\varphi\in L^p$ give
\[
\frac{\varphi(x-h\zeta,t+h)-\varphi(x+h\zeta,t+h)}{h}
=-\int_{-1}^{1}\zeta\cdot D\varphi(x+sh\zeta,t+h)\,ds
\longrightarrow-2\zeta\cdot D\varphi
\]
strongly in $L^p$, hence in $L^2$. Also $D[\varphi]^R_h\to D\varphi$ strongly in $L^2$. Pairing with $u\in L^2$, we obtain
\begin{align}\label{eq: 1.16}
    &\lim_{h \to 0}\,\, (\mathrm{III}+\mathrm{IV})\\
    &=\miint{U_{R, S}(z_0)}-u D\varphi(x, t)\cdot \zeta\, dx\, dt\nonumber\\
    &=\miint{U_{R, S}(z_0)}Du \varphi(x, t)\cdot \zeta\, dx\, dt=\miint{U_{R, S}(z_0)}Dw(x, t) \eta^q (x, t) \beta^{\tau_1, \varepsilon}(t)w(x, t)\cdot \zeta\, dx\, dt
\end{align}
\noindent \textbf{Estimate of $\mathrm{I}$:} Recall from Lemma~\ref{lem:steklov lemma}, with $\zeta$ the vector in \eqref{eq:a increasing},
\begin{align}\label{equation L}
\partial_t[w]^L_h&=\tfrac1h\big(w(x,t)-w(x+h\zeta,t-h)\big)+\zeta\cdot D[w]^L_h. \tag{L}
\end{align} Using \eqref{equation L} and convexity, we have the following estimate.
\begin{align}\label{eq: 1.17}
   ( \partial_t[w]^L_h-D[w]^L_h\cdot \zeta)\, \varphi (x, t)&=\frac{1}{h}\eta^q(x, t)\beta^{\tau_1, \varepsilon}(t)w(x, t)\left(w(x, t)-w(x+h\zeta, t-h)\right)\nonumber\\
    &\geq \eta^q(x, t)\beta^{\tau_1, \varepsilon}(t)\frac{1}{2h}\left(w^2(x, t)-w^2(x+h\zeta, t-h)\right)\nonumber\\
    &=\frac{1}{2}\eta^q(x, t)\beta^{\tau_1, \varepsilon}(t)\partial_t[w^2]^L_h(x, t)-\underbrace{\frac{1}{2}\eta^q(x, t)\beta^{\tau_1, \varepsilon}(t)\zeta \cdot D[w^2]^L_h(x, t)}_{\mathrm {J}_h(x, t)}.\nonumber\\
\end{align}
Denoting $\chi:=\eta^q\beta^{\tau_1,\varepsilon}$, spatial integration by parts gives
\[
\miint{U_{R,S}(z_0)}\mathrm J_h\,dz
=-\frac12\miint{U_{R,S}(z_0)}[w^2]^L_h\,\zeta\cdot D\chi\,dz
\longrightarrow-\frac12\miint{U_{R,S}(z_0)}w^2\,\zeta\cdot D\chi\,dz.
\]
Here $[w^2]^L_h\to w^2$ locally in $L^1$ by translation continuity. Since $w,Dw\in L^2$ on the support of $\chi$, another spatial integration by parts yields
\begin{align*}
    \lim_{h\to 0}\miint{U_{R, S}(z_0)} \mathrm{J}_h(x, t)\, dx\, dt=\miint{U_{R, S}(z_0)}Dw(x, t) \eta^q (x, t) \beta^{\tau_1, \varepsilon}(t)w(x, t)\cdot \zeta\, dx\, dt
\end{align*}
and hence
\begin{align*}
    \lim_{h\to 0}\miint{U_{R, S}(z_0)}\mathrm{J}_h(x,t) \, dx\, dt=\lim_{h\to 0}\,\,(\mathrm{III}+\mathrm{IV}).
\end{align*}

We note that the term $-\mathrm J_h$ in \eqref{eq: 1.17} must be retained until this cancellation is taken. For fixed $\varepsilon>0$, integration by parts in time and the preceding limits yield
\[
\liminf_{h\to0}(\mathrm I+\mathrm{III}+\mathrm{IV})
\geq-\frac12\miint{U_{R,S}(z_0)}w^2\,\partial_t\chi\,dz.
\]
The integral with $A$ converges by Lemma~\ref{lem:steklov lemma} and weighted H\"older's inequality. Set
\begin{align}\label{eq : estimate of I in Lemma 3.1}
K_\varepsilon
&:=-\frac12\miint{U_{R,S}(z_0)}w^2\,\partial_t\chi\,dz\nonumber\\
&=-\miint{U_{R,S}(z_0)}\frac q2 w^2\eta^{q-1}\partial_t\eta\,\beta^{\tau_1,\varepsilon}\,dz
-\miint{U_{R,S}(z_0)}\frac12 w^2\eta^q\partial_t\beta^{\tau_1,\varepsilon}\,dz,\\
F_\varepsilon
&:=\lim_{h\to0}\mathrm{II}
=\miint{U_{R,S}(z_0)}A(z,u,Du)\cdot D\varphi\,dz.\nonumber
\end{align}
Thus $K_\varepsilon+F_\varepsilon\leq0$. All estimates below are made after this $h\to0$ limit; we then let $\varepsilon\to0$.

    For the first term on the right-hand side of \eqref{eq : estimate of I in Lemma 3.1}, it follows from the definition of $\eta$ and $\beta^{\tau_1, \varepsilon}$ and the Lemma \ref{lem:steklov lemma} that
    $$
    -\miint{U_{R,S}(z_0)} w^2(x,t)\eta^{q-1}\partial_t\eta(x, t)\beta^{\tau_1, \varepsilon}(t)\;dx\, dt\geq -c \miint{U_{R,S}(z_0)}\frac{|u-(u)_{U_{R,S}(z_0)}|^2}{S-\tau} \; dx\, dt.
    $$
    Next, by the definition of $\beta^{\tau_1, \varepsilon}$ and Lemma \ref{lem:steklov lemma}, we get 
    \begin{align*}
        -\miint{U_{R,S}(z_0)} \frac{1}{2} w^2(x,t)\eta^q (x, t)\partial_t \beta^{\tau_1, \varepsilon} \, dx\, dt &=\frac{1}{|U_{R, S}|}\dashint_{\tau_1-\varepsilon}^{\tau_1}\int_{B_R(x_0)}\frac{1}{2} |u-(u)_{U_{R, S}(z_0)}|^2 \eta^q \,dx \,dt\\
        & \geq \frac{1}{|U_{R, S}|}\dashint_{\tau_1-\varepsilon}^{\tau_1}\int_{B_r(x_0)} \frac{1}{2} |u-(u)_{U_{R, S}(z_0)}|^2\, dx \,dt.
    \end{align*}
    Thus, passing to the limit as $\varepsilon \to 0,$ we obtain
    \begin{equation}\label{eq: final estimate of I in Lemma 3.1}
        \begin{aligned}
             \liminf_{\varepsilon\to0^+}K_\varepsilon &\geq -c\miint{U_{R,S}(z_0)}\frac{|u-(u)_{U_{R,S}(z_0)}|^2}{S-\tau} \; dx\, dt +\frac{1}{2|U_{R,S}|}\int_{B_r(x_0)} |u(x,\tau_1)-(u)_{U_{R, S}(z_0)}|^2\; dx.
        \end{aligned}
    \end{equation}
\noindent \textbf{Estimate of $F_\varepsilon$}: It follows from the definition of $\varphi$ and Lemma \ref{lem:steklov lemma} that
    \begin{align}\nonumber
        F_\varepsilon
        &=\miint{U_{R,S}(z_0)}A(z, u,  Du)\beta^{\tau_1, \varepsilon}(t)\cdot D(w \eta^q (x, t))\, dx\, dt\nonumber \\
        &=\miint{U_{R,S}(z_0)}A(z, u,  Du)\beta^{\tau_1, \varepsilon}(t)\cdot\left(Du \eta^q(x, t)+qwD\eta \eta^{q-1}(x, t)\right)\, dx\, dt.\label{eq : estimate of II in Lemma 3.1}
    \end{align}
For the first term of \eqref{eq : estimate of II in Lemma 3.1}, passing to the limit as $\varepsilon \to 0,$ we deduce from \eqref{eq:growth_cond} that
    \begin{align*}
        &\lim_{\varepsilon \rightarrow 0^+} \miint{U_{R,S}(z_0)} A(z, u, Du)\cdot Du\eta^q \beta^{\tau_1, \varepsilon} \;dx\, dt \geq \frac{\nu}{|U_{R, S}|}\int_{I_S(t_0)\cap(-\infty,\tau_1)}\int_{B_R(x_0)} H(z,|Du|)\eta^q \; dx\, dt.
    \end{align*}
    Next, to estimate the second term in \eqref{eq : estimate of II in Lemma 3.1}, we apply \eqref{eq:growth_cond}, the definition of $\eta$ and Lemma \ref{lem:steklov lemma} to find that 
    \begin{align*}
        &\lim_{\varepsilon \rightarrow 0^+} q\miint{U_{R,S}(z_0)} A(z, u, Du)\cdot w D\eta \eta^{q-1}\beta^{\tau_1, \varepsilon}(t) \; dx\, dt\\
        &\quad \geq -\frac{Lq}{|U_{R, S}|}\int_{I_S(t_0)\cap(-\infty,\tau_1)}\int_{B_R(x_0)} \left(|Du|^{p-1}+a(z)|Du|^{q-1}\right)\eta^{q-1}\frac{|u-(u)_{U_{R, S}(z_0)}|}{R-r}\; dx dt\\
        &\geq -\frac{Lq \upepsilon}{|U_{R, S}|}\int_{I_S(t_0)\cap(-\infty,\tau_1)}\int_{B_R(x_0)}\left(|Du|^p+a(z)|Du|^q\right)\eta^q(x, t)\, dx\, dt \\
        &-c(\upepsilon)\frac{Lq }{|U_{R, S}|}\int_{I_S(t_0)\cap(-\infty,\tau_1)}\int_{B_R(x_0)}\left(\frac{|u-(u)_{U_{R, S}(z_0)}|^p}{(R-r)^p}+a(z)\frac{|u-(u)_{U_{R, S}(z_0)}|^q}{(R-r)^q}\right)\,dx\, dt.
    \end{align*}
    Choosing $\upepsilon=\frac{\nu}{8Lq}$ we get
    \begin{align*}
       &\lim_{\varepsilon \rightarrow 0^+} q\miint{U_{R,S}(z_0)} A(z,u, Du)\cdot w D\eta \eta^{q-1}\beta^{\tau_1, \varepsilon}(t) \; dx\, dt\\
        &\quad \ge -\frac{\nu}{8|U_{R, S}|}\int_{I_S(t_0)\cap(-\infty,\tau_1)} \int_{B_R(x_0)} H(z,|Du|)\eta^q(x, t) \;dx dt -c\miint{U_{R,S}(z_0)} H\left(z,\frac{|u-(u)_{U_{R, S}(z_0)}|}{R-r}\right) \,dx \,dt
    \end{align*}
    for some constant $c>0$ depending on $p,q,\nu$ and $L.$
    It follows that 
    \begin{equation}\label{eq: final estimate of II in Lemma 3.1}
        \begin{aligned}
            \lim_{\varepsilon \rightarrow 0^+} F_\varepsilon &\geq \frac{7\nu}{8|U_{R, S}|} \int_{I_S(t_0)\cap(-\infty,\tau_1)} \int_{B_R(x_0)} H(z,|Du|)\eta^q(x, t) \;dx dt\\
            &\quad -c\miint{U_{R,S}(z_0)} H\left(z,\frac{|u-(u)_{U_{R, S}(z_0)}|}{R-r}\right) \,dx\, dt.
        \end{aligned}
    \end{equation}
    Using $K_\varepsilon+F_\varepsilon\leq0$ and combining \eqref{eq: final estimate of I in Lemma 3.1} and \eqref{eq: final estimate of II in Lemma 3.1} gives
    \begin{align*}
        &\frac{1}{|U_{R, S}|}\int_{B_r(x_0)} |u(x,\tau_1)-(u)_{U_{R, S}(z_0)}|^2\; dx\\
        &\qquad +\frac{1}{|U_{R, S}|} \int_{I_S(t_0)\cap(-\infty,\tau_1)} \int_{B_R(x_0)} H(z,|Du|)\eta^q(x, t) \;dx dt\\
        &\quad \leq c\miint{U_{R,S}(z_0)} H\left(z,\frac{|u-(u)_{U_{R, S}(z_0)}|}{R-r}\right) \,dx\,dt+ c \miint{U_{R,S}(z_0)}\frac{|u-(u)_{U_{R,S}(z_0)}|^2}{S-\tau} \; dx\,dt
    \end{align*}
    where $c$ depends only on $p,q,\nu$ and $L$. As $\tau_1\in I_\tau(t_0)$ is arbitrary, $\frac{R}{2}\leq r <R$ and $\frac{S}{4}\leq \tau < S$, we have
    \begin{align*}
        &\sup_{t\in (t_0-\tau,t_0+\tau)}\dashint_{B_r(x_0)} \frac{|u-(u)_{U_{R, S}(z_0)}|^2}{\tau}\; dx + \miint{U_{r,\tau}(z_0)} H(z,|Du|) \;dx\, dt\\
        &\qquad \leq c\miint{U_{R,S}(z_0)} H\left(z,\frac{|u-(u)_{U_{R, S}(z_0)}|}{R-r}\right)\; dx\, dt+ c \miint{U_{R,S}(z_0)}\frac{|u-(u)_{U_{R,S}(z_0)}|^2}{S-\tau} \; dx\, dt
    \end{align*}
    for some constant $c>0$ depending on $N,p,q,\nu$ and $L$. Also, we note that
    \begin{equation*}
        \sup_{t\in (t_0-\tau,t_0+\tau)}\dashint_{B_r(x_0)} \frac{|u-(u)_{U_{r, \tau}(z_0)}|^2}{\tau}\; dx \leq 2\sup_{t\in (t_0-\tau,t_0+\tau)}\dashint_{B_r(x_0)} \frac{|u-(u)_{U_{R, S}(z_0)}|^2}{\tau}\; dx.
    \end{equation*}
     This completes the proof.
\end{proof}  
\subsection{H\"older regularity for \eqref{eq:main equation} with $a\in \mathcal{Z}^{\kappa}(\Omega_T)$} In this section, we revisit the necessary lemmas from \cite{Kim2025holder} to prove H\"older regularity for \eqref{eq:main equation} as \cite{Kim2025holder} considered the coefficient $a \in C^{\alpha, \frac{\alpha}{2}}_{\mathrm{loc}}(\Omega_T)$ with $\alpha \in (0, 1].$

Throughout this subsection, we assume $2\leq p\leq q\leq p+\kappa$ and $u$ is a bounded weak solution of \eqref{eq:main equation}. We establish the H\"older continuity of $u$ under the coefficient assumptions \eqref{eq:Z_kappa} and \eqref{almost increasing for a}.

\begin{lemma}[Caccioppoli inequality II]\label{Caccio 1}
Let $u$ be a weak solution to \eqref{eq:main equation} and $a$ satisfy \eqref{eq:Z_kappa} and \eqref{almost increasing for a}. Suppose $U_{R, S}(z_0)\subset\Omega_T$ with $R,S>0$ and $S<R^2$. Let $\varphi$ be a Lipschitz continuous function vanishing on $\partial B_R \times (t_0-S,t_0+S)$ with $0\le \varphi\le 1$. 
Then for $k \in \RR$ with $|k| \leq \|u\|_\infty$, there exists a constant $c(\texttt{data})$ such that
	\begin{align*}
		\begin{split}
			& \sup_{t\in (t_0-S,t_0+S)}\int_{B_{R}(x_0)}  (u-k)_{\pm}^2\varphi^q\, dx 
            + \iint_{U_{R,S}(z_0)} (|D ((u-k)_\pm\varphi^\frac{q}{p})|^p+a(z)|D ((u-k)_\pm\varphi)|^q)\, dx dt  \\
			&\qquad\le c\iint_{U_{R,S}(z_0)} ((u-k)_\pm^p (|D \varphi|^p+R^\kappa |D \varphi|^{q} ) +a(z_0)(u-k)_\pm^q |D \varphi|^q )\, dx dt\\
			&\qquad+c\iint_{U_{R,S}(z_0)} (u-k)_\pm^2 |\partial_t \varphi^q| \, dx dt + c\int_{B_R(x_0) \times \{t_0-S\}}  (u-k)_\pm^2\varphi^q \, dx.
		\end{split}
	\end{align*}
\end{lemma}
\begin{proof}
The proof of this lemma follows similar lines to the energy estimate above combined with the property \eqref{eq:Z_kappa}. We define $\tau_{t_1, \delta} \in C^{0, 1}(\RR)$ for $t_0-S<t_1<t_0+S$ and $\delta \in (0, S)$ by 
\begin{align*}
    \tau_{t_1, \delta}=\begin{cases}
        1\quad &\text{for} \quad t \leq t_1-\delta\\
        \frac{t_1-t}{\delta}\quad &\text{for} \quad t_1-\delta<t\leq t_1\\
        0\quad &\text{for} \quad t \geq t_1.
    \end{cases}
\end{align*}
and choose the admissible test function as 
\begin{align*}
    \xi(x, t)=\varphi(x, t)^q\tau_{t_1, \delta}(t)[u-k]_{\pm}(x, t).
\end{align*}
Now the time term can be handled as previously, and following the proof of \cite[Lemma 3.1]{Kim2025holder}, we can complete the proof.
\end{proof}

Next, we prove \cite[Lemma 4.1 ]{Kim2025holder} under the assumption that $a\in \mathcal{Z}^{\kappa}(\Omega_T)$. The proof is similar, but we include it for the sake of completeness.
 \begin{lemma}
    Let $K_1=\max \{1, 2^{\kappa+1}c_a\|u\|_\infty^{q-p}\}$. If $K_1\bigl( \tfrac{\omega}{\varrho}\bigr)^p<a(z_0)\bigl( \tfrac{\omega}{\varrho}\bigr)^q$ and $q\leq p+\kappa$ holds, then we have
    $$
        \frac{a(z_0)}{2c_a}\le a(z)\le 2c_a a(z_0)\quad\text{for all}\quad z\in Q_{\varrho}(z_0).
    $$
\end{lemma}
\begin{proof}
First, we claim that 
\begin{align}\label{claim_infa(z)}
\inf_{Q_{\varrho}(z_0)}a(z)\geq \varrho^{\kappa}\quad  \text{for } \kappa \in (0, \infty).
\end{align} 
Suppose not, then it holds $\inf_{Q_{\varrho}(z_0)}a(z)< \varrho^{\kappa}$ and using the assumption $K_1\left(\frac{\omega}{\varrho}\right)^p\leq a(z_0)\left(\frac{\omega}{\varrho}\right)^q,$ we get
\begin{align*}
    K_1\left(\frac{\omega}{\varrho}\right)^p\leq a(z_0)\left(\frac{\omega}{\varrho}\right)^q\leq c_a\left(\inf_{Q_{\varrho}(z_0)}a(z)+\varrho^{\kappa}\right)\left(\frac{\omega}{\varrho}\right)^q< 2c_a \varrho^{\kappa}\left(\frac{\omega}{\varrho}\right)^{q}.
\end{align*}
Thus, for $\varrho<1,$ we have
\begin{align*}
    2^{\kappa+1}c_a||u||^{q-p}_{\infty} \leq K_1< 2c_{a}\omega^{q-p}\varrho^{\kappa+p-q}\leq 2^{q-p+1}c_{a}||u||^{q-p}_{\infty}
\end{align*}
which is a contradiction. Therefore, using \eqref{claim_infa(z)}, we obtain
\begin{align*}
    a(z)\leq c_a\left(\inf_{Q_{\varrho}(z_0)}a(z)+\varrho^{\kappa}\right)\leq 2c_a \inf_{Q_{\varrho}(z_0)}a(z)\leq 2c_a a(z_0),
\end{align*}
which completes the proof.
\end{proof} 
Now we can complete the proof of H\"older regularity of \eqref{eq:main equation} as in \cite{Kim2025holder} and below we state the H\"older regularity result of this subsection.
We denote $\|a\|_{\infty}=\|a \|_{L^\infty(\Omega_T)},$ $\|u\|_{\infty}=\|u \|_{L^\infty(\Omega_T)}$ and the parabolic boundary $ \Gamma  = (\Omega \times \{0\}) \cup (\partial \Omega \times [0,T)).$ 

\begin{theorem}\label{holder cts thm}
Let $u$ be a bounded weak solution to \eqref{eq:main equation} according to Definition~\ref{def:weak_solution} such that~\eqref{eq:growth_cond} is in force. Moreover, the coefficient $a$ satisfies \eqref{eq:Z_kappa} and \eqref{almost increasing for a}. Then $u$ is locally H\"older continuous in $\Omega_T$. Moreover, there exist $c_H > 0$ and $\upkappa \in(0,1)$ depending only on $N,p,q,\kappa, \nu, L, c_a, ||a||_{\infty}$, $\dist(\mathcal C, \Gamma)$  and $\|u\|_\infty$ such that for any compact set $\mathcal{C} \subset \Omega_T$
     $$
     |u(x_1,t_1)-u(x_2,t_2)|\le c_H\left(|x_1-x_2|+|t_1-t_2|^\frac{1}{q} \right)^\upkappa
     $$
     holds for every pair of points $(x_1,t_1),(x_2,t_2)\in \mathcal{C}$. Here $\dist$ denotes the usual distance in $\mathbb R^{N+1}$.
\end{theorem}

\subsection{Further auxiliary lemmas} Here we collect some more auxiliary lemmas that will be used throughout the paper. The following lemma presents a Gagliardo–Nirenberg inequality, which will play a key role in establishing reverse H\"{o}lder's inequalities.
\begin{lemma}[\cite{Hasto_2021}, Lemma 2.12]\label{lem : lemma 2.1}
For an open ball $B\subset \RR^N$, let $\psi \in W^{1, s}(B)$ and take $\sigma, s, r \in[1, \infty)$ and $\vartheta \in(0,1)$ such that
    $$
    -\frac{N}{\sigma} \leq \vartheta\left(1-\frac{N}{s}\right)-(1-\vartheta) \frac{N}{r}.
    $$
    Then we have
    $$
    \dashint_{B} \frac{|\psi|^\sigma}{\rho^\sigma} \, d x \leq c\left(\dashint_{B}\left(\frac{|\psi|^s}{\rho^s}+|D\psi|^s\right) d x\right)^{\frac{\vartheta \sigma}{s}}\left(\dashint_{B} \frac{|\psi|^r}{\rho^r} d x\right)^{\frac{(1-\vartheta) \sigma}{r}}
    $$
    for some constant $c=c(N, \sigma)$.
\end{lemma}
We next record a standard iteration lemma from \cite[Lemma 1.1]{GG82}:
\begin{lemma}
	\label{iter_lemma}
	Let $0< r< R<\infty$ be given and $h : [r, R] \to \RR$ be a non-negative bounded function. Furthermore, let $\vartheta \in (0,1)$ and $A,B,\beta, \geq 0$ be fixed constants and 
	suppose that
	$$
	h(\varrho_1) \leq \vartheta h(\varrho_2) + \frac{A}{(\varrho_2-\varrho_1)^{\beta}} + B,
	$$
	holds for all $r \leq \varrho_1 < \varrho_2 \leq R$.
    Then the following conclusion holds:
	$$
	h(r) \lesssim_{(\vartheta,\beta)} \frac{A}{(R-r)^{\beta}} + B.
	$$
\end{lemma} 

Next, we recall the gluing lemma \cite[Lemma 2.4]{2023_Gradient_Higher_Integrability_for_Degenerate_Parabolic_Double-Phase_Systems} and the parabolic  Poincar\'{e} inequality \cite[\text{Lemma}~2.5]{2023_Gradient_Higher_Integrability_for_Degenerate_Parabolic_Double-Phase_Systems}.
\begin{lemma}    \label{lem: gluing lemma}
Let $u$ be a weak solution to \eqref{eq:main equation}, and let $\eta\in C_0^\infty(B_R(x_0))$ be a function such that
\begin{equation*}    \label{eq: definition of eta in Lemma 3.2}
\eta\geq 0,\quad \dashint_{B_R(x_0)} \eta \; dx =1 \quad \text{and} \quad \|\eta\|_{L^\infty}+R\|D\eta\|_{L^\infty} \leq c(N).
\end{equation*}
Then for $U_{R,S}(z_0)\subset \Omega_T$, there exists a constant $c=c(N, L)$ such that 
\begin{align*}
\sup_{t_1,t_2\in(t_0-S,t_0+S)}|(u\eta)_{x_0;R}(t_2)-(u\eta)_{x_0;R}(t_1)|
 \leq c\frac{S}{R}\miint{U_{R,S}(z_0)}\widetilde{H}(z,|Du|)\; dz,
\end{align*}
 where $\widetilde{H}:\Omega_T\times \mr^+ \rightarrow \mr^+$ is denoted by $$\widetilde{H}(z,\zeta)=\zeta^{p-1}+a(z)\zeta^{q-1}$$ for $z \in \Omega_T$ and $\zeta\in\mr^+$.
\end{lemma}
\begin{lemma}\label{lem : parabolic poincare inequality in Section 3}
Let $u$ be a weak solution to \eqref{eq:main equation}. Then for $U_{R,S}(z_0)\subset \Omega_T$, $m\in(1,q]$ and $\theta\in(1/m,1]$, there exists a constant $c$ depending on $N$, $m$ and $L$ such that 
    \begin{align}\label{eq : parabolic poincare inequality in Section 3_1}
        \miint{U_{R,S}(z_0)} &\frac{|u-(u)_{U_{R,S}(z_0)}|^{\theta m}}{R^{\theta m}}\, dz \leq c\miint{U_{R,S}(z_0)} |Du|^{\theta m} \,dz +c\left(\frac{S}{R^2}\miint{U_{R,S}(z_0)} \widetilde{H}(z,|Du|)\, dz \right)^{\theta m}.
    \end{align}
\end{lemma}
\section{Reverse H\"{o}lder inequalities}
In this section, we prove the reverse H\"{o}lder inequalities for the $p$- and $(p, q)$-phase. We set
\begin{align}\label{definition of M-tilde}
    \widetilde M\geq 2c_a 4^{\kappa}\geq 1
\end{align}
and 

\begin{align}\label{defn of K}
    K:=c_a\widetilde M
    +2\Bigg[8^{q+1}\max\Bigg\{&
    c_E[u]^p_{\gamma,\gamma/q},
    2
    c_a\widetilde M c_E[u]^q_{\gamma,\gamma/q},\nonumber \\
    &(2c_E)^{p/2}
    [u]^p_{\gamma,\gamma/q}
    \Bigg\}\Bigg]^{\frac{q\kappa}{p(q-2\gamma)}}
    +2^{\kappa+1}c_a\|u\|_\infty^{q-p}.
\end{align}
\subsection{Reverse H\"{o}lder inequality for $p$-phase}  We assume the following $p$-intrinsic conditions:
\begin{equation}\label{eq: p-phase condition}
    \left\{
    \begin{aligned}
        &\quad \inf_{Q_{\rho, \lambda}(z_0)}a(z)\leq \widetilde M\rho^{\kappa}\\
        &\miint{Q_{\tau,\lambda}(z_0)} H(z,|Du|) \; dz <\lambda^p \ \, \text{ for every }\tau\in(\rho,c_v(2\rho)],\\
        &\miint{Q_{\rho,\lambda}(z_0)} H(z,|Du|) \; dz =\lambda^p.
    \end{aligned}
    \right.
\end{equation}
\begin{remark}
These $p$-intrinsic assumptions differ slightly from those in \cite{2023_Gradient_Higher_Integrability_for_Degenerate_Parabolic_Double-Phase_Systems} and appear more natural within our framework. In what follows, we assume $\rho<1.$
\end{remark}
We begin with a result which relates the scaling parameter and the radius of the cylinder. 
\begin{lemma} \label{lem: lambda rho relation}
Let $u$ be a weak solution to \eqref{eq:main equation}. If $u$ is locally $\gamma$-H\"{o}lder continuous in space and $\gamma/q$-H\"older continuous in time with $\gamma \in [0, 1),$ and $q\leq p+q\kappa/(q-2\gamma),$ then we have
\begin{align}\label{relation lambda rho}
    \lambda \leq \Bigg[8^{q+1}\max\Bigg\{&
    c_E[u]^p_{\gamma,\gamma/q},
    2
    c_a\widetilde M c_E[u]^q_{\gamma,\gamma/q},
    (2c_E)^{p/2}
    [u]^p_{\gamma,\gamma/q}
    \Bigg\}\Bigg]^{\frac{1}{p}}
    \rho^{\frac{2\gamma-q}{q}}.
\end{align}
where the constant $c_E$ is from Lemma \ref{lem: energy inequality} and depends only on $N, p, q, \nu, L.$
\end{lemma}
\begin{proof}
Using the Caccioppoli inequality (Lemma \ref{lem: energy inequality}) with $R=2\rho, S=4\lambda^{2-p}\rho^2, r=\rho,$ and $\tau=\lambda^{2-p}\rho^2,$ and the stopping time from \eqref{eq: p-phase condition}, we get
\begin{align*}
    \lambda^p
    &=\miint{Q_{\rho,\lambda}(z_0)} H(z,|Du|)\; dz\\
    &\leq c_E\miint{Q_{2\rho,\lambda}(z_0)}
    H\left(z,\frac{|u-(u)_{Q_{2\rho,\lambda}(z_0)}|}{\rho}\right) dz
    +c_E\lambda^{p-2}\miint{Q_{2\rho,\lambda}(z_0)}
    \frac{|u-(u)_{Q_{2\rho,\lambda}(z_0)}|^2}{\rho^2}\;dz.
\end{align*}
First, we note that
\begin{align*}
    \miint{Q_{2\rho,\lambda}(z_0)}
    \left|\frac{u-(u)_{Q_{2\rho,\lambda}(z_0)}}{\rho}\right|^p\,dz
    &\leq \rho^{-p}
    \left(\osc_{Q_{2\rho,\lambda}(z_0)}u\right)^p\\
    &\leq 8^p[u]^p_{\gamma,\frac{\gamma}{q}}
    \rho^{-p}\rho^{\frac{2p\gamma}{q}}
    =8^p[u]^p_{\gamma,\frac{\gamma}{q}}
    \rho^{-p\left(1-\frac{2\gamma}{q}\right)}.
\end{align*}
Note that in the last line above we used the boundedness of $u$ for $\gamma=0,$ and for $\gamma>0,$ the additional H\"older regularity assumption is used. Indeed, since $\rho<1,$ $\lambda\geq1,$ and $2\leq p\leq q,$ we have
\begin{align*}
    \osc_{Q_{2\rho,\lambda}(z_0)}u
    \leq [u]_{\gamma,\gamma/q}
    \left((4\rho)^\gamma+
    (8\lambda^{2-p}\rho^2)^{\gamma/q}\right)
    \leq 8[u]_{\gamma,\gamma/q}\rho^{2\gamma/q}
\end{align*}
for $\gamma>0,$ while for $\gamma=0$ we use
$\osc_{Q_{2\rho,\lambda}(z_0)}u\leq2\|u\|_\infty.$
Next, we note that, since \eqref{eq:Z_kappa} holds for all $z_1,z_2\in\Omega_T,$ for any $z\in Q_{\rho,\lambda}(z_0)$ we have the following.
\begin{align}\label{eq: a inequality}
    a(z)\leq c_a\left(
    \inf_{Q_{\rho,\lambda}(z_0)}a(z)
    +(2\rho)^\kappa\right).
\end{align}
Using \eqref{eq: a inequality}, $p$-intrinsic conditions, and
$\widetilde M\geq2c_a4^\kappa,$ we get
\begin{align*}
    \miint{Q_{2\rho,\lambda}(z_0)}
    a(z)\left|\frac{u-(u)_{Q_{2\rho,\lambda}(z_0)}}{\rho}\right|^q\,dz
    &\leq 4^\kappa c_a
    \miint{Q_{2\rho,\lambda}(z_0)}
    \rho^\kappa
    \left|\frac{u-(u)_{Q_{2\rho,\lambda}(z_0)}}{\rho}\right|^q\,dz\\
    &\quad+c_a\miint{Q_{2\rho,\lambda}(z_0)}
    \inf_{Q_{2\rho,\lambda}(z_0)}a(z)
    \left|\frac{u-(u)_{Q_{2\rho,\lambda}(z_0)}}{\rho}\right|^q\,dz\\
    &\leq 2 c_a\widetilde M\rho^\kappa
    \miint{Q_{2\rho,\lambda}(z_0)}
    \left|\frac{u-(u)_{Q_{2\rho,\lambda}(z_0)}}{\rho}\right|^q\,dz\\
    &\leq 2
    \widetilde M c_a\rho^{\kappa-q}
    \left(\osc_{Q_{2\rho,\lambda}(z_0)}u\right)^q\\
    &\leq 2\cdot8^q
    c_a\widetilde M[u]^q_{\gamma,\frac{\gamma}{q}}
    \rho^{\kappa-q+2\gamma}\\
    &\leq 2\cdot8^q
    c_a\widetilde M[u]^q_{\gamma,\frac{\gamma}{q}}
    \rho^{-p\left(1-\frac{2\gamma}{q}\right)}.
\end{align*}
Note that in the last line of the above inequality we used
$q\leq p+q\kappa/(q-2\gamma).$ Similarly, using Young's inequality, we estimate
\begin{align*}
    \lambda^{p-2}\miint{Q_{2\rho,\lambda}(z_0)}
    \frac{|u-(u)_{Q_{2\rho,\lambda}(z_0)}|^2}{\rho^2}\;dz
    &\leq \lambda^{p-2}
    \left(\osc_{Q_{2\rho,\lambda}(z_0)}u\right)^2\rho^{-2}\\
    &\leq 8^2
    [u]^2_{\gamma,\frac{\gamma}{q}}
    \lambda^{p-2}\rho^{2\left(\frac{2\gamma}{q}-1\right)}\\
    &\leq \frac{p-2}{2pc_E}\lambda^p
    +\frac{2\cdot8^p}{p}(2c_E)^{\frac{p-2}{2}}
    [u]^p_{\gamma,\frac{\gamma}{q}}
    \rho^{-p\left(1-\frac{2\gamma}{q}\right)}\\
    &=\frac{p-2}{2pc_E}\lambda^p
    +\frac{8^p(2c_E)^{p/2}}{pc_E}
    [u]^p_{\gamma,\frac{\gamma}{q}}
    \rho^{-p\left(1-\frac{2\gamma}{q}\right)}.
\end{align*}
For $p=2,$ the last inequality holds directly.
Absorbing $\frac{p-2}{2p}\la^p$ on the left hand side and combining all the above estimates, and finally taking the maximum of the common constants, we obtain \eqref{relation lambda rho}.
\end{proof}
\begin{remark}\label{REM3.3}
The first condition of \eqref{eq: p-phase condition}, together with the above lemma, yields
$$K^2\lambda^p\geq \inf_{Q_{\rho, \lambda}(z_0)}a(z)\lambda^q$$
(see Section~\ref{subsec: stopping time} for the proof). Consequently, once the Lemma~\ref{lem: lambda rho relation} has been established, the first condition of \eqref{eq: p-phase condition} may be replaced by the above inequality, and the two formulations may be used interchangeably in what follows.
\end{remark}
Next, we prove parabolic Sobolev-Poincar\'e inequalities on the $p$-intrinsic cylinders.
\begin{lemma}\label{p_intrinsic poincare 2}
Let $u$ be a weak solution to \eqref{eq:main equation} and let the assumption \ref{eq: p-phase condition} be in force. Then for any $Q_{4\rho, \lambda}(z_0)\subset \Omega_T,$ $\tau \in [2\rho, 4\rho],$ and $\theta\in \left(\frac{q-1}{q}, 1\right],$ there exists a constant $c=c(\textnormal{\texttt{data}})>K$ such that 
 \begin{align}\label{1_lem7.4} 
\miint{Q_{\tau, \lambda}(z_0)}\left|\frac{u-(u)_{Q_{\tau, \lambda}(z_0)}}{\tau}\right|^{\theta p} dz \leq c \miint{Q_{\tau, \lambda}(z_0)}H(z, |Du|)^{\theta}\, dz.
\end{align}
\end{lemma}
\begin{proof} We use Lemma \ref{lem : parabolic poincare inequality in Section 3} with $m=p$ and $U_{R, S}(z_0)=Q_{\tau, \lambda}(z_0),$ and get
 \begin{align}\label{eq: 2.5}
     \miint{Q_{\tau, \lambda}(z_0)}\frac{\left|u-(u)_{Q_{\tau, \lambda}(z_0)}\right|^{\theta p}}{\tau^{\theta p}}\, dz &\leq c\miint{Q_{\tau, \lambda}(z_0)}|Du|^{\theta p}\, dz\nonumber\\
     &+ c \left[\lambda^{2-p}\miint{Q_{\tau, \lambda}(z_0)}|Du|^{p-1}+a(z)|Du|^{q-1}\, dz\right]^{\theta p}
 \end{align}
 For the first term, using \eqref{eq: a inequality} we estimate
 \begin{align*}
& \miint{Q_{\tau, \lambda}(z_0)}|Du|^{p-1}+a(z)|Du|^{q-1}\, dz\\
 &\leq   \miint{Q_{\tau, \lambda}(z_0)}|Du|^{p-1}\, dz+c_a \miint{Q_{\tau, \lambda}(z_0)}\inf_{Q_{\tau, \lambda}(z_0)} a(z)^{\frac{1}{q}}a(z)^{\frac{q-1}{q}}|Du|^{q-1}\, dz \\
 &+c_a \tau^{\frac{\kappa}{q}}\miint{Q_{\tau, \lambda}(z_0)}a(z)^{\frac{q-1}{q}}|Du|^{q-1}\, dz
 \end{align*}
  Now we use Lemma \ref{lem: lambda rho relation} and the gap bound $q-p\le \frac{q\kappa}{q-2\gamma}$ to get
  \begin{align}\label{eq: lambda_rho_relation_2}
   \tau^{\kappa}\lesssim \frac{K}{2}\lambda ^{-\frac{q\kappa}{q-2\gamma}}\lesssim\frac{K}{2}\lambda^{p-q}   
  \end{align}
  and using this relation and $p$-intrinsic assumption, we derive
 \begin{align*}
\miint{Q_{\tau, \lambda}(z_0)}|Du|^{p-1}+a(z)|Du|^{q-1}\, dz \leq c \miint{Q_{\tau, \lambda}(z_0)}|Du|^{p-1}\, dz+ c \lambda^{\frac{p-q}{q}}\miint{Q_{\tau, \lambda}(z_0)} a(z)^{\frac{q-1}{q}} |Du|^{q-1}\, dz.    
 \end{align*}
 Next we use H\"{o}lder's inequality and consider the case $p\neq q.$ Also, we use the range $q-1<\theta q.$
 \begin{align*}
  &\miint{Q_{\tau, \lambda}(z_0)}|Du|^{p-1}+a(z)|Du|^{q-1}\, dz\\
  &\leq c \miint{Q_{\tau, \lambda}(z_0)}|Du|^{p-1}\, dz+ c \lambda^{\frac{p-q}{q}}\left(\miint{Q_{\tau, \lambda}(z_0)} a(z)^{\theta} |Du|^{\theta q}\, dz\right)^{\frac{q-1}{\theta q}}\\
  &\leq c \miint{Q_{\tau, \lambda}(z_0)}|Du|^{p-1}\, dz+ c\lambda^{\frac{p-q}{q}}\left(\miint{Q_{\tau, \lambda}(z_0)}a(z)^{\theta}|Du|^{\theta q}\, dz\right)^{\frac{\eta}{\theta q}}   \left(\miint{Q_{\tau, \lambda}(z_0)}a(z)^{\theta}|Du|^{\theta q}\, dz\right)^{\frac{q-1-\eta}{\theta q}}\\
  &\leq c \miint{Q_{\tau, \lambda}(z_0)}|Du|^{p-1}\, dz+c\lambda^{\frac{p-q}{q}}\left(\miint{Q_{\tau, \lambda}(z_0)}a(z)|Du|^{q}\, dz\right)^{\frac{\eta}{q}}\left(\miint{Q_{\tau, \lambda}(z_0)}a(z)^{\theta}|Du|^{\theta q}\, dz\right)^{\frac{q-1-\eta}{\theta q}}\\
&\leq c \miint{Q_{\tau, \lambda}(z_0)}|Du|^{p-1}\, dz+ c\lambda^{\frac{p-q}{q}}\lambda^{\frac{p\eta}{q}}\left(\miint{Q_{\tau, \lambda}(z_0)}a(z)^{\theta}|Du|^{\theta q}\, dz\right)^{\frac{q-1-\eta}{\theta q}}.
 \end{align*}
 Now we choose $\eta=\frac{q-p}{p}>0$ to conclude
 \begin{align*}
 \miint{Q_{\tau, \lambda}(z_0)}|Du|^{p-1}+a(z)|Du|^{q-1}\, dz &\leq c \left(\miint{Q_{\tau, \lambda}(z_0)}|Du|^{\theta p}\, dz\right)^{\frac{p-1}{\theta p}}+ c\left(\miint{Q_{\tau, \lambda}(z_0)}a(z)^{\theta}|Du|^{\theta q}\, dz\right)^{\frac{p-1}{\theta p}}\\
 &\leq c \left(\miint{Q_{\tau, \lambda}(z_0)} H(z, |Du|)^{\theta}\, dz\right)^{\frac{p-1}{\theta p}}.   
 \end{align*}
 Note that the above conclusion holds for $p=q$ by using H\"{o}lder's inequality in the first step. Substituting these estimates into \eqref{eq: 2.5}, we get
 \begin{align*}  
 \miint{Q_{\tau, \lambda}(z_0)}\frac{\left|u-(u)_{Q_{\tau, \lambda}(z_0)}\right|^{\theta p}}{\tau^{\theta p}}\, dz \leq c \miint{Q_{\tau, \lambda}(z_0)}|Du|^{\theta p}\, dz + c \lambda^{(2-p)\theta p}\left(\miint{Q_{\tau, \lambda}(z_0)}H(z, |Du|)^{\theta
 }\, dz\right)^{p-1}.
 \end{align*}
 In addition, we use $p$-intrinsic assumptions and $p\geq 2$ to obtain the following.
 \begin{align*}
   &\lambda^{(2-p)\theta p}\left(\miint{Q_{\tau, \lambda}(z_0)}H(z, |Du|)^{\theta}\, dz\right)^{p-1}\\
   &= \lambda^{(2-p)\theta p}\left(\miint{Q_{\tau, \lambda}(z_0)}H(z, |Du|)^{\theta}\, dz\right)^{p-2}\left(\miint{Q_{\tau, \lambda}(z_0)}H(z, |Du|)^{\theta}\, dz\right)\\
   &\leq c \lambda^{(2-p)\theta p}\lambda^{(p-2)\theta p}\left(\miint{Q_{\tau, \lambda}(z_0)}H(z,|Du|)^{\theta}\, dz\right)\\
   &= \left(\miint{Q_{\tau, \lambda}(z_0)}H(z, |Du|)^{\theta}\, dz\right).
 \end{align*}
 Combining these estimates, we finally arrive at \eqref{1_lem7.4}. 
\end{proof}

\begin{lemma}\label{LEMMA7.5}
Let $u$ be a weak solution to \eqref{eq:main equation} with the assumption \ref{eq: p-phase condition} in force. Then there exists a constant $c=c(\textnormal{\texttt{data}})$ such that the following estimate,
\begin{align*}
\lambda^{p-2} \sup_{I_{2\rho}(t_0)}\miint{B_{2\rho}(x_0)}\left|\frac{u-(u)_{Q_{\rho, \lambda}(z_0)}}{2\rho}\right|^2 dx \leq c \lambda^p
\end{align*}
holds whenever $Q_{c_v(2\rho), \lambda}(z_0)\subset \Omega_T.$
\end{lemma}
\begin{proof}   
For any $2\varrho \leq \rho_1< \rho_2 \leq 4\varrho,$ from the Caccioppoli inequality (Lemma \ref{lem: energy inequality}) we get,
\begin{align*}
    &\la^{p-2}\sup_{I_{\rho_1, \lambda}(t_0)}\dashint_{B_{\rho_1}(x_0)}\left|\frac{u-(u)_{Q_{\rho_1, \lambda}(z_0)}}{\rho_1}\right|^2 dx\nonumber\\
    &\leq c \left(\frac{\rho_2}{\rho_2-\rho_1}\right)^q\miint{Q_{\rho_2, \lambda}(z_0)} \left(\left|\frac{u-(u)_{Q_{\rho_2, \lambda}(z_0)}}{\rho_2}\right|^p + a(z)\left|\frac{u-(u)_{Q_{\rho_2, \lambda}(z_0)}}{\rho_2}\right|^q\right)\, dz\nonumber\\
    &+ c\la^{p-2}\left(\frac{\rho_2}{\rho_2-\rho_1}\right)^2\miint{Q_{\rho_2, \lambda}(z_0)}\left|\frac{u-(u)_{Q_{\rho_2, \lambda}(z_0)}}{\rho_2}\right|^2 \,dz=\mathrm{I}+\mathrm{II}. 
\end{align*}
\textbf{Estimate of $\mathrm{I}:$} We estimate the first term of $\mathrm{I}$ by using Lemma \ref{p_intrinsic poincare 2} with $\theta =1$ and $p$-intrinsic assumption.
\begin{align*}
    \left(\frac{\rho_2}{\rho_2-\rho_1}\right)^q\miint{Q_{\rho_2, \lambda}(z_0)}\left|\frac{u-(u)_{Q_{\rho_2, \lambda}(z_0)}}{\rho_2}\right|^p\, dz \leq c \left(\frac{\rho_2}{\rho_2-\rho_1}\right)^q \lambda^p.
\end{align*}
Next we estimate the second term of $\mathrm{I}.$ We use $a \in \mathcal{Z}^{\kappa}(\Omega_T)$ and $p$-intrinsic assumption.
\begin{align*}
&\left(\frac{\rho_2}{\rho_2-\rho_1}\right)^q \miint{Q_{\rho_2, \lambda}(z_0)} a(z)\left|\frac{u-(u)_{Q_{\rho_2, \lambda}(z_0)}}{\rho_2}\right|^q\, dz\\
&\leq c\left(\frac{\rho_2}{\rho_2-\rho_1}\right)^q \miint{Q_{\rho_2, \lambda}(z_0)} \inf_{Q_{\rho_2, \lambda}(z_0)} a(z) \left|\frac{u-(u)_{Q_{\rho_2, \lambda}(z_0)}}{\rho_2}\right|^q\, dz\\
&+c \rho_2^{\kappa} \left(\frac{\rho_2}{\rho_2-\rho_1}\right)^q \miint{Q_{\rho_2, \lambda}(z_0)}\left|\frac{u-(u)_{Q_{\rho_2, \lambda}(z_0)}}{\rho_2}\right|^q\, dz\\
&\leq c \rho_2^{\kappa}\left(\frac{\rho_2}{\rho_2-\rho_1}\right)^q \miint{Q_{\rho_2, \lambda}(z_0)}\left|\frac{u-(u)_{Q_{\rho_2, \lambda}(z_0)}}{\rho_2}\right|^q\, dz\\
&\leq c \rho^{\kappa}_2\left(\frac{\rho_2}{\rho_2-\rho_1}\right)^q \miint{Q_{\rho_2, \lambda}(z_0)}\left|\frac{u-(u)_{Q_{\rho_2, \lambda}(z_0)}}{\rho_2}\right|^{q-p}\left|\frac{u-(u)_{Q_{\rho_2, \lambda}(z_0)}}{\rho_2}\right|^p\, dz\\
&\leq c \rho^{\kappa}_2 \rho^{(\frac{2\gamma}{q}-1)(q-p)}_2\left(\frac{\rho_2}{\rho_2-\rho_1}\right)^q\miint{Q_{\rho_2, \lambda}(z_0)}\left|\frac{u-(u)_{Q_{\rho_2, \lambda}(z_0)}}{\rho_2}\right|^p\, dz\leq c \left(\frac{\rho_2}{\rho_2-\rho_1}\right)^q \lambda^p.
\end{align*}
Combining these two estimates, we get
\begin{align*}
    \mathrm{I}\leq c \left(\frac{\rho_2}{\rho_2-\rho_1}\right)^q \lambda^p.
\end{align*}
\textbf{Estimate of $\mathrm{II}:$} Now we estimate $\mathrm{II}$ by using Lemma \ref{lem : lemma 2.1} with $s=p, \sigma=2, \vartheta=\frac{1}{2}$ and $r=2.$ We note that
\[-\frac{N}{2}\leq \frac{1}{2}\left(1-\frac{N}{p}\right)-\frac{N}{4}\quad \text{ if}\quad p\geq 2.\] Thus we have
\begin{align*}
\mathrm{II} &\leq c\la^{p-2}\left(\frac{\rho_2}{\rho_2-\rho_1}\right)^2\dashint_{I_{\rho_2, \lambda}(t_0)}\left(\dashint_{B_{\rho_2}(x_0)}\left(\left|\frac{u-(u)_{Q_{\rho_2, \lambda}}}{\rho_2}\right|^p+|D u|^p\right) dx\right)^{\frac{1}{p}} \\
&\times\left(\dashint_{B^{\la}_{\rho_2}(x_0)}\left|\frac{u-(u)_{Q_{\rho_2, \lambda}}}{\rho_2}\right|^2 dx\right)^{\frac{1}{2}}dt\\
&\leq c\la^{p-2}\left(\frac{\rho_2}{\rho_2-\rho_1}\right)^2\left(\miint{Q_{\rho_2, \lambda}(z_0)}\left|\frac{u-(u)_{Q_{\rho_2, \lambda}}}{\rho_2}\right|^p dz+ \miint{Q_{\rho_2, \lambda}(z_0)}|D u|^p dz\right)^{\frac{1}{p}}\\
&\times\left(\sup_{I_{\rho_2, \lambda}(t_0)}\dashint_{B_{\rho_2}(x_0)}\left|\frac{u-(u)_{Q_{\rho_2, \lambda}}}{\rho_2}\right|^2 dx\right)^{\frac{1}{2}}.
\end{align*}
Using Lemma \ref{p_intrinsic poincare 2} and the second assumption of \eqref{eq: p-phase condition}, we get
\begin{align*}
    \mathrm{II} \leq c \lambda^{p-2}\left(\frac{\rho_2}{\rho_2-\rho_1}\right)^2 \lambda \left(\sup_{I_{\rho_2, \lambda}(t_0)}\dashint_{B_{\rho_2}(x_0)}\left|\frac{u-(u)_{Q_{\rho_2, \lambda}}}{\rho_2}\right|^2 dx\right)^{\frac{1}{2}}.
\end{align*}
Combining the estimates of $\mathrm{I}$ and $\mathrm{II}, $ we obtain
\begin{align*}
 &\sup_{I_{\rho_1, \lambda}(t_0)}\dashint_{B_{\rho_1}(x_0)}\left|\frac{u-(u)_{Q_{\rho_1, \lambda}(z_0)}}{\rho_1}\right|^2 dx\\ 
 &\leq c \left(\frac{\rho_2}{\rho_2-\rho_1}\right)^q \lambda^2
 +c \left(\frac{\rho_2}{\rho_2-\rho_1}\right)^2 \lambda \left(\sup_{I_{\rho_2, \lambda}(t_0)}\dashint_{B_{\rho_2}(x_0)}\left|\frac{u-(u)_{Q_{\rho_2, \lambda}}}{\rho_2}\right|^2 dx\right)^{\frac{1}{2}}.   
\end{align*}
Applying Young's inequality, we get
\begin{align*}
 \sup_{I_{\rho_1, \lambda}(t_0)}\dashint_{B_{\rho_1}(x_0)}\left|\frac{u-(u)_{Q_{\rho_1, \lambda}(z_0)}}{\rho_1}\right|^2 dx &\leq \frac{1}{2}\left(\sup_{I_{\rho_2, \lambda}(t_0)}\dashint_{B_{\rho_2}(x_0)}\left|\frac{u-(u)_{Q_{\rho_2, \lambda}}}{\rho_2}\right|^2 dx\right)\\
 &+ \left[\left(\frac{\rho_2}{\rho_2-\rho_1}\right)^q+\left(\frac{\rho_2}{\rho_2-\rho_1}\right)^4\right]\lambda^2.
\end{align*}
Finally, using the iteration lemma (see Lemma \ref{iter_lemma}) we get the desired estimate on $2\rho.$ This completes the proof.
\end{proof}
\begin{lemma}\label{LEM5.2}
  Let $u$ be a weak solution to \eqref{eq:main equation} with the assumption \ref{eq: p-phase condition} in force. Then for $Q_{c_v (2\rho), \lambda}(z_0)\subset \Omega_T,$ there exist constants $c=c(\textnormal{\texttt{data}})$ and $\theta_0 (N, p, q) \in (0, 1)$ such that for any $\theta \in (\theta_0, 1)$ the following estimate,  
  \begin{align*}
		\miint{Q_{2\rho, \lambda}(z_0)} H\left(z, \left|\frac{u-(u)_{Q_{2\rho, \lambda}(z_0)}}{2\rho}\right|\right)\,dz
		&\leq c\lambda^{(1-\theta)p} \miint{Q_{2\rho, \lambda}(z_0)}H(z, |D u|)^{\theta}dz.
	\end{align*}
\end{lemma}
\begin{proof}
We use assumption \eqref{eq: p-phase condition} and Lemma \ref{lem : lemma 2.1}. To estimate the $p$-term in $H(z, \cdot),$ we use Lemma \ref{lem : lemma 2.1} with $s=\theta p, \sigma =p, \vartheta=\theta$ and $r=2.$ It is easy to see that these choices are admissible, 
\begin{align*}
-\frac{N}{p} \leq \theta \left(1-\frac{N}{\theta p}\right)-(1-\theta)\frac{N}{2} \,\,\, \text{if and only if}\,\,\, \theta \geq \frac{N}{N+2}.    
\end{align*}
Therefore, an application of Lemma \ref{lem : lemma 2.1} gives,
\begin{align}\label{eq: 4.8}
&\miint{Q_{2\rho, \lambda}(z_0)} \left|\frac{u-(u)_{Q_{2\rho, \lambda}(z_0)}}{2\rho}\right|^p dz \nonumber\\
&\lesssim \miint{Q_{2\rho, \lambda}(z_0)}\left(\left|\frac{u-(u)_{Q_{2\rho, \lambda}(z_0)}}{2\rho}\right|^{\theta p} + |D u|^{\theta p} dz\right)\left(\sup_{I_{2\rho}(t_0)}\dashint_{B_{2\rho}(x_0)}\left|\frac{u-(u)_{Q_{2\rho, \lambda}(z_0)}}{2\rho}\right|^2 dx\right)^{\frac{(1-\theta)p}{2}}\nonumber\\
&\overset{\text{Lemma}~\ref{p_intrinsic poincare 2}}{\lesssim} \left(\miint{Q_{2\rho, \lambda}(z_0)}H(z,|D u|)^{\theta} dz\right)\left(\sup_{I_{2\rho}(t_0)}\dashint_{B_{2\rho}(x_0)}\left|\frac{u-(u)_{Q_{2\rho, \lambda}(z_0)}}{2\rho}\right|^2 dx\right)^{\frac{(1-\theta)p}{2}}\nonumber\\
&\overset{\text{Lemma}~\ref{LEMMA7.5}}{\lesssim} \lambda^{(1-\theta)p} \left(\miint{Q_{2\rho, \lambda}(z_0)}H(z, |Du|)^{\theta} dz\right).
\end{align}
To estimate the $q$-term in $H(z, \cdot)$, we use $a\in \mathcal{Z}^{\kappa}(\Omega_T)$ and the assumption \eqref{eq: p-phase condition}.
\begin{align}\label{eq: 4.9}
&\miint{Q_{2\rho, \lambda}(z_0)} a(z)\left|\frac{u-(u)_{Q_{2\rho, \lambda}(z_0)}}{2\rho}\right|^q dz\nonumber\\
&\leq c_a \miint{Q_{2\rho, \lambda}(z_0)} \inf_{Q_{2\rho, \lambda}(z_0)}a(z) \left|\frac{u-(u)_{Q_{2\rho, \lambda}(z_0)}}{2\rho}\right|^q dz\nonumber\\
&+ c_a\miint{Q_{2\rho, \lambda}(z_0)} \rho^{\kappa} \left|\frac{u-(u)_{Q_{2\rho, \lambda}(z_0)}}{2\rho}\right|^q dz \overset{\eqref{eq: p-phase condition}}{\lesssim} \rho^{\kappa}\miint{Q_{2\rho, \lambda}(z_0)} \left|\frac{u-(u)_{Q_{2\rho, \lambda}(z_0)}}{2\rho}\right|^q dz\nonumber\\
\end{align}
Now we use the fact that
\begin{align*}
    \kappa + (q-p)\left(\frac{2\gamma}{q}-1\right)\geq 0,
\end{align*}
and estimate \eqref{eq: 4.9} as
\begin{align}\label{3rd estimate}
&\rho^{\kappa}\miint{Q_{2\rho, \lambda}(z_0)} \left|\frac{u-(u)_{Q_{2\rho, \lambda}(z_0)}}{2\rho}\right|^q dz\nonumber\\
&\leq \rho^{\kappa}\miint{Q_{2\rho, \lambda}(z_0)}\left|\frac{u-(u)_{Q_{2\rho, \lambda}(z_0)}}{2\rho}\right|^{q-p}\left|\frac{u-(u)_{Q_{2\rho, \lambda}(z_0)}}{2\rho}\right|^p dz\nonumber\\
&\leq c \rho^{ \kappa + (q-p)\left(\frac{2\gamma}{q}-1\right)}\miint{Q_{2\rho, \lambda}(z_0)}\left|\frac{u-(u)_{Q_{2\rho, \lambda}(z_0)}}{2\rho}\right|^p dz\leq c \lambda^{(1-\theta)p}\left(\miint{Q_{2\rho, \lambda}(z_0)} H(z, |Du|)^{\theta}\, dz\right)
\end{align}
Combining the estimates \eqref{eq: 4.8} and \eqref{3rd estimate}, we complete the proof.
\end{proof}
\begin{lemma}\label{LEM5.3}
Let $u$ be a weak solution to \eqref{eq:main equation} and the assumption \ref{eq: p-phase condition} is in force. Then there exist constants $c=c(\textnormal{\texttt{data}})$ and $\theta_0 \in (0,1)$ depending on $N,p,q$ such that for any $\theta\in (\theta_0,1)$ and $\varepsilon \in (0, 1),$ we have
\begin{align*}
\lambda^{p-2} \miint{Q_{2\rho, \lambda}(z_0)}\left|\frac{u-(u)_{Q_{2\rho, \lambda}(z_0)}}{2\rho}\right|^2 dz \leq  c \left(\miint{Q_{2\rho, \lambda}(z_0)}H(z, |D u|)^{\theta}\, dz\right)^{\frac{1}{\theta}}+\varepsilon \lambda^{p}    
\end{align*}
whenever $Q_{c_v(2\rho), \lambda}(z_0)\subset \Omega_T.$
\end{lemma}
\begin{proof}
We use Lemma \ref{lem : lemma 2.1}  with $\sigma =2, s=\theta p, \vartheta=\frac{1}{2}$ and $r=2,$  Lemma \ref{p_intrinsic poincare 2} and Lemma \ref{LEMMA7.5}. This gives
\begin{align*} 
\miint{Q_{2\rho, \lambda}(z_0)}\left|\frac{u-(u)_{Q_{2\rho, \lambda}(z_0)}}{2\varrho}\right|^2 dz &\lesssim \left(\miint{Q_{2\rho, \lambda}(z_0)}H(z, |D u|)^{\theta}dz\right) ^{\frac{1}{\theta p}}\left(\sup_{I_{2\rho}(t_0)}\dashint_{B_{2\rho}(x_0)}\left|\frac{u-(u)_{Q_{2\rho, \lambda}(z_0)}}{2\varrho}\right|^2 dx\right)^{\frac{1}{2}}\\
&\overset{\text{Lemma}~\ref{LEMMA7.5}}{\lesssim} \lambda \left(\miint{Q_{2\varrho, \lambda}(z_0)} H(z, |D u|)^{\theta}dz\right)^{\frac{1}{\theta p}}.
\end{align*}
Therefore, applying Young's inequality with $(\frac{p}{p-1}, p)$, we get
\begin{align*}
  \lambda^{p-2} \miint{Q_{2\rho, \lambda}(z_0)}\left|\frac{u-(u)_{Q_{2\rho, \lambda}(z_0)}}{2\varrho}\right|^2 dz &\lesssim\lambda^{p-1}\left(\miint{Q_{2\rho, \lambda}(z_0)}H(z,  |D u|)^{\theta}dz\right)^{\frac{1}{\theta p}}.\\
  &\lesssim \varepsilon \lambda^p + c(\varepsilon) \left(\miint{Q_{2\varrho, \lambda}(z_0)} H(z, |Du|)^{\theta}dz\right)^{\frac{1}{\theta}}.
\end{align*}
This completes the proof.
\end{proof}
Now we are ready to prove the reverse H\"{o}lder inequality for the $p$-phase.
\begin{lemma}\label{LEM5.4}
Let $u$ be a weak solution to \eqref{eq:main equation}. Moreover, we assume that the assumption \eqref{eq: p-phase condition} on the $p$-intrinsic cylinders holds. Then there exist constants $c=c(\textnormal{\texttt{data}})$ and $\theta_0 \in (0,1)$ depending on $N,p,q$ such that for any $\theta \in (\theta_0, 1),$  we have
\begin{align*}
\miint{Q_{\rho, \lambda}(z_0)}H(z, |D u|)\, dz \leq c \left(\miint{Q_{2\rho, \lambda}(z_0)}H(z, |D u|)^{\theta}\,dz\right)^{\frac{1}{\theta}}
\end{align*}
whenever $Q_{c_v(2\rho), \la}(z_0)\subset \Omega_T.$
\end{lemma}
\begin{proof}
From the Caccioppoli inequality (Theorem \ref{lem: energy inequality}) with $R=2\rho, S=4\lambda^{2-p}\rho^2, r=\rho,$ and $\tau=\lambda^{2-p}\rho^2,$ we have
\begin{align}\label{EQUATION7.22}
    \miint{Q_{\rho, \lambda}(z_0)} H(z, |D u|)\, dz &\lesssim \miint{Q_{2\rho, \lambda}(z_0)} H\left(z, \left|\frac{u-(u)_{Q_{2\rho, \lambda}(z_0)}}{2\rho}\right|\right)\,dz\nonumber\\
	 &+\lambda^{p-2} \miint{Q_{2\rho, \lambda}(z_0)}\left|\frac{u-(u)_{Q_{2\rho, \lambda}(z_0)}}{2\rho}\right|^2\,dz.
\end{align}
From Lemma \ref{LEM5.2}, we have
\begin{align*}
\miint{Q_{2\rho, \lambda}(z_0)} H\left(z, \left|\frac{u-(u)_{Q_{2\rho, \lambda}(z_0)}}{2\rho}\right|\right)\,dz
&\lesssim \lambda^{(1-\theta)p} \miint{Q_{2\rho, \lambda}(z_0)}H(z, |D u|)^{\theta}\,dz\\
&\overset{\text{Young's inequality}}{\leq} \varepsilon \la^p + c \left(\miint{Q_{2\rho, \lambda}(z_0)}H(z,|D u|)^{\theta}\, dz\right)^{\frac{1}{\theta}}.
\end{align*}
 On the other hand, from Lemma \ref{LEM5.3}, we obtain
\begin{align*}
 \lambda^{p-2} \miint{Q_{2\rho, \lambda}(z_0)}\left|\frac{u-(u)_{Q_{2\rho, \lambda}(z_0)}}{2\rho}\right|^2\,dz &\leq \varepsilon \la^p + c \left(\miint{Q_{2\rho, \lambda}(z_0)}H(z, |D u|)^{\theta}\, dz\right)^{\frac{1}{\theta}}.
\end{align*}
Substituting the above estimates into \eqref{EQUATION7.22}, we get
\begin{align*}
   \miint{Q_{\rho, \lambda}(z_0)} H\left(z, |D u|\right)\, dz \leq \varepsilon \la^p + c \left(\miint{Q_{2\rho, \lambda}(z_0)}H(z, |D u|)^{\theta}dz\right)^{\frac{1}{\theta}}.
\end{align*}
The proof can be completed by absorbing $\varepsilon \la^p$ on the left hand side, which is allowed because of the stopping time (see the third display) in \eqref{eq: p-phase condition}.
\end{proof}
\begin{remark} If $\kappa < 1 - \dfrac{2\gamma}{q}$, then $q<p+1,$ and we may take $\theta\in\bigl(\tfrac{q-1}{p},\,1\bigr].$ We can observe that the reverse H\"older inequality in the $p$-phase then follows from the $p$-gradient term alone, i.e., both terms of $H(z,|Du|)$ are bounded by
$$\miint{Q_{\rho,\lambda}(z_0)}|Du|^{\theta p}\,dz.$$ This matches the heuristic idea of $p$-phase, i.e., when 
$a$ is small, the equation behaves like the parabolic $p$-Laplace equation. When $\kappa$ is larger, so that $q \ge p+1$, the $p$-gradient term no longer suffices, and we need to get a bound by
$$\miint{Q_{\rho,\lambda}(z_0)} H(z,|Du|)^{\theta}\,dz\quad \text{with }\quad  \theta\in\bigl(\tfrac{q-1}{q},1\bigr].$$
\end{remark}

\subsection{Reverse H\"{o}lder inequality for $(p, q)$-phase}  We assume the following $(p, q)$-intrinsic conditions:
\begin{equation}    \label{eq: pq-phase condition}
    \left\{
    \begin{aligned}
        &\frac{a(z_0)}{2c_a}\leq a(z)\leq 2c_a a(z_0)\quad \text{for every $z \in G_{4\rho, \lambda}(z_0)$,}\\
        &\miint{G_{\tau,\lambda}(z_0)} H(z,|Du|) \; dz <\lambda^p+a(z_0)\lambda^q \ \, \text{ for every }\tau\in(\rho,c_v(2\rho)],\\
        &\miint{G_{\rho,\lambda}(z_0)} H(z,|Du|) \; dz =\lambda^p+a(z_0)\lambda^q.
    \end{aligned}
    \right.
\end{equation}
\begin{lemma}\label{Lemma 3.13}
Let $u$ be a weak solution to \eqref{eq:main equation}. Then there exists a constant $c=c(\textnormal{\texttt{data}}) \geq K^{2p/q}$ such that for any $G_{4\rho,\lambda}(z_0)\subset \Omega_T$ with \eqref{eq: pq-phase condition}, $\tau \in [2\rho,4\rho]$ and $\theta \in ((q-1)/q,1]$, we have
\begin{align}\label{Eq 3.10}
&\miint{G_{\tau,\lambda}(z_0)} \frac{|u-(u)_{G_{\tau,\lambda}(z_0)}|^{\theta p}}{\tau^{\theta p}}\; dz \leq c \miint{G_{\tau,\lambda}(z_0)} H(z, |Du|)^{\theta}\, dz
\end{align}
and 
\begin{align}\label{Eq: 3.11}
&\miint{G_{\tau,\lambda}(z_0)} \frac{|u-(u)_{G_{\tau,\lambda}(z_0)}|^{\theta q}}{\tau^{\theta q}}\; dz \leq c \miint{G_{\tau,\lambda}(z_0)} |Du|^{\theta q}\, dz.
\end{align}
\end{lemma}
\begin{proof} We prove \eqref{Eq 3.10} for $K^2\lambda^p\leq a(z_0)\lambda^q$ and $K^2\lambda^p\geq a(z_0)\lambda^q.$ Similarly, we prove \eqref{Eq: 3.11} for $K^2\lambda^p\leq a(z_0)\lambda^q$ and $K^2\lambda^p\geq a(z_0)\lambda^q.$
First, we observe from \eqref{eq : parabolic poincare inequality in Section 3_1} of Lemma \ref{lem : parabolic poincare inequality in Section 3} that
\begin{align}\label{EQQ3.12}
    &\miint{G_{\tau, \lambda}(z_0)}\frac{|u-(u)_{G_{\tau, \lambda}(z_0)}|^{\theta p}}{\tau^{\theta p}}\,dz\nonumber\\
    &\leq c\miint{G_{\tau, \lambda}(z_0)}|Du|^{\theta p}\, dz + c\left(\frac{\lambda^2}{\lambda^p+a(z_0)\lambda^q}\miint{G_{\tau, \lambda}(z_0)}|Du|^{p-1}+a(z)|Du|^{q-1} dz\right)^{\theta p}\nonumber\\
    &\leq c\miint{G_{\tau, \lambda}(z_0)}|Du|^{\theta p}\, dz+ c\left(\frac{\lambda^2}{\lambda^p+a(z_0)\lambda^q}\miint{G_{\tau, \lambda}(z_0)}|Du|^{p-1} dz\right)^{\theta p}\nonumber\\
    &\quad +c\left(\frac{\lambda^2}{\lambda^p+a(z_0)\lambda^q}\miint{G_{\tau, \lambda}(z_0)}a(z)|Du|^{q-1} dz\right)^{\theta p}=c\miint{G_{\tau, \lambda}(z_0)}|Du|^{\theta p}\, dz+\mathrm{I}+\mathrm{II}.
\end{align}
In view of Remark \ref{REM3.3}, we divide the proof into two cases.

\textbf{Case I:} $K^2\lambda^p\leq a(z_0)\lambda^q.$ In this case, we estimate $\mathrm{I}$ and $\mathrm{II}$ of \eqref{EQQ3.12}. First we observe, since $a(z_0)\neq 0,$ from $K^2\lambda^p\leq a(z_0)\lambda^q$, the first and second conditions of \eqref{eq: pq-phase condition}, we get
\begin{align*}
 \miint{G_{\tau, \lambda}(z_0)} a(z_0)|Du|^q\,dz \leq 2 c_a \miint{G_{\tau, \lambda}(z_0)} a(z)|Du|^q\,dz&\leq 2c_a \left(\lambda^p+a(z_0)\lambda^q\right)\leq c\, a(z_0)\lambda^q.
\end{align*}
Thus, we have
\begin{align}\label{bound by q}
    \miint{G_{\tau, \lambda}(z_0)}|Du|^q\, dz \leq c\lambda^q.
\end{align}

\noindent\textbf{Estimate of $\mathrm{I}$:} Using the second condition of \eqref{eq: pq-phase condition} and H\"{o}lder's inequality, we obtain
\begin{align*}
    &\left(\frac{\lambda^2}{\lambda^p+a(z_0)\lambda^q}\miint{G_{\tau, \lambda}(z_0)}|Du|^{p-1}\,dz\right)^{\theta p}\leq \left(\frac{\lambda^2}{\lambda^p+a(z_0)\lambda^q}\right)^{\theta p}\left(\miint{G_{\tau, \lambda}(z_0)} |Du|^{\theta p}\,dz\right)^{p-1}\\
    &\leq c\left(\frac{\lambda^2}{\lambda^p+a(z_0)\lambda^q}\right)^{\theta p}(\lambda^p+a(z_0)\lambda^q)^{\theta (p-2)} \miint{G_{\tau, \lambda}(z_0)} |Du|^{\theta p}\,dz \leq c \miint{G_{\tau, \lambda}(z_0)}|Du|^{\theta p}\, dz,
\end{align*}
where we used
\begin{align*}
\left(\frac{\lambda^2}{\lambda^p+a(z_0)\lambda^q}\right)^{\theta p}(\lambda^p+a(z_0)\lambda^q)^{\theta (p-2)} \leq 1 .   
\end{align*}
Thus, we finally obtain
\begin{align*}
\left(\frac{\lambda^2}{\lambda^p+a(z_0)\lambda^q}\miint{G_{\tau, \lambda}(z_0)}|Du|^{p-1}\,dz\right)^{\theta p} \leq c \miint{G_{\tau, \lambda}(z_0)} |Du|^{\theta p}\,dz.    
\end{align*}

\noindent\textbf{Estimate of $\mathrm{II}$:} Similarly, using the first condition of \eqref{eq: pq-phase condition} and H\"{o}lder's inequality, we get
\begin{align*}
&\left(\frac{\lambda^2}{\lambda^p+a(z_0)\lambda^q}\miint{G_{\tau, \lambda}(z_0)}a(z)|Du|^{q-1}\,dz\right)^{\theta p}\\
&\leq (2c_a)^{\frac{\theta p}{q}}\left(\frac{\lambda^2 a(z_0)^{\frac{1}{q}}}{\lambda^p+a(z_0)\lambda^q}\right)^{\theta p}\left(\miint{G_{\tau, \lambda}(z_0)}a(z)^{\frac{q-1}{q}}|Du|^{q-1}\, dz\right)^{\theta p}\\
&\leq 2c_a \left(\frac{\lambda^2 a(z_0)^{\frac{1}{q}}}{\lambda^p+a(z_0)\lambda^q}\right)^{\theta p} \left(\miint{G_{\tau, \lambda}(z_0)}a(z)^{\theta}|Du|^{\theta q}\, dz\right)^{\frac{p(q-1)}{q}}\\
&=2c_a \left(\frac{\lambda^2 a(z_0)^{\frac{1}{q}}}{\lambda^p+a(z_0)\lambda^q}\right)^{\theta p} \left(\miint{G_{\tau, \lambda}(z_0)}a(z)^{\theta}|Du|^{\theta q}\, dz\right)^{\frac{p(q-1)}{q}-1}\left(\miint{G_{\tau, \lambda}(z_0)}a(z)^{\theta}|Du|^{\theta q}\, dz\right).
\end{align*}
Since $\frac{p(q-1)}{q}-1=\frac{pq-p-q}{q}\geq 0$ for $2\leq p\leq q$, using \eqref{bound by q} and denoting $\Lambda=\lambda^p+a(z_0)\lambda^q$ we get,
\begin{align*}
    2c_a&\left(\frac{\lambda^{2}a(z_0)^{1/q}}{\lambda^{p}+a(z_0)\lambda^{q}}\right)^{\theta p}
    \left(\miint{G_{\tau,\lambda}(z_0)} a(z)^{\theta}|Du|^{\theta q}\,dz\right)^{\frac{p(q-1)}{q}-1}
    \notag\\
    &\leq c\left(\frac{\lambda^{2}a(z_0)^{1/q}}{\Lambda}\right)^{\theta p}
        a(z_0)^{\theta\left(\frac{p(q-1)}{q}-1\right)}\,
        \lambda^{\theta q\left(\frac{p(q-1)}{q}-1\right)}
    \notag\\
    &= c\,\frac{\lambda^{\theta(pq-q+p)}\,a(z_0)^{\theta(p-1)}}{\Lambda^{\theta p}}
    = c\left(\frac{\lambda^{\,pq-q+p}\,a(z_0)^{p-1}}{\Lambda^{p}}\right)^{\theta}.
\end{align*}
Finally, using $\Lambda\geq a(z_0)\lambda^{q}$ and the case assumption $K^{2}\lambda^{p}\leq a(z_0)\lambda^{q}$, we obtain
\begin{align*}
    \frac{\lambda^{\,pq-q+p}\,a(z_0)^{p-1}}{\Lambda^{p}}
    \leq \frac{\lambda^{\,pq-q+p}\,a(z_0)^{p-1}}{\bigl(a(z_0)\lambda^{q}\bigr)^{p}}
    = \frac{1}{a(z_0)\lambda^{\,q-p}}
    \leq \frac{1}{K^{2}}\leq 1,
\end{align*}
so that
\begin{align*}
    2c_a\left(\frac{\lambda^{2}a(z_0)^{1/q}}{\lambda^{p}+a(z_0)\lambda^{q}}\right)^{\theta p}
    \left(\miint{G_{\tau,\lambda}(z_0)} a(z)^{\theta}|Du|^{\theta q}\,dz\right)^{\frac{p(q-1)}{q}-1}
    \leq c.
\end{align*}
Thus, combining the above estimates, we obtain \eqref{Eq 3.10}.

\noindent\textbf{Case II:} $K^2\lambda^p\geq a(z_0)\lambda^q.$ Again, we estimate $\mathrm{I}$ and $\mathrm{II}$ of \eqref{EQQ3.12}. Using this and the second assumption of \eqref{eq: pq-phase condition}, we get
\begin{align}\label{bound by lambda-p}
    \miint{G_{\tau, \lambda}(z_0)}|Du|^p\,dz< \lambda^p+a(z_0)\lambda^q\leq c \lambda^p,
\end{align}
The estimates in this case are similar to those in Case I. For the sake of completeness, we give the details.

\noindent \textbf{Estimate of $\mathrm{I}$:} Using \eqref{bound by lambda-p} and H\"{o}lder's inequality, we obtain
\begin{align*}
    &\left(\frac{\lambda^2}{\lambda^p+a(z_0)\lambda^q}\miint{G_{\tau, \lambda}(z_0)}|Du|^{p-1}\,dz\right)^{\theta p}\\
    & \leq c \left(\frac{\lambda^2}{\lambda^p+a(z_0)\lambda^q}\right)^{\theta p}\lambda^{\theta p (p-2)} \miint{G_{\tau, \lambda}(z_0)}|Du|^{\theta p}\, dz\\
    &=c\left(\frac{\lambda^p}{\lambda^p+a(z_0)\lambda^q}\right)^{\theta p} \miint{G_{\tau, \lambda}(z_0)}|Du|^{\theta p}\, dz \leq c\miint{G_{\tau, \lambda}(z_0)}|Du|^{\theta p}\, dz.
\end{align*}
\textbf{Estimate of $\mathrm{II}$:} Similarly, we estimate
\begin{align*}
 &\left(\frac{\lambda^2}{\lambda^p+a(z_0)\lambda^q}\miint{G_{\tau, \lambda}(z_0)}a(z)|Du|^{q-1}\,dz\right)^{\theta p}\\
&\leq2c_a \left(\frac{\lambda^2 a(z_0)^{\frac{1}{q}}}{\lambda^p+a(z_0)\lambda^q}\right)^{\theta p} \left(\miint{G_{\tau, \lambda}(z_0)}a(z)^{\theta}|Du|^{\theta q}\, dz\right)^{\frac{p(q-1)}{q}-1}\left(\miint{G_{\tau, \lambda}(z_0)}a(z)^{\theta}|Du|^{\theta q}\, dz\right)\\
    &\leq 2c_a\left(\frac{\lambda^{2}a(z_0)^{1/q}}{\Lambda}\right)^{\theta p}\Lambda^{\theta\left(\frac{p(q-1)}{q}-1\right)}
        \left(\miint{G_{\tau,\lambda}(z_0)} a(z)^{\theta}|Du|^{\theta q}\,dz\right)
    \notag\\
    &= 2c_a\left[\frac{\lambda^{2p}\,a(z_0)^{p/q}}{\Lambda^{(p+q)/q}}\right]^{\theta}
        \left(\miint{G_{\tau,\lambda}(z_0)} a(z)^{\theta}|Du|^{\theta q}\,dz\right).
\end{align*}
Using $\Lambda\geq\lambda^{p}$ and the case assumption $a(z_0)\lambda^{q}\leq K^{2}\lambda^{p}$,
\begin{align*}
    \frac{\lambda^{2p}\,a(z_0)^{p/q}}{\Lambda^{(p+q)/q}}
    \leq\frac{\lambda^{2p}\,a(z_0)^{p/q}}{\lambda^{p(p+q)/q}}
    =\bigl(a(z_0)\lambda^{\,q-p}\bigr)^{p/q}\leq K^{2p/q},
\end{align*}
hence
\begin{align*}
    \left(\frac{\lambda^{2}}{\lambda^{p}+a(z_0)\lambda^{q}}
        \miint{G_{\tau,\lambda}(z_0)} a(z)|Du|^{q-1}\,dz\right)^{\theta p}
    \leq K^{2p/q}\,\miint{G_{\tau,\lambda}(z_0)} a(z)^{\theta}|Du|^{\theta q}\,dz.
\end{align*}
Thus, combining the above estimates, we obtain \eqref{Eq 3.10}.

Next, we prove the estimate \eqref{Eq: 3.11}. Again, it follows from \eqref{eq : parabolic poincare inequality in Section 3_1} of Lemma~\ref{lem : parabolic poincare inequality in Section 3} that
\begin{align}\label{EQQ3.14}
\miint{G_{\tau, \lambda}(z_0)}\frac{|u-(u)_{G_{\tau, \lambda}(z_0)}|^{\theta q}}{\tau^{\theta q}}\,dz &\leq c\miint{G_{\tau, \lambda}(z_0)}|Du|^{\theta q}\, dz+ c\left(\frac{\lambda^2}{\lambda^p+a(z_0)\lambda^q}\miint{G_{\tau, \lambda}(z_0)}|Du|^{p-1}\,dz\right)^{\theta q}\nonumber\\
&+c\left(\frac{\lambda^2}{\lambda^p+a(z_0)\lambda^q}\miint{G_{\tau, \lambda}(z_0)}a(z)|Du|^{q-1}\,dz\right)^{\theta q}\nonumber\\
&=c\miint{G_{\tau, \lambda}(z_0)}|Du|^{\theta q}\, dz+\mathrm{I}+\mathrm{II}.
\end{align}

\noindent \textbf{Case I:} $K^2\lambda^p\leq a(z_0)\lambda^q$. We estimate $\mathrm{I}$ and $\mathrm{II}$ of \eqref{EQQ3.14}.

\noindent \textbf{Estimate of $\mathrm{I}$:} Using H\"{o}lder's inequality and \eqref{bound by q} with $p-1<\theta q$, we see that
\begin{align*}
&\left(\frac{\lambda^2}{\lambda^p+a(z_0)\lambda^q}\miint{G_{\tau, \lambda}(z_0)}|Du|^{p-1}\,dz\right)^{\theta q}\\
&\leq c\left(\frac{\lambda^2}{\lambda^p+a(z_0)\lambda^q}\right)^{\theta q}\left(\miint{G_{\tau, \lambda}(z_0)}|Du|^{\theta q}\,dz\right)^{p-1}\\
&\leq c \left(\frac{\lambda^2}{\lambda^p+a(z_0)\lambda^q}\right)^{\theta q}\left(\miint{G_{\tau, \lambda}(z_0)}|Du|^q\, dz\right)^{\theta(p-2)} \miint{G_{\tau, \lambda}(z_0)}|Du|^{\theta q}\, dz\\
&\leq c\left(\frac{\lambda^p}{\lambda^p+a(z_0)\lambda^q}\right)^{\theta q} \miint{G_{\tau, \lambda}(z_0)}|Du|^{\theta q}\, dz\leq c \miint{G_{\tau, \lambda}(z_0)}|Du|^{\theta q}\, dz.
\end{align*}
\textbf{Estimate of $\mathrm{II}$:} Similarly,  using the first condition of \eqref{eq: pq-phase condition} and $q-1<\theta q$, we can estimate
\begin{align*}
&\left(\frac{\lambda^2}{\lambda^p+a(z_0)\lambda^q}\miint{G_{\tau, \lambda}(z_0)}a(z)|Du|^{q-1}\,dz\right)^{\theta q}\\
&\leq c\left(\frac{\lambda^2 a(z_0)}{\lambda^p+a(z_0)\lambda^q}\right)^{\theta q}\left(\miint{G_{\tau, \lambda}(z_0)}|Du|^q\,dz\right)^{\theta (q-2)} \miint{G_{\tau, \lambda}(z_0)}|Du|^{\theta q}\,dz\\
&\leq c \left(\frac{a(z_0)\lambda^q}{\lambda^p+a(z_0)\lambda^q}\right)^{\theta q} \miint{G_{\tau, \lambda}(z_0)}|Du|^{\theta q}\,dz\leq c \miint{G_{\tau, \lambda}(z_0)}|Du|^{\theta q}\,dz.
\end{align*}
Combining the above estimates completes the proof of \eqref{Eq: 3.11}.

\noindent \textbf{Case II:} $K^2\lambda^p\geq a(z_0)\lambda^q$. We estimate $\mathrm{I}$ and $\mathrm{II}$ of \eqref{EQQ3.14} in this case. 

\noindent\textbf{Estimate of $\mathrm{I}$:} Using H\"{o}lder's inequality and \eqref{bound by lambda-p}, we get
\begin{align*}
&\left(\frac{\lambda^2}{\lambda^p+a(z_0)\lambda^q}\miint{G_{\tau, \lambda}(z_0)}|Du|^{p-1}\,dz\right)^{\theta q}\\
&\leq c\left(\frac{\lambda^2}{\lambda^p+a(z_0)\lambda^q}\right)^{\theta q}\left(\miint{G_{\tau, \lambda}(z_0)}|Du|)^{\theta p}\,dz\right)^{\frac{(p-1)q}{p}}\\
&=c\left(\frac{\lambda^2}{\lambda^p+a(z_0)\lambda^q}\,\right)^{\theta q}\left(\miint{G_{\tau, \lambda}(z_0)}|Du|^{\theta p}\,dz\right)^{\frac{(p-2)q}{p}}\left(\miint{G_{\tau, \lambda}(z_0)}|Du|^{\theta p}\,dz\right)^{\frac{q}{p}}\\
&\leq c\left(\frac{\lambda^2}{\lambda^p+a(z_0)\lambda^q}\right)^{\theta q}\lambda^{\theta q(p-2)} \miint{G_{\tau, \lambda}(z_0)}|Du|^{\theta q} dz\\
&=c\left(\frac{\lambda^p}{\lambda^p+a(z_0)\lambda^q}\right)^{\theta q} \miint{G_{\tau, \lambda}(z_0)}|Du|^{\theta q}\,dz\leq c\miint{G_{\tau, \lambda}(z_0)}|Du|^{\theta q} dz.
\end{align*}

\noindent \textbf{Estimate of $\mathrm{II}$:} Similarly, we estimate
\begin{align*}
&\left(\frac{\lambda^2}{\lambda^p+a(z_0)\lambda^q}\miint{G_{\tau, \lambda}(z_0)}a(z)|Du|^{q-1}\,dz\right)^{\theta q}\\
&\leq c\left(\frac{\lambda^2}{\lambda^p+a(z_0)\lambda^q}\right)^{\theta q}a(z_0)^{\theta}\left(\miint{G_{\tau, \lambda}(z_0)}a(z)^{\frac{q-1}{q}}|Du|^{q-1}\, dz\right)^{\theta q}\\
&\leq c \left(\frac{\lambda^2}{\lambda^p+a(z_0)\lambda^q}\right)^{\theta q}a(z_0)^{\theta}\left(\miint{G_{\tau, \lambda}(z_0)}a(z)^{\theta}|Du|^{\theta q}dz\right)^{q-1}\\
&\leq c \left(\frac{\lambda^2}{\lambda^p+a(z_0)\lambda^q}\right)^{\theta q}(\lambda^p+a(z_0)\lambda^q)^{\theta (q-2)}a(z_0)^{2\theta} \miint{G_{\tau, \lambda}(z_0)}|Du|^{\theta q}dz.
\end{align*}
As in the previous case, using $K^2\lambda^p\geq a(z_0)\lambda^q,$ we obtain
\begin{align*}
 \left(\frac{\lambda^2}{\lambda^p+a(z_0)\lambda^q}\right)^{\theta q}(\lambda^p+a(z_0)\lambda^q)^{\theta (q-2)}a(z_0)^{2\theta} \leq 1.   
\end{align*}
It follows that
\begin{align*}
    \left(\frac{\lambda^2}{\lambda^p+a(z_0)\lambda^q}\miint{G_{\tau, \lambda}(z_0)}a(z)|Du|^{q-1}\,dz\right)^{\theta q} \leq c \miint{G_{\tau, \lambda}(z_0)} |Du|^{\theta q} dz.
\end{align*}
Thus, combining the above estimates, we obtain \eqref{Eq: 3.11} for this case.
\end{proof}
Next, we prove the following:
\begin{lemma}\label{Lemma 3.14}
Let $u$ be a weak solution to \eqref{eq:main equation}. Then for $G_{c_v(2 \rho),\lambda}(z_0) \subset \Omega_T$ with \eqref{eq: pq-phase condition}, there exists a constant $c=c(\textnormal{\texttt{data}}) \geq 1$ such that
\begin{align*}
S(u)_{G_{2\rho, \lambda}(z_0)}=\sup _{I^{(p,q)}_{2 \rho,\lambda}(t_0)} \dashint_{B_{2 \rho}(x_0)} \frac{\left|u-(u)_{G_{2 \rho,\lambda}(z_0)}\right|^2}{(2 \rho)^2}\; dx \leq c \lambda^2 .
\end{align*}
\end{lemma}
\begin{proof}
Let $2\rho \leq \rho_1<\rho_2 \leq 4\rho$. By Lemma~\ref{lem: energy inequality}, there exists a constant $c$ depending on $N,p,q,\nu$ and $L$ such that 
\begin{align}\label{eq : 3.12}
&\left(\lambda^{p-2}+a(z_0)\lambda^{q-2}\right)S(u)_{G_{\rho_1, \lambda}(z_0)}\nonumber\\
&\leq \frac{c\rho_2^q}{(\rho_2-\rho_1)^q}\miint{G_{\rho_2,\lambda}(z_0)}\left( \frac{|u-(u)_{G_{\rho_2,\lambda}(z_0)}|^p}{\rho^p_2}+a(z)\frac{|u-(u)_{G_{\rho_2,\lambda}(z_0)}|^q}{\rho^q_2} \right) dz \nonumber\\
&\quad + \frac{c\rho_2^2 \left(\lambda^{p-2}+a(z_0)\lambda^{q-2}\right)}{(\rho_2-\rho_1)^2}\miint{G_{\rho_2,\lambda}(z_0)}\frac{|u-(u)_{G_{\rho_2,\lambda}(z_0)}|^2}{\rho_2^2}\; dz
\end{align}
Using Lemma ~\ref{Lemma 3.13} with $\theta =1$, we estimate the first term of \eqref{eq : 3.12} as follows:
\begin{align*}
 &\miint{G_{\rho_2,\lambda}(z_0)}\left( \frac{|u-(u)_{G_{\rho_2,\lambda}(z_0)}|^p}{\rho^p_2}+a(z)\frac{|u-(u)_{G_{\rho_2,\lambda(z_0)}}|^q}{\rho^q_2} \right) dz\\
 &\leq c\miint{G_{\rho_2,\lambda}(z_0)} \left(\frac{|u-(u)_{G_{\rho_2,\lambda}(z_0)}|^p}{\rho^p_2}+a(z_0)\frac{|u-(u)_{G_{\rho_2,\lambda}(z_0)}|^q}{\rho^q_2} \right) dz\\
 &\overset{\text{Lemma}\,\ref{Lemma 3.13}}{\leq} c\miint{G_{\rho_2,\lambda}(z_0)}|Du|^p+a(z_0)|Du|^q\,dz\\
 &\leq c \miint{G_{\rho_2,\lambda}(z_0)}|Du|^p+a(z)|Du|^q\,dz\leq c\left(\lambda^p+a(z_0)\lambda^q\right).
\end{align*}
For the case $K^2\lambda^p\leq a(z_0)\lambda^q$,  using Lemma~\ref{lem : lemma 2.1} with $\sigma =2, s=q, r=2, \vartheta=\frac{1}{2},$ we get
\begin{align*}
   \miint{G_{\rho_2,\lambda}(z_0)}\frac{|u-(u)_{G_{\rho_2,\lambda}(z_0)}|^2}{\rho_2^2}\; dz \leq c\left(\miint{G_{\rho_2, \lambda}(z_0)}\left(\frac{|u-(u)_{G_{\rho_2,\lambda}(z_0)}|^q}{\rho_2^q}+|Du|^q\right)dz\right)^{\frac{1}{q}}\left(S(u)_{G_{\rho_2, \lambda}(z_0)}\right)^{\frac{1}{2}}.
\end{align*}
Now using \eqref{Eq: 3.11} of Lemma ~\ref{Lemma 3.13} and \eqref{bound by q}, we obtain
\begin{align*}
\left(\miint{G_{\rho_2, \lambda}(z_0)}\left(\frac{|u-(u)_{G_{\rho_2,\lambda}(z_0)}|^q}{\rho_2^q}+|Du|^q\right)dz\right)^{\frac{1}{q}} \leq c \lambda.    
\end{align*}
On the other hand, for the case $K^2\lambda^p\geq a(z_0)\lambda^q$, similarly,  we derive
\begin{align*}
   \miint{G_{\rho_2,\lambda}(z_0)}\frac{|u-(u)_{G_{\rho_2,\lambda}(z_0)}|^2}{\rho_2^2}\; dz \leq c\left(\miint{G_{\rho_2, \lambda}(z_0)}\left(\frac{|u-(u)_{G_{\rho_2,\lambda}(z_0)}|^p}{\rho_2^p}+|Du|^p\right)dz\right)^{\frac{1}{p}}\left(S(u)_{G_{\rho_2, \lambda}(z_0)}\right)^{\frac{1}{2}}.
\end{align*}
Now using \eqref{Eq 3.10} of Lemma ~\ref{Lemma 3.13} and \eqref{bound by lambda-p}, we obtain
\begin{align*}
    \left(\miint{G_{\tau, \lambda}(z_0)}\left(\frac{|u-(u)_{G_{\rho_2,\lambda}(z_0)}|^p}{\rho_2^p}+|Du|^p\right)dz\right)^{\frac{1}{p}} \leq c \lambda.
\end{align*}
Thus, together with the above estimates, we get
    \begin{align*}
     \miint{G_{\rho_2,\lambda}(z_0)}\frac{|u-(u)_{G_{\rho_2,\lambda}(z_0)}|^2}{\rho_2^2}\; dz \leq \lambda \left(S(u)_{G_{2\rho, \lambda}(z_0)}\right)^{\frac{1}{2}}.
    \end{align*}
    Now applying Young's inequality and the iteration lemma (Lemma~\ref{iter_lemma}), we complete the proof.
\end{proof}
\begin{lemma}\label{Lemma 3.15}
Let $u$ be a weak solution to \eqref{eq:main equation}. Then for $G_{c_v(2 \rho),\lambda}\left(z_0\right) \subset \Omega_T$ with \eqref{eq: pq-phase condition}, there exist constants $c=c(\textnormal{\texttt{data}}) \geq 1$ and $\theta_0=\theta_0(N, p, q) \in(0,1)$ such that for any $\theta \in\left(\theta_0, 1\right)$, we have
\begin{align}\label{EQQ3.15}
&\miint{G_{2 \rho,\lambda}\left(z_0\right)} H\left(z,\frac{\left|u-(u)_{G_{2\rho, \lambda}(z_0)}\right|}{2 \rho}\right) dz \leq c\left(\lambda^p+a(z_0)\lambda^q\right)^{1-\theta}\left(\miint{G_{2 \rho,\lambda}\left(z_0\right)} H(z, |Du|)^{\theta}\, dz\right). 
\end{align}
\end{lemma}
\begin{proof}
 We use Lemma~\ref{lem : lemma 2.1} with $\sigma=p, s=\theta p, \vartheta= \theta, r=2$ to deduce
 \begin{align*}
    &\miint{G_{2\rho, \lambda}(z_0)}\frac{\left|u-(u)_{G_{2\rho, \lambda}(z_0)}\right|^p}{(2 \rho)^p}\, dz\\
     &\quad \leq c\left(\miint{G_{2\rho, \lambda}(z_0)}\frac{\left|u-(u)_{G_{2\rho, \lambda}(z_0)}\right|^{\theta p}}{(2 \rho)^{\theta p}}+|Du|^{\theta p}\, dz\right)\left(S(u)_{G_{2\rho, \lambda}(z_0)}\right)^{\frac{(1-\theta)p}{2}}.
 \end{align*}
 Now we use Lemma~\ref{Lemma 3.13} and Lemma ~\ref{Lemma 3.14} to get
 \begin{align*}
&\left(\miint{G_{2\rho, \lambda}(z_0)}\frac{\left|u-(u)_{G_{2\rho, \lambda}(z_0)}\right|^{\theta p}}{(2 \rho)^{\theta p}}+|Du|^{\theta p}\, dz\right)\left(S(u)_{G_{2\rho, \lambda}(z_0)}\right)^{\frac{(1-\theta)p}{2}}\\
&\quad \leq c \lambda^{p(1-\theta)}\miint{G_{2\rho, \lambda}(z_0)}H(z, |Du|)^{\theta}\, dz.
 \end{align*}
 Hence, it follows that
 \begin{align*}
 \miint{G_{2\rho, \lambda}(z_0)}\frac{\left|u-(u)_{G_{2\rho, \lambda}(z_0)}\right|^p}{(2 \rho)^p}\, dz\leq c \lambda^{p(1-\theta)}\miint{G_{2\rho, \lambda}(z_0)}H(z, |Du|)^{\theta}\, dz.    
 \end{align*}
 Similarly, using Lemma~\ref{lem : lemma 2.1} with $\sigma=q, s=\theta q, \vartheta= \theta, r=2,$ we get
 \begin{align*}
 \miint{G_{2\rho, \lambda}(z_0)}a(z)\frac{\left|u-(u)_{G_{2\rho, \lambda}(z_0)}\right|^q}{(2 \rho)^q}\, dz\leq c (a(z_0)\lambda^q)^{1-\theta}\miint{G_{2\rho, \lambda}(z_0)}a(z)^{\theta}|Du|^{\theta q}\, dz.        
 \end{align*}
 Combining the above two estimates and using H\"{o}lder's inequality, we get the desired estimate \eqref{EQQ3.15}.
\end{proof}
\begin{lemma}
    Let $u$ be a weak solution to \eqref{eq:main equation}. Then for $G_{c_v(2 \rho),\lambda}\left(z_0\right) \subset \Omega_T$ with \eqref{eq: pq-phase condition}, there exist constants $c=c(\textnormal{\texttt{data}}) \geq 1$ and $\theta_0=\theta_0(N,p,q) \in (0,1)$ such that for any $\theta\in (\theta_0,1),$
    \begin{align*}
        \miint{G_{\rho,\lambda}(z_0)} H(z,|Du|)\; dz\leq c \left(\miint{G_{2\rho,\lambda}(z_0)} H(z,|Du|)^\theta\;dz\right)^\frac{1}{\theta}.
    \end{align*}
\end{lemma}
\begin{proof}
    From the energy estimate, Lemma \ref{lem: energy inequality}, we have
\begin{align*}
&\miint{G_{\rho, \lambda}(z_0)}H(z,|Du|)\;dz\\
& \leq c \miint{G_{2\rho, \lambda}(z_0)}  H\left(z,\frac{|u-(u)_{G_{2\rho, \lambda}(z_0)}|}{2\rho}\right) dz+ c\frac{\lambda^p+a(z_0)\lambda^q}{\lambda^2}\miint{G_{2\rho, \lambda}(z_0)}\frac{|u-(u)_{G_{2\rho, \lambda}(z_0)}|^2}{(2\rho)^2}\;dz\\
&= \mathrm{I}+\mathrm{II}.
\end{align*}
\textbf{Estimate of $\mathrm{I}$:} From Lemma \ref{Lemma 3.15}, applying Young's inequality with $\left(\frac{1}{1-\theta}, \frac{1}{\theta}\right)$ yields
\begin{align*}
\miint{G_{2\rho, \lambda}(z_0)}  H\left(z,\frac{|u-(u)_{z_0, 2\rho, \lambda}|}{2\rho}\right) dz&\leq c\left(\lambda^p+a(z_0)\lambda^q\right)^{1-\theta}\left(\miint{G_{2 \rho,\lambda}\left(z_0\right)} H(z, |Du|)^{\theta}\, dz\right)\\
&\leq \varepsilon \left(\lambda^p+a(z_0)\lambda^q\right)+ c(\varepsilon) \left(\miint{G_{2 \rho,\lambda}\left(z_0\right)} H(z, |Du|)^{\theta}\, dz\right)^{\frac{1}{\theta}}.
\end{align*}
\textbf{Estimate of $\mathrm{II}$:} We estimate $\mathrm{II}$ considering it in two separate parts. We proceed as follows:
\begin{align*}
\lambda^{p-2}\miint{G_{2\rho, \lambda}(z_0)}\frac{|u-(u)_{z_0, 2\rho, \lambda}|^2}{(2\rho)^2}\;dz +a(z_0)\lambda^{q-2}\miint{G_{2\rho, \lambda}(z_0)}\frac{|u-(u)_{z_0, 2\rho, \lambda}|^2}{(2\rho)^2}\;dz=\mathrm{J_1}+\mathrm{J_2}.
\end{align*}
First, we estimate $\mathrm{J}_1$ using Lemma \ref{lem : lemma 2.1} with $\sigma =2, s=\theta p, \vartheta=\frac{1}{2}, r=2.$ Indeed, 
\begin{align*}
 \mathrm{J}_1=&\lambda^{p-2}\miint{G_{2\rho, \lambda}(z_0)}\frac{|u-(u)_{G_{2\rho, \lambda}(z_0)}|^2}{(2\rho)^2}\;dz\\
 &\overset{\text{Lemma}\,\ref{lem : lemma 2.1}}{\leq} c \lambda^{p-2} \left(\miint{G_{2\rho, \lambda}(z_0)}\frac{|u-(u)_{z_0, 2\rho, \lambda}|^{\theta p}}{(2\rho)^{\theta p}}+|Du|^{\theta p}\,dz\right)^{\frac{1}{\theta p}}\left(S(u)_{G_{2\rho, \lambda}(z_0)}\right)^{\frac{1}{2}}\\
 &\overset{\text{Lemma} \,\ref{Lemma 3.14}}{\leq} c \lambda^{p-1}\left(\miint{G_{2\rho, \lambda}(z_0)}H(z, |Du|)^{\theta}\,dz\right)^{\frac{1}{\theta p}}.
\end{align*}
Now applying Young's inequality with $(\frac{p}{p-1}, p),$ we obtain
\begin{align*}
    \mathrm{J}_1 &\leq \varepsilon \lambda^{p}+ c(\varepsilon)\left(\miint{G_{2\rho, \lambda}(z_0)}H(z, |Du|)^{\theta}\,dz\right)^{\frac{1}{\theta }}\leq \varepsilon \lambda^p+ c(\varepsilon)\left(\miint{G_{2 \rho,\lambda}\left(z_0\right)} H(z, |Du|)^{\theta}\, dz\right)^{\frac{1}{\theta}}.
\end{align*}
Similarly, we can estimate $\mathrm{J}_2$ as follows:
\begin{align*}
    \mathrm{J}_2 \leq \varepsilon a(z_0)\lambda^q + c(\varepsilon)\left(\miint{G_{2 \rho,\lambda}\left(z_0\right)} H(z, |Du|)^{\theta}\, dz\right)^{\frac{1}{\theta}}.
\end{align*}
Thus, combining the above estimates gives
\begin{align*}
 &\frac{\lambda^p+a(z_0)\lambda^q}{\lambda^2}\miint{G_{2\rho, \lambda}(z_0)}\frac{|u-(u)_{z_0, 2\rho, \lambda}|^2}{(2\rho)^2}\;dz \leq \varepsilon (\lambda^p+a(z_0)\lambda^q)+c\left(\miint{G_{2 \rho,\lambda}\left(z_0\right)} H(z, |Du|)^{\theta}\, dz\right)^{\frac{1}{\theta}}. 
\end{align*}
Finally, using the third condition of \eqref{eq: pq-phase condition} and absorbing $\varepsilon (\lambda^p+a(z_0)\lambda^q)$ on the left-hand side, we complete the proof of this lemma.
\end{proof}
\section{Proof of Theorem \ref{main theorem}} In this section, we complete the proof of Theorem \ref{main theorem}. We begin with the stopping time argument.
\subsection{Stopping time argument}\label{subsec: stopping time}
We prove the existence of $p$- and $(p, q)$-intrinsic cases in \eqref{eq: p-phase condition} and \eqref{eq: pq-phase condition} using the stopping time argument. We set
\begin{align}\label{defn: la_0}
    \lambda^2_0:=\frac{\left(2+||a||_{\infty}\right)}{r^{q-2\gamma}}+1
\end{align}
and 
\begin{align}\label{defn: La_0}
    \Lambda_0=\lambda_0^p+\sup_{Q_{2r}(z_0)}a(z)\lambda_0^q+1.
\end{align}

We define the superlevel sets
\begin{align}\label{Defn E_lambda}
E(\Lambda)=\left\{z \in \Omega_T:\, H\left(z, |Du(z)|\right)>\Lambda\right\}.
\end{align}
For $E(\Lambda),$ and $\rho\in [r,2r]$, we write
\begin{align}\label{defn_E(Lambda, rho)}
E(\Lambda,\rho):=E(\Lambda)\cap Q_{\rho}(z_0)=\left\{z \in Q_{\rho}(z_0):\, H(z, |Du(z)|)>\Lambda)\right\}.
\end{align}
Now, we discuss the stopping time argument. Let $r\leq r_1< r_2\leq 2r.$ For every $\mathfrak{z} \in E(\Lambda,r_1)$ and $K>1$ given by \eqref{defn of K},  we first show the existence of stopping time radius $\rho_{p; \mathfrak z}>0.$ Let
\begin{equation}    \label{eq: definition of Lambda}
    \Lambda >(c_Ec_{\gamma})^{\frac{q}{2}} \left(\frac{4c_v r}{r_2-r_1}\right)^\frac{q(N+2)}{2}\Lambda_0.
\end{equation}
For every $\mathfrak{z} \in E(\Lambda,r_1)$, we take $\lambda_\mathfrak{z}>0$ such that 
\begin{equation}    \label{eq: definition of lambda z}
\Lambda=\lambda_\mathfrak{z}^p+a(\mathfrak{z})\lambda_\mathfrak{z}^q.
\end{equation}
We claim that 
\begin{equation} \label{eq: lambda bigger than 1}
    \lambda_\mathfrak{z}>\sqrt{c_Ec_{\gamma}}\left(\frac{4c_v r}{r_2-r_1}\right)^\frac{N+2}{2}\lambda_0,
\end{equation}
where the constant $c_{\gamma}$ is given by
\begin{align}\label{defn of cgamma}
 c_{\gamma}:=\max\left\{1, [u]^p_{\gamma, \gamma/q}, [u]^q_{\gamma, \gamma/q}, [u]^2_{\gamma, \gamma/q}\right\}.   
\end{align}
Suppose \eqref{eq: lambda bigger than 1} is not true. Then 
\begin{align*}    
\Lambda=\lambda_\mathfrak{z}^p+a(\mathfrak{z})\lambda_\mathfrak{z}^q
&\leq (c_Ec_{\gamma})^{\frac{q}{2}}\left(\frac{4c_v r}{r_2-r_1}\right)^\frac{q(N+2)}{2}(\lambda_0^p+a(\mathfrak{z})\lambda_0^q)\\
&\leq (c_Ec_{\gamma})^{\frac{q}{2}}\left(\frac{4c_v r}{r_2-r_1}\right)^\frac{q(N+2)}{2}\Lambda_0,
\end{align*}
which contradicts \eqref{eq: definition of Lambda}. Thus, for $\tau\in[(r_2-r_1)/(2c_v),r_2-r_1)$, we have 
\begin{align}\label{eq : basic inequality 1 in section 5}
    \miint{Q_{\tau,\lambda_\mathfrak{z}}(\mathfrak{z})} H(z,|Du|)\; dz &\leq \lambda_\mathfrak{z}^{p-2}\left(\frac{2r}{\tau}\right)^{N+2}\miint{Q_{2r}(z_0)} H(z,|Du|)\; dz.
\end{align}
Now we estimate the right hand side of \eqref{eq : basic inequality 1 in section 5} using the energy estimate Lemma \ref{lem: energy inequality}. For $Q_{3r}(z_0) \subset \Omega_T,$ we have
\begin{align*}
  \miint{Q_{2r}(z_0)}H(z, |Du|)\, dz &\leq c_E \miint{Q_{3r}(z_0)} H\left(z, \frac{|u-(u)_{Q_{3r}(z_0)}|}{r}\right)\, dz+c_E \miint{Q_{3r}(z_0)}\frac{|u-(u)_{Q_{3r}(z_0)}|^2}{r^2}\, dz\\
  &=\mathrm{I}+\mathrm{II}.
\end{align*}
\textbf{Estimate of $\mathrm{I}:$} To estimate $\mathrm{I},$ we use the H\"{o}lder continuity of $u.$
\begin{align*}
\mathrm{I}&=\miint{Q_{3r}(z_0)}\left|\frac{u-(u)_{Q_{3r}(z_0)}}{r}\right|^p+a(z)\left|\frac{u-(u)_{Q_{3r}()z_0}}{r}\right|^q\, dz\\
&\leq \frac{\left(\displaystyle{\osc_{Q_{3r}(z_0)} u}\right)^p}{r^p}+||a||_{\infty} \frac{\left(\displaystyle{\osc_{Q_{3r}(z_0)} u}\right)^q}{r^q}\\
&\le \frac{3^p[u]^p_{\gamma, \gamma/q}}{r^{\frac{p(q-2\gamma)}{q}}}+||a||_{\infty}\frac{3^q[u]^q_{\gamma, \gamma/q}}{r^{q-2\gamma}}\leq 3^q\max\left\{[u]^p_{\gamma, \gamma/q}, [u]^q_{\gamma, \gamma/q}\right\} (1+||a||_{\infty})\frac{1}{r^{q-2\gamma}}.
\end{align*}
\textbf{Estimate of $\mathrm{II}:$} Similarly, we estimate
\begin{align*}
\mathrm{II}=\miint{Q_{3r}(z_0)}\frac{|u-(u)_{Q_{3r}(z_0)}|^2}{r^2}\, dz \leq \frac{3^2[u]^2_{\gamma, \gamma/q}}{r^{\frac{2(q-2\gamma)}{q}}} \leq \frac{3^2[u]^2_{\gamma, \gamma/q}}{r^{q-2\gamma}}.     
\end{align*}
Thus, recalling the constant $c_{\gamma}$ from \eqref{defn of cgamma}, we get
\begin{align}\label{eq: estimate of H by lambda_0}
    \miint{Q_{2r}(z_0)}H(z, Du)\, dz\leq c_E c_{\gamma} (2+||a||_{\infty})\frac{1}{r^{q-2\gamma}}.
\end{align}
Therefore, using the above estimate in \eqref{eq : basic inequality 1 in section 5}, we obtain,
\begin{align*}
 \miint{Q_{\tau,\lambda_\mathfrak{z}}(\mathfrak{z})} H(z,|Du|)\; dz &\leq \lambda_\mathfrak{z}^{p-2}\left(\frac{2r}{\tau}\right)^{N+2}\miint{Q_{2r}(z_0)} H_(z,|Du|)\; dz\nonumber\\
 &\leq  \lambda_\mathfrak{z}^{p-2}\left(\frac{2r}{\tau}\right)^{N+2} c_Ec_{\gamma}(2+||a||_{\infty})\frac{1}{r^{q-2\gamma}}\nonumber\\
 &\leq c_Ec_{\gamma}\left(\frac{4c_vr}{r_2-r_1}\right)^{N+2}\lambda^{p-2}_{\mathfrak{z}}\lambda^2_0< \lambda^p_{\mathfrak{z}}.
\end{align*}
The facts $\mathfrak{z}\in E(\Lambda,r_1)$ and \eqref{eq: definition of lambda z} imply that $\mathfrak{z}\in E(\lambda_\mathfrak{z}^p,r_1)$. We define a continuous function $g: (0, r_2-r_1]\to \mathbb{R}$ as
\begin{align*}
  g (\tau)=  \miint{Q_{\tau,\lambda_\mathfrak{z}}(\mathfrak{z})} H(z,|Du|)\; dz.
\end{align*}
By the Lebesgue differentiation theorem, we have 
\begin{align*}
    \lim_{\tau \to 0}g(\tau)=H(\mathfrak z, |Du(\mathfrak z)|)>\lambda^p_{\mathfrak z}.
\end{align*}
Therefore, by the intermediate value property of continuous function, there exists $\rho_{p;\mathfrak{z}}\in (0,(r_2-r_1)/(2c_v))$ such that 
\begin{equation}    \label{eq: choice condition 1 of rho}
    \miint{Q_{\rho_{p;\mathfrak{z}},\lambda_\mathfrak{z}}(\mathfrak{z})} H(z,|Du|)\; dz=\lambda_\mathfrak{z}^p
\end{equation}
and
\begin{equation}    \label{eq: choice condition 2 of rho}
    \miint{Q_{\tau,\lambda_\mathfrak{z}}(\mathfrak{z})} H(z,|Du|)\; dz< \lambda_\mathfrak{z}^p
\end{equation}
for every $\tau\in(\rho_{p;\mathfrak{z}},r_2-r_1)$. If $\inf_{Q_{\rho_{p; \mathfrak z}, \lambda_{\mathfrak z}}(\mathfrak z)}a(z)\leq\widetilde M\rho^{\kappa}_{p; \mathfrak z},$ together with \eqref{eq: choice condition 2 of rho} and \eqref{eq: choice condition 1 of rho}, we get the second and third conditions of \eqref{eq: p-phase condition}. 

We note that if $\inf_{Q_{\rho_{p; \mathfrak z}, \lambda_{\mathfrak z}}(\mathfrak z)}a( z)\leq \widetilde M\rho^{\kappa}_{p; \mathfrak z},$ by Lemma \ref{lem: lambda rho relation}, $\lambda\leq K \rho^{\frac{2\gamma -q}{q}}.$ Now using this relation, for every $\mathfrak z \in E(\Lambda, r_1,),$ we have $\lambda_{\mathfrak z}$ that satisfies $K^2\lambda^p_{\mathfrak z}\geq \inf_{Q_{\rho_{p; \mathfrak z}, \lambda_{\mathfrak z}}(\mathfrak z)}a( z)\lambda^q_{\mathfrak z}.$ Indeed, we have
\begin{align*}
 \inf_{Q_{\rho_{p; \mathfrak z}, \lambda_{\mathfrak z}}(\mathfrak z)}a( z)\leq \widetilde M\rho^{\kappa}_{p; \mathfrak z}\leq K\widetilde M \lambda^{-\frac{q\kappa}{q-2\gamma}}_{\mathfrak z} \leq K^2 \lambda^{p-q}_{\mathfrak z}.
\end{align*}
This completes the existence of the $p$-intrinsic assumption in \eqref{eq: p-phase condition}.

First, we show the second and third assumption of \eqref{eq: pq-phase condition}. For every $\tau \in (\rho_{p; \mathfrak z}, r_2-r_1),$ using \eqref{eq: choice condition 2 of rho}, we get 
\begin{align*}\label{eq : basic inequality 1 in section 5}
    \miint{G_{\tau,\lambda_\mathfrak{z}}(\mathfrak{z})} H(z,|Du|)\; dz \leq \frac{\lambda^p_{\mathfrak z}+a(\mathfrak z)\lambda^q_{\mathfrak z}}{\lambda^p_{\mathfrak z}}\miint{Q_{\tau, \lambda_{\mathfrak z}}(\mathfrak z)} H(z, |Du|)\, dz< \lambda^p_{\mathfrak z}+(\mathfrak z)\lambda^q_{\mathfrak z}.
\end{align*}
Since $\mathfrak{z}\in E(\Lambda,r_1),$ by Lebesgue differentiation theorem and intermediate value theorem,  there exists $\rho_{p, q; \mathfrak z}\in (0, \rho_{p; \mathfrak z}]$ such that 
\begin{equation*}    \label{eq: choice condition 1 of pq phase}
    \miint{G_{\rho_{p,q;\mathfrak{z}},\lambda_\mathfrak{z}}(\mathfrak{z})} H(z,|Du|)\; dz=\lambda^p_{\mathfrak z}+a(\mathfrak{z})\lambda_\mathfrak{z}^q
\end{equation*}
and
\begin{equation*}    \label{eq: choice condition 2 of pq phase}
    \miint{G_{\tau,\lambda_\mathfrak{z}}(\mathfrak{z})} H(z,|Du|)\; dz< \lambda^p_{\mathfrak z}+a(\mathfrak{z})\lambda_\mathfrak{z}^q
\end{equation*}
for every $\tau\in(\rho_{p,q; \mathfrak{z}},r_2-r_1).$ 

If $\rho_{p; \mathfrak z}$ satisfies $\inf_{Q_{\rho_{p; \mathfrak z}, \lambda_{\mathfrak z}}(\mathfrak z)}a(z)\geq \widetilde M\rho^{\kappa}_{p; \mathfrak z}$, we note that for any $z \in G_{4\rho_{p,q; \mathfrak z}, \lambda_{\mathfrak z}}(\mathfrak z),$ we have 
\begin{align*}
    a(z)\leq c_a \left(\inf_{G_{4\rho_{p,q; \mathfrak z}}, \lambda_{\mathfrak z}(\mathfrak z)}a(z)+(4\rho_{p,q;\mathfrak z})^{\kappa}\right)\leq c_a(a(\mathfrak z)+ 4^{\kappa}\widetilde M^{-1}a(\mathfrak z))\leq 2c_a a(\mathfrak z),
\end{align*}
where we used $a\in \mathcal{Z}^{\kappa}(\Omega_T)$ and $$a(\mathfrak z)\geq \inf_{G_{\rho_{p,q; \mathfrak z}, \lambda_{\mathfrak z}}(\mathfrak z)}a( z)\geq \inf_{Q_{\rho_{p; \mathfrak z}, \lambda_{\mathfrak z}}(\mathfrak z)}a(z)\geq  \widetilde M\rho^{\kappa}_{p; \mathfrak z}\geq \widetilde M \rho^{\kappa}_{p,q; \mathfrak z}.$$ 
On the other hand, we have
\begin{align*}
  a(\mathfrak z)\leq c_a\left(a(z)+(4\rho_{p,q;\mathfrak z})^\kappa\right)
  \leq c_a a(z)
      +c_a4^\kappa\widetilde M^{-1}a(\mathfrak z)
  &\leq c_a a(z)+\frac12a(\mathfrak z),
\end{align*}
where the last inequality follows from
$\widetilde M\geq2c_a4^\kappa$ (see \eqref{definition of M-tilde}). Absorbing the last term on the left-hand side, we conclude that
\[
  \frac{a(\mathfrak z)}{2c_a}
  \leq a(z)\leq2c_a a(\mathfrak z)
  \qquad\text{for every }
  z\in G_{4\rho_{p,q;\mathfrak z},\lambda_{\mathfrak z}}(\mathfrak z).
\]
This proves the first assumption of \eqref{eq: pq-phase condition}. Thus we conclude the proof of $(p, q)$-intrinsic case \eqref{eq: pq-phase condition}. 

We note that, when $\inf_{Q_{\rho_{p; \mathfrak z}, \lambda_{\mathfrak z}}(\mathfrak z)}a(z)\geq M\rho^{\kappa}_{p, \mathfrak{z}},$ $\lambda_{\mathfrak z}$ satisfies either $K^2\lambda_{\mathfrak z}^{p}\le a(\mathfrak z)\lambda_{\mathfrak z}^{q}$ or $K^2\lambda_{\mathfrak z}^{p}\ge a(\mathfrak z)\lambda_{\mathfrak z}^{q}.$  However, we show that $K^2\lambda_{\mathfrak{z}}^p \leq a(\mathfrak{z})\lambda_{\mathfrak{z}}^q,\,\, \text{and}\,\,\,  \inf_{Q_{\rho_{p; \mathfrak z}, \lambda_{\mathfrak z}}(\mathfrak z)}a(z)\leq M\rho^{\kappa}_{p, \mathfrak{z}},$  cannot occur together. Indeed, this gives
\begin{align}\label{eq: the case that cannot occur}
    K^2\lambda^p_{\mathfrak{z}} \leq a(\mathfrak{z})\lambda^q_{\mathfrak{z}} \leq c_a\left(\inf_{Q_{\rho_{p; \mathfrak z}, \lambda_{\mathfrak z}}(\mathfrak z)}a(z)+\rho^{\kappa}_{p, \mathfrak z}\right)\lambda^q_{\mathfrak z}\leq 2c_a \widetilde M\rho^{\kappa}_{p; \mathfrak{z}}\lambda^q_{\mathfrak{z}}.
\end{align}
From the definition of $K$ given in \eqref{defn of K}, we have
\begin{align}\label{eq: from defn of K}
    \frac{K}{2}>\Bigg[8^{q+1}\max\Bigg\{&
    c_E[u]^p_{\gamma,\gamma/q},
    2
    c_a\widetilde M c_E[u]^q_{\gamma,\gamma/q},(2c_E)^{p/2}
    [u]^p_{\gamma,\gamma/q}
    \Bigg\}\Bigg]^{\frac{q\kappa}{p(q-2\gamma)}}.
\end{align}
Using the inequality \eqref{eq: from defn of K} in Lemma \ref{lem: lambda rho relation}, we get
\begin{align}\label{eq: new lambda rho relation}
    \lambda_{\mathfrak{z}}< \left(\frac{K}{2}\right)^{\frac{q-2\gamma}{q\kappa}}\rho^{\frac{2\gamma-q}{q}}_{p, \mathfrak{z}}\implies \rho^{\kappa}_{p, \mathfrak{z}}< \left(\frac{K}{2}\right)\lambda^{-\frac{q\kappa}{q-2\gamma}}_{\mathfrak z}.
\end{align}
Now using \eqref{eq: new lambda rho relation} in \eqref{eq: the case that cannot occur} and the gap relation $q\leq p+\frac{q\kappa}{q-2\gamma},$ we derive the following.
\begin{align*}
    K^2\leq 2c_a \widetilde M\rho^{\kappa}_{p, \mathfrak{z}}\lambda^{q-p}_{\mathfrak{z}}< 2c_a\widetilde M\left(\frac{K}{2}\right)\lambda^{q-p-\frac{q\kappa}{q-2\gamma}}_{\mathfrak{z}}\leq c_a\widetilde MK.
\end{align*}
This implies $K<c_a \widetilde M,$ which is a contradiction to the definition of $K$ (see \eqref{defn of K}).

\subsection{Vitali covering lemma} To complete the proof of Theorem \ref{main theorem}, here we prove a Vitali covering lemma for the following parabolic intrinsic cylinders: For each $\mathfrak{z}\in E(\Lambda,r_1)$, we denote the intrinsic cylinder of $\mathfrak{z}$ by
\begin{align*}
\displaystyle
\mathcal{Q}(\mathfrak{z})=\left\{\begin{array}{l}
Q_{2\rho_{p;\mathfrak{z}},\lambda_\mathfrak{z}}(\mathfrak{z}) \quad\,\text{in $p$-intrinsic case},\\
G_{2\rho_{p,q; \mathfrak{z}},\lambda_\mathfrak{z}}(\mathfrak{z}) \quad\,\text{in $(p, q)$-intrinsic case}.
\end{array}\right.
\end{align*}
Let us define the family of cylinders and the radii as
\begin{align*}
\mathcal{F}=\left\{\mathcal{Q}(\mathfrak{z})\, : \,\mathfrak{z}\in E(\Lambda,r_1)\right\}\quad \text{and} \quad \ell(\mathcal{Q}(\mathfrak{z}))=\left\{
    \begin{array}{l }
        2\rho_{p;\mathfrak{z}} \quad\,\,\,\,\,\, \text{in $p$-intrinsic case},\\
        2\rho_{p,q;\mathfrak{z}} \quad \text{in $(p,q)$-intrinsic case}.
    \end{array}
\right.
\end{align*} 
Moreover, for $\mathfrak{z}, \mathfrak{w}\in E(\Lambda, r_1),$ we set
\begin{align*}
    \Lambda=\lambda^p_{i}+a(i)\lambda^q_{i}, \quad \text{where} \ \, i\in \left\{\mathfrak{z}, \mathfrak{w}\right\}.
\end{align*}
We note from the previous section that $\ell(\mathcal{Q}(\mathfrak{z}))\in (0, R)$ by setting $R:=\frac{r_2-r_1}{c_v}$, where $c_v$ will be chosen later. For $j \in  \NN,$ we consider the following subcollection of $\mathcal{F},$
\begin{align*}
  \mathcal{F}_j:=\left\{\mathcal{Q}(\mathfrak{z})\in \mathcal{F} : \frac{R}{2^j}\leq \ell({\mathcal{Q}(\mathfrak{z})})<\frac{R}{2^{j-1}}\right\}.
\end{align*}
We construct disjoint subcollections $\mathcal{I}_j \subset \mathcal{F}_j$ for $j \in \NN$ as follows. Let $\mathcal{I}_1$ be the maximal disjoint  subcollection in $\mathcal{F}_1.$ We note from the following
\begin{align*}
    \lim_{\La\to \infty}\La|E(\Lambda)|\leq \lim_{\La \to \infty}\iint_{E(\Lambda)}H(z, |D u|)\, dz=0
\end{align*}
that $|E(\Lambda)|<\infty.$ 
Moreover, in the same way as in \eqref{eq : basic inequality 1 in section 5}-\eqref{eq: choice condition 1 of rho} and using \eqref{defn: la_0} we have 
\begin{equation}    \label{eq : relation between lambda_z and lambda_0}
    \lambda_\mathfrak{z}\leq \sqrt{c_{E}c_{\gamma}}\left(\frac{2r}{\rho_\mathfrak{z}}\right)^{\frac{N+2}{2}}\lambda_0.
\end{equation}
Using $|E(\La, r_1)|\leq |E(\La)|< \infty$ and \eqref{eq : relation between lambda_z and lambda_0}, we conclude that the number of cylinders in $\mathcal{I}_1$ is finite. Suppose that we have already chosen $\mathcal{I}_j \subset \mathcal{F}_j$ for $j=1,..,k-1.$ Then we construct $\mathcal{I}_k$ as
\begin{align*}
    \mathcal{I}_k=\left\{\mathcal{Q}(\mathfrak{z})\in \mcf_k \, : \, \mathcal{Q}(\mathfrak{z})\cap \mathcal{Q}(\mathfrak{w})=\emptyset\,\,\, \text{for every}\,\,\, \mathcal{Q}(\mathfrak{w}) \in \bigcup_{j=1}^{k-1}\mathcal{I}_j \right\}. 
\end{align*}
Therefore, 
$$
\mathcal{I}=\bigcup_{j=1}^{\infty}\mathcal{I}_j
$$
would be the maximal collection of pairwise disjoint cylinders in $\mcf.$ Now we state the precise theorem which we prove in this section.

In what follows, we set $\Gamma$ depending on $c_a, p, q, \kappa, K$ as 
\begin{align*}
   \Gamma:= 4^{\frac{1}{p}}+\left(c_a6^{\kappa+1}\max\{2^{\kappa}, 2\}K\right)^{\frac{1}{p}}.
\end{align*}
\begin{lemma}[Vitali covering lemma]\label{vitali lemma}
 Let $E(\Lambda, r_1)$ be given by \eqref{defn_E(Lambda, rho)} and $$c_v:= 16\Gamma \sqrt{(\Gamma^p+c_aK^2 \Gamma^q)+1}.$$ Then there exists a collection $\{\mathcal{Q}({z_i})\}_{i\in \NN}$ of cylinders that satisfies the following.
\begin{itemize}
\item [(i)] $E(\La, r_1)\subseteq \cup_{i\in\NN}\,\,c_v \mathcal{Q}({z_i})$
\item[(ii)] $\mathcal{Q}({z_i})\cap \mathcal{Q}({z_j})= \emptyset$ for every $i,j\in \NN$ with $i\ne j$.
\end{itemize}   
\end{lemma}
To prove Vitali's covering lemma, we need to show:
\begin{itemize}
\item[(i)] For any $\mathcal{Q}(\mathfrak{z}) \in \mcf,$ there exists $\mathcal{Q}(\mathfrak{w}) \in \mathcal{I}$ such that $\mathcal{Q}(\mathfrak{z}) \cap \mathcal{Q}(\mathfrak{w}) \neq \emptyset.$
\item[(ii)] There exists a universal constant $c_v>1$ such that $\mathcal{Q}(\mathfrak{z}) \subset c_v \mathcal{Q}(\mathfrak{w})$ where 
$$
c_v\mathcal{Q}(\mathfrak{w})=\left\{
    \begin{array}{l c}
        Q_{c_v(2\rho_{p;\mathfrak{w}}),\lambda_\mathfrak{w}}(\mathfrak{w})\,\,\,\,\,\, \quad \text{in $p$-intrinsic case},\\
        G_{c_v(2\rho_{p,q;\mathfrak{w}}),\lambda_\mathfrak{w}}(\mathfrak{w}) \quad \text{in $(p,q)$-intrinsic case}.
    \end{array}
    \right.
$$
\end{itemize}
To show the first assertion, fix some $Q(\mathfrak{z}) \in \mcf.$ Then $Q(\mathfrak{z}) \in \mcf_j$ for some $j \in \NN.$ Using the maximality of $\mathcal{I}_j,$ we find that there exists $\mathcal{Q}(\mathfrak{w})\in \cup_{i=1}^{j}\mathcal{I}_i$ such that $\mathcal{Q}(\mathfrak{z})\cap \mathcal{Q}(\mathfrak{w})\neq \emptyset.$ 

To show the second assertion, we first observe that for any $\mathcal{Q}(\mathfrak{w}) \in \mathcal{I}_i,$ with $i\leq j,$ we have
\begin{align}\label{radius compare}
    \ell(\mathcal{Q}(\mathfrak{z}))\leq 2\ell(\mathcal{Q}(\mathfrak{w})).
\end{align}
This is due to the fact that
$\frac{R}{2^j}\leq \frac{R}{2^i}<\ell(\mathcal Q(\mathfrak w)) \leq \frac{R}{2^{i-1}}$ for $i\leq j$, and $\frac{R}{2^j}<\ell(\mathcal Q(\mathfrak z)) \leq \frac{R}{2^{j-1}}$. 

In the rest of the subsection, we show the existence of the Vitali constant $c_v>1$ such that $\mathcal{Q}(\mathfrak{z}) \subset c_v \mathcal{Q}(\mathfrak{w})$ by considering the following four distinct cases.
\begin{itemize}
    \item [(i)] $\mathcal{Q}(\mathfrak z)=Q_{2\rho_{p; \mathfrak z}, \lambda_{\mathfrak z}}(\mathfrak z)$ and $\mathcal{Q}(\mathfrak w)=Q_{2\rho_{p; \mathfrak w}, \lambda_\mathfrak w}(\mathfrak w),$
    \item [(ii)] $\mathcal{Q}(\mathfrak z)=Q_{2\rho_{p; \mathfrak z}, \lambda_{\mathfrak z}}(\mathfrak z)$ and $\mathcal{Q}(\mathfrak w)=G_{2\rho_{p,q; \mathfrak w}, \lambda_\mathfrak w}(\mathfrak w),$
    \item [(iii)] $\mathcal{Q}(\mathfrak z)=G_{2\rho_{p,q; \mathfrak z}, \lambda_{\mathfrak z}}(\mathfrak z)$ and $\mathcal{Q}(\mathfrak w)=G_{2\rho_{p,q; \mathfrak w}, \lambda_\mathfrak w}(\mathfrak w),$
    \item [(iv)] $\mathcal{Q}(\mathfrak z)=G_{2\rho_{p,q; \mathfrak z}, \lambda_{\mathfrak z}}(\mathfrak z)$ and $\mathcal{Q}(\mathfrak w)=Q_{2\rho_{p; \mathfrak w}, \lambda_\mathfrak w}(\mathfrak w).$
\end{itemize}
Before we proceed to the proof of Lemma \ref{vitali lemma}, we show the comparability of the scaling factors $\lambda_{i}$, for $i \in \{\mathfrak z, \mathfrak w\}$.
\begin{lemma}
Let $\mathcal{Q}(\mathfrak z), \mathcal{Q}(\mathfrak w)\in \mathcal{F}$ with $\mathcal{Q}(\mathfrak z)\cap \mathcal{Q}(\mathfrak w)\neq \emptyset.$ Moreover, if $\ell(\mathcal{Q}(\mathfrak z))\leq 2 \ell(\mathcal{Q}(\mathfrak w)),$ then $\lambda_{\mathfrak z}$ and $\lambda_{\mathfrak w}$ are comparable, i.e., there exists a constant $\Gamma$ depending on $c_a, p, q, \kappa$ and $K$ such that
\begin{align*}
    \Gamma^{-1}\lambda_{\mathfrak z} \leq \lambda_{\mathfrak w}\leq \Gamma \lambda_{\mathfrak z}.
\end{align*}
\end{lemma}
\begin{proof}
We consider case (i) above, i.e., $\mathcal{Q}(\mathfrak z)=Q_{2\rho_{p;\mathfrak z}, \lambda_{\mathfrak z}}(\mathfrak z)$ and $\mathcal{Q}(\mathfrak w)=Q_{2\rho_{p;\mathfrak w}, \lambda_\mathfrak w}(\mathfrak w).$ The other cases can be handled following the similar computations. We shall show that 
\begin{align*}
\lambda_{\mathfrak{w}}\leq \left[4^{\frac{1}{p}}+\left(c_a6^{\kappa+1}\max\{2^{\kappa}, 2\}K\right)^{\frac{1}{p}}\right]\lambda_{\mathfrak{z}}=\Gamma(c_a, p, q, \kappa, K)\lambda_{\mathfrak z}.
\end{align*}
Using the assumption $Q_{2\rho_{p;\mathfrak z}, \lambda_{\mathfrak z}}(\mathfrak z) \cap Q_{2\rho_{p;\mathfrak w}, \lambda_\mathfrak w}(\mathfrak w) \neq \emptyset,$ $ \ell(\mathcal{Q}(\mathfrak z))\leq 2 \ell(\mathcal{Q}(\mathfrak w)),$ and $a \in \mathcal{Z}^{\kappa}(\Omega_T),$ we get
\begin{align}\label{eq: inq in the proof of comaparison of lambda}
    a(\mathfrak z)\leq c_a \left(a(\mathfrak w)+\max\{1, 2^{\kappa-1}\}(2^{\kappa}\rho^{\kappa}_{p;\mathfrak z}+2^{\kappa}\rho^{\kappa}_{p;\mathfrak w})\right)\leq c_a a(\mathfrak w)+c_a 6^{\kappa}\max\{1, 2^{\kappa-1}\}\rho^{\kappa}_{p; \mathfrak w}
\end{align}
for any $\kappa \in (0, \infty).$
In contrast, we assume that $\lambda_{\mathfrak{w}}>\left[4^{\frac{1}{p}}+\left(c_a6^{\kappa+1}\max\{2^{\kappa}, 2\}K\right)^{\frac{1}{p}}\right]\lambda_{\mathfrak{z}}.$ Using this counter assumption and  \eqref{eq: inq in the proof of comaparison of lambda}, we obtain 
\begin{align*}
\Lambda=\lambda^p_{\mathfrak{z}}+a(\mathfrak{z})\lambda^q_{\mathfrak{z}} &\leq \lambda^p_{\mathfrak z} + \left(c_a a(\mathfrak w)+c_a 6^{\kappa}\max\{2^{\kappa-1}, 1\}\rho^{\kappa}_{p; \mathfrak w}\right)\lambda^q_{\mathfrak z} \\
&<\frac{1}{4}\lambda^p_{\mathfrak w}+\frac{a(\mathfrak w)\lambda^q_{\mathfrak w}}{\max6^{\kappa+1}\{2^{\kappa}, 2\}K}+ \frac{c_a6^{\kappa}\max\{2^{\kappa-1}, 1\}}{c_a 6^{\kappa+1}\max\{2^{\kappa}, 2\}K} \rho^{\kappa}_{p; \mathfrak w} \lambda^q_{\mathfrak w}\\
&\leq \frac{1}{4}\lambda^p_{\mathfrak w}+\frac{1}{2}a(\mathfrak w)\lambda^q_{\mathfrak w}+ \frac{1}{12K}\frac{K}{2}\lambda^p_{\mathfrak w}\\
&<\left(\frac{1}{4}+\frac{1}{24}\right)\lambda^p_{\mathfrak w}+\frac{1}{2}a(\mathfrak w)\lambda^q_{\mathfrak w}\\
&< \frac{1}{2}\left(\lambda^p_{\mathfrak w}+a(\mathfrak w)\lambda^q_{\mathfrak w}\right)=\frac{1}{2}\Lambda,
\end{align*}
which is a contradiction.

Next, we show that
$\lambda_{\mathfrak z}\leq\Gamma\lambda_{\mathfrak w}$.
Using $a\in\mathcal{Z}^{\kappa}(\Omega_T)$ in the reverse
direction, we have
\[
    a(\mathfrak w)
    \leq c_a a(\mathfrak z)
        +c_a6^\kappa\max\{1,2^{\kappa-1}\}
            \rho_{p;\mathfrak w}^{\kappa}.
\]
Suppose, to the contrary, that $\lambda_{\mathfrak z}>\Gamma\lambda_{\mathfrak w}$.
Using the above inequality, we obtain
\begin{align*}
    \Lambda
    =\lambda_{\mathfrak w}^{p}
        +a(\mathfrak w)\lambda_{\mathfrak w}^{q}
    &\leq
    \lambda_{\mathfrak w}^{p}
    +c_a a(\mathfrak z)\lambda_{\mathfrak w}^{q}
    +c_a6^\kappa\max\{1,2^{\kappa-1}\}
        \rho_{p;\mathfrak w}^{\kappa}
        \lambda_{\mathfrak w}^{q}
    \\
    &\leq
    \left(
        1+\frac{c_a6^\kappa
            \max\{1,2^{\kappa-1}\}K}{2}
    \right)\lambda_{\mathfrak w}^{p}
    +c_a a(\mathfrak z)\lambda_{\mathfrak w}^{q}
    \\
    &<
    \left(
        \frac1{\Gamma^p}
        +\frac{c_a6^\kappa
            \max\{1,2^{\kappa-1}\}K}{2\Gamma^p}
    \right)\lambda_{\mathfrak z}^{p}
    +\frac{c_a}{\Gamma^q}
        a(\mathfrak z)\lambda_{\mathfrak z}^{q}
    \\
    &<
    \left(\frac14+\frac1{24}\right)
        \lambda_{\mathfrak z}^{p}
    +\frac12a(\mathfrak z)\lambda_{\mathfrak z}^{q}
    \\
    &<
    \frac12\left(
        \lambda_{\mathfrak z}^{p}
        +a(\mathfrak z)\lambda_{\mathfrak z}^{q}
    \right)
    =\frac12\Lambda,
\end{align*}
which is again a contradiction. Hence,
\[
    \Gamma^{-1}\lambda_{\mathfrak z}
    \leq\lambda_{\mathfrak w}
    \leq\Gamma\lambda_{\mathfrak z}.
\] This completes the proof.
\end{proof}
\begin{proof}[Proof of Lemma \ref{vitali lemma}] As mentioned before, it remains to find the  Vitali constant $c_v$ such that $\mathcal{Q}(\mathfrak z)\subset c_v \mathcal{Q}(\mathfrak w)$ considering the different cases (i)-(iv) described above. We begin with the space inclusion.  We only show case when $\mathcal{Q}(\mathfrak z)=Q_{2\rho_{\mathfrak z}, \lambda_{\mathfrak z}}(\mathfrak z),$ and $\mathcal{Q}(\mathfrak w)=Q_{2\rho_{\mathfrak w}, \lambda_\mathfrak w}(\mathfrak w),$ and the other cases follow similarly. Let $\mathfrak{z}=(x_{\mathfrak{z}}, t_{\mathfrak{z}})$ and $\mathfrak{w}=(x_{\mathfrak{w}}, t_{\mathfrak{w}})$ and we have $B_{2\rho_{\mathfrak{z}}}(x_{\mathfrak{z}})\cap B_{2\rho_{\mathfrak{w}}}(x_{\mathfrak{w}})\neq \emptyset.$ For $x \in B_{2\rho_{\mathfrak{z}}}(x_{\mathfrak{z}}),$ we have
\begin{align*}
    |x-x_{\mathfrak{w}}|\leq |x-x_{\mathfrak{z}}|+|x_{\mathfrak{z}}-x_{\mathfrak{w}}|\leq 2\rho_{\mathfrak{z}}+2\rho_{\mathfrak{z}}+2\rho_{\mathfrak{w}}\overset{\eqref{radius compare}}{\leq} 10\rho_{\mathfrak{w}}\leq c_v (2\rho_{\mathfrak{w}}). 
\end{align*}
Choosing $c_{v}\geq 5,$ we have the conclusion that $x\in B_{c_{\nu}(2\rho_{\mathfrak{w}})}(x_{\mathfrak w}).$

Next we show the time inclusion part. In the first case, we have
\begin{align*}
    I_{2\rho_{\mathfrak{z}}, \lambda_{\mathfrak{z}}}(t_{\mathfrak{z}})\cap I_{2\rho_{\mathfrak{w}}, \lambda_{\mathfrak{w}}}(t_{\mathfrak{w}})\neq \emptyset.
\end{align*}
Let $t \in I_{2\rho_{\mathfrak z}, \lambda_{\mathfrak z}}(t_{\mathfrak z})$ and we obtain
\begin{align*}
 |t-t_{\mathfrak{w}}|\leq |t-t_{\mathfrak{z}}|+|t_{\mathfrak{z}}-t_{\mathfrak{w}}|&\leq \frac{2\lambda^{2}_{\mathfrak{z}}(2\rho_{\mathfrak{z}})^2}{\lambda^p_{\mathfrak z}}+\frac{\lambda^{2}_{\mathfrak{w}}(2\rho_{\mathfrak{w}})^2}{\lambda^p_{\mathfrak w}}\\ &\leq \Gamma^{p-2}32 \rho^2_{\mathfrak w}\lambda^{2-p}_{\mathfrak w}+\lambda^{2-p}_{\mathfrak w}(2 \rho_{\mathfrak w})^2\leq \lambda^{2-p}_{\mathfrak w}(c_v 2 \rho_{\mathfrak w})^2
\end{align*}
provided $c_v \geq \sqrt{8\Gamma^{p-2}+1}.$

When $I_{2\rho_{\mathfrak z}, \lambda_{\mathfrak z}}(t_{\mathfrak z})\cap J_{2 \rho_{\mathfrak w}, \lambda_{\mathfrak w}}(t_{\mathfrak w})\neq \emptyset,$ and $t \in I_{2\rho_{\mathfrak z}, \lambda_{\mathfrak z}}(t_{\mathfrak z}),$ we have
\begin{align} \label{eq: 5.23}
|t-t_{\mathfrak{w}}|\leq |t-t_{\mathfrak{z}}|+|t_{\mathfrak{z}}-t_{\mathfrak{w}}|\leq 2 \lambda^{2-p}_{\mathfrak z}(2\rho_{\mathfrak z})^2+\frac{\la^2_{\mathfrak w}}{\La}(2\rho_{\mathfrak w})^2.
\end{align}
On the other hand, we note that
\begin{align*}
    \lambda^p_{\fw}+a(\fw)\lambda^q_{\fw} \leq \lambda^p_{\fw}+c_a\left(\inf_{Q_{2\rho_{\fz}, \lambda_{\fw}}(\fz)}a(z)+\rho^{\kappa}\right)\lambda^q_{\fw}\leq \lambda^p_{\fw}+ c_a K^2 \lambda^{p-q}_{\fz}\lambda^q_{\fw}\leq \left(\Gamma^p+c_aK^2 \Gamma^q\right)\lambda^p_{\fz}
\end{align*}
Using the above estimate in \eqref{eq: 5.23}, we get
\begin{align*}
|t-t_{\mathfrak{w}}| \leq \left[8\Gamma^2(\Gamma^p+c_aK^2 \Gamma^q)+1\right]\frac{\lambda^2_{\fw}}{\La}(2\rho_{\fw})^2 \leq \frac{\lambda^2_{\fw}}{\La}(c_v 2\rho_{\fw})^2    
\end{align*}
provided $c_v\geq \sqrt{8\Gamma^2(\Gamma^p+c_aK^2 \Gamma^q)+1}.$

When $J_{2\rho_{\mathfrak z}, \lambda_{\mathfrak z}}(t_{\mathfrak z})\cap I_{2 \rho_{\mathfrak w}, \lambda_{\mathfrak w}}(t_{\mathfrak w})\neq \emptyset,$ and $t \in J_{2\rho_{\mathfrak z}, \lambda_{\mathfrak z}}(t_{\mathfrak z}),$ we have
\begin{align*}
|t-t_{\mathfrak{w}}|\leq |t-t_{\mathfrak{z}}|+|t_{\mathfrak{z}}-t_{\mathfrak{w}}|\leq 2 \frac{\lambda^2_{\fz}}{\La}(2\rho_{\mathfrak z})^2+\la^{2-p}_{\mathfrak w}(2\rho_{\mathfrak w})^2 \leq (8\Gamma^{p-2}+1)\lambda^{2-p}_{\fw}(2\rho_{\fw})^2\leq \lambda^{2-p}_{\fw}(c_v 2\rho_{\fw})^2
\end{align*}
provided $c_v \geq \sqrt{8\Gamma^{p-2}+1}.$

Finally, when $J_{2\rho_{\mathfrak z}, \lambda_{\mathfrak z}}(t_{\mathfrak z})\cap J_{2 \rho_{\mathfrak w}, \lambda_{\mathfrak w}}(t_{\mathfrak w})\neq \emptyset,$ and $t \in J_{2\rho_{\mathfrak z}, \lambda_{\mathfrak z}}(t_{\mathfrak z}),$ we have
\begin{align*}
|t-t_{\mathfrak{w}}|\leq |t-t_{\mathfrak{z}}|+|t_{\mathfrak{z}}-t_{\mathfrak{w}}|\leq 2 \frac{\lambda^2_{\fz}}{\La}(2\rho_{\mathfrak z})^2+\frac{\la^{2}_{\mathfrak w}}{\La}(2\rho_{\mathfrak w})^2 \leq (8\Gamma^2+1)\frac{\lambda^2_{\fw}}{\La}(2\rho_{\fw})^2\leq \frac{\lambda^2_{\fw}}{\La}(c_v2\rho_{\fw})^2
\end{align*}
provided $c_v \geq \sqrt{8\Gamma^2+1}.$

Now we choose $c_v = 16\Gamma \sqrt{(\Gamma^p+c_aK^2 \Gamma^q)+1}$ and this completes the proof.
\end{proof}
\subsection{Proof of Theorem \ref{main theorem} (completed)} Now we are ready to complete the proof of Theorem \ref{main theorem}. We begin with a result that is a consequence of reverse H\"{o}lder inequalities. There exist constants $c (\texttt{data})\geq 1$ and $\theta_0\in (0, 1)$ such that for all $\theta \in (\theta_0, 1)$ we have
\begin{align*}
\iint_{Q_{2c_v\rho, \lambda}(z_0)}H(z, |Du|)\, dz &\leq c\Lambda^{1-\theta}\iint_{Q_{2\rho, \lambda}(z_0) \cap E\left(\frac{\Lambda}{c}\right)} H(z, |D u|)^{\theta}\, dz
\end{align*}   
and 
\begin{align*}
  \iint_{G_{2c_v\rho, \lambda}(z_0)}H(z, |Du|)\, dz &\leq c\Lambda^{1-\theta}\iint_{G_{2\rho, \lambda}(z_0) \cap E\left(\frac{\Lambda}{c}\right)} H(z, |D u
  |)^{\theta}\, dz.
\end{align*}
Using Lemma \ref{vitali lemma}, we get a countable pairwise disjoint collection $\mathcal{I}:=\left\{\mathcal{Q}({z_j})\right\}_{j=1}^{\infty}$ for $z_j \in E\left({\Lambda, r_1}\right).$ From the above estimates, there exist constants $c \geq 1$ and $\theta_0\in (0,1)$ such that
\begin{align*}
    \iint_{c_v \mathcal{Q}(z_j)} H(z, |D u|)\, dz \leq c\Lambda^{1-\theta}\iint_{\mathcal{Q}(z_j) \cap E\left({\frac{\Lambda}{c}}\right)}H(z, |Du|)^{\theta}\, dz
\end{align*}
holds for every $j \in \NN$ and $\theta \in (\theta_0, 1).$ Now, following the usual procedure (e.g.,\cite[Section 5]{2023_Gradient_Higher_Integrability_for_Degenerate_Parabolic_Double-Phase_Systems}) and using the Fubini theorem, we arrive at the following
\begin{align*}
    \miint{Q_r(z_0)}H(z, |Du|)^{1+\varepsilon}\, dz\leq c \La^{\varepsilon}_0\miint{Q_{2r}(z_0)}H(z, |Du|)\, dz.
\end{align*}
Finally, we use the definitions of $\lambda_0$ and $\Lambda_0$ given in \eqref{defn: la_0} and \eqref{defn: La_0} and estimate \eqref{eq: estimate of H by lambda_0} to obtain
\begin{align*}
  \miint{Q_r(z_0)}H(z, |Du|)^{1+\varepsilon}\, dz\leq c \left(\frac{2+||a||_{\infty}}{r^{q-2\gamma}}\right)^{1+\frac{\varepsilon q}{2}}   
\end{align*}
which completes the proof.

\bigskip 
\noindent\textbf{Acknowledgment.}
 The authors would like to express their sincere gratitude to the anonymous referee who provided valuable comments and suggestions on the earlier version, which have greatly improved the quality and clarity of the manuscript. Abhrojyoti Sen is supported by postdoctoral research grants from the Alexander von Humboldt Foundation, Germany. The author also thanks Goethe University Frankfurt for its support and research environment.

\bigskip
\noindent\textbf{Conflict of interest.} On behalf of all authors, the corresponding author states that there is no conflict of interest.

\bigskip
\noindent\textbf{Data availability.} This manuscript does not have associated data.

\bibliographystyle{plain}
\bibliography{ref}

\end{document}